\documentclass[11pt,reqno,tbtags,a4paper]{amsart}
\usepackage{amssymb}
\usepackage{cases}
\usepackage{url}
\usepackage[square,numbers]{natbib}
\bibpunct[, ]{[}{]}{,}{n}{,}{,}

\title[A piecewise contractive dynamical system]{A piecewise contractive
  dynamical system and election methods} 

\date{19 September, 2017}

\author{Svante Janson and Anders \"Oberg}
\address{Svante Janson and Anders \"Oberg, Department of Mathematics, Uppsala University, PO Box 480,
SE-751~06 Uppsala, Sweden}
\email{svante.janson@math.uu.se; anders@math.uu.se}


\numberwithin{equation}{section}

\renewcommand\le{\leqslant}
\renewcommand\ge{\geqslant}

\allowdisplaybreaks

\theoremstyle{plain}
\newtheorem{theorem}{Theorem}[section]
\newtheorem{lemma}[theorem]{Lemma}

\newtheorem{corollary}[theorem]{Corollary}

\theoremstyle{definition}
\newtheorem{example}[theorem]{Example}

\newtheorem{problem}[theorem]{Problem}
\newtheorem{remark}[theorem]{Remark}

\theoremstyle{remark}

\newenvironment{romenumerate}[1][-10pt]{
\addtolength{\leftmargini}{#1}\begin{enumerate}
 }{\end{enumerate}}

\newcounter{oldenumi}
{\setcounter{oldenumi}{\value{enumi}}
\begin{romenumerate} \setcounter{enumi}{\value{oldenumi}}}
{\end{romenumerate}}

\newenvironment{alphenumerate}[1][0pt]{
\addtolength{\leftmargini}{#1}\begin{enumerate}
 }{\end{enumerate}}

\newcounter{thmenumerate}
\newenvironment{thmenumerate}
{\setcounter{thmenumerate}{0}%
 \def\item{\par
 \refstepcounter{thmenumerate}\textup{(\roman{thmenumerate})\enspace}}
}
{}

\newcounter{xenumerate}   
\newenvironment{xenumerate}
  {\begin{list}
    {\upshape(\roman{xenumerate})}
    {\setlength{\leftmargin}{0pt}
     \setlength{\rightmargin}{0pt}
     \setlength{\labelwidth}{0pt}
     \setlength{\itemindent}{\labelsep}
     \setlength{\topsep}{0pt}
     \usecounter{xenumerate}} }
  {\end{list}}

\newcommand\pfitemx[1]{\par#1:}
\newcommand\pfitemref[1]{\pfitemx{\ref{#1}}}

\newcounter{pfcase}
\newcommand\newpfcase{\setcounter{pfcase}0}
\newcommand\pfcase[1]{\smallskip\noindent
\refstepcounter{pfcase}%
\emph{Case \arabic{pfcase}: #1}.}
\newcommand\refpfcase{Case \arabic{pfcase}}

\newcommand{\refT}[1]{Theorem~\ref{#1}}
\newcommand{\refC}[1]{Corollary~\ref{#1}}
\newcommand{\refL}[1]{Lemma~\ref{#1}}
\newcommand{\refR}[1]{Remark~\ref{#1}}
\newcommand{\refS}[1]{Section~\ref{#1}}
\newcommand{\refSS}[1]{Section~\ref{#1}}

\newcommand{\refE}[1]{Example~\ref{#1}}

\begingroup
  \divide\count255 by 60
  \multiply\count255 by -60
  \advance\count255 by \time
  \ifnum \count255 < 10 \xdef\klockan{\the\count1.0\the\count255}
  \else\xdef\klockan{\the\count1.\the\count255}\fi
\endgroup

\newcommand\nopf{\qed}   

\newcommand{\sumjo}{\sum_{j=0}^\infty}

\newcommand{\sumk}{\sum_{k=1}^\infty}

\newcommand{\sumiN}{\sum_{i=1}^N}

\newcommand\set[1]{\ensuremath{\{#1\}}}
\newcommand\bigset[1]{\ensuremath{\bigl\{#1\bigr\}}}
\newcommand\Bigset[1]{\ensuremath{\Bigl\{#1\Bigr\}}}

\newcommand\xpar[1]{(#1)}
\newcommand\bigpar[1]{\bigl(#1\bigr)}
\newcommand\Bigpar[1]{\Bigl(#1\Bigr)}

\newcommand\lrpar[1]{\left(#1\right)}

\newcommand\abs[1]{|#1|}
\newcommand\bigabs[1]{\bigl|#1\bigr|}

\newcommand\biggabs[1]{\biggl|#1\biggr|}
\newcommand\lrabs[1]{\left|#1\right|}
\def\rompar(#1){\textup(#1\textup)}    

\newcommand\xpqfrac[2]{(#1)/(#2)}
\newcommand\parfrac[2]{\lrpar{\frac{#1}{#2}}}

\newcommand\Bigparfrac[2]{\Bigpar{\frac{#1}{#2}}}

\def\xexp(#1){e^{#1}}
\newcommand\ceil[1]{\lceil#1\rceil}
\newcommand\floor[1]{\lfloor#1\rfloor}

\newcommand\frax[1]{\{#1\}}
\newcommand\fraxx[1]{\{#1\}_+}

\newcommand\ntoo{\ensuremath{{n\to\infty}}}
\newcommand\Ntoo{\ensuremath{{N\to\infty}}}

\newcommand\downto{\searrow}
\newcommand\upto{\nearrow}
\newcommand\punkt{.\spacefactor=1000}    
    
\newcommand\ie{i.e\punkt}
\newcommand\eg{e.g\punkt}
\newcommand\viz{viz\punkt}
\newcommand\cf{cf\punkt}
\newcommand{\as}{a.s\punkt}

\newcommand\bbR{\mathbb R}

\newcommand\bbQ{\mathbb Q}
\newcommand\bbZ{\mathbb Z}

\newcounter{CC}
\newcounter{cc}
\newcommand{\cc}{\stepcounter{cc}\ccx} 
\newcommand{\ccx}{c_{\arabic{cc}}}     
\newcommand{\ccdef}[1]{\xdef#1{\ccx}}     

\newcommand\E{\operatorname{\mathbb E{}}}

\newcommand\ga{\alpha}
\newcommand\gb{\beta}
\newcommand\gd{\delta}
\newcommand\gD{\Delta}
\newcommand\gf{\varphi}

\newcommand\gG{\Gamma}

\newcommand\gl{\lambda}
\newcommand\gL{\Lambda}
\newcommand\go{\omega}

\newcommand\gs{\sigma}

\newcommand\eps{\varepsilon}

\newcommand\oi{[0,1]}
\newcommand\oo{[0,1)}

\newcommand\cE{\mathcal E}

\newcommand\cH{\mathcal H}

\newcommand\cM{\mathcal M}

\newcommand\cP{\mathcal P}

\newcommand\bU{\overline U}

\newcommand\bbZgto{\mathbb Z_{>0}}

\newcommand\qw{^{-1}}
\newcommand\qww{^{-2}}

\newcommand\intoi{\int_0^1}

\newcommand\oio{[0,1)}
\newcommand\ooi{(0,1]}
\newcommand\ooio{(0,1)}
\newcommand\ooo{[0,\infty)}

\newcommand\setoi{\set{0,1}}

\newcommand\dd{\,\mathrm{d}}

\newcommand\rhs{right-hand side}

\newcommand\rhox{\bar\rho}
\newcommand\phirho{\phi_\rho}

\newcommand\bI{\bar I}

\newcommand\gol{$\go$-limit}
\newcommand\phragmen{Phragm{\'e}n}
\newcommand\Drho{D_\rho}
\newcommand\glo{\gL_0}
\newcommand\gli{\gL_1}
\newcommand\bglo{\overline{\glo}}
\newcommand\case[1]{\smallskip\noindent\emph{#1}}
\newcommand\AAA{A^-}
\newcommand\leb[1]{\lrabs{#1}}
\newcommand\vex{\mathbf}

\newcommand\vn{\vex n}
\newcommand\vp{\vex p}
\newcommand\vx{\vex x}
\newcommand\vy{\vex y}

\newcommand\gab{\zeta}
\newcommand\fx{f}

\newcommand\fSo{\fS^\circ}
\newcommand\ddi{\partial_i}
\newcommand\ddq{\partial}
\newcommand\ddx{\partial^*}
\newcommand\ddxi{\ddx_i}
\newcommand\nx{n_*}
\newcommand\px{p_*}
\theoremstyle{definition}
\newtheorem*{historical}{Historical note}
\theoremstyle{plain}
\newtheorem{thxmetod}{}
\newenvironment{metod}[1]%
{\renewcommand\thethxmetod{\textrm{\sc #1}}\begin{thxmetod}}{\end{thxmetod}}
\newcommand\vfree{\mathbf{\free}}
\newcommand\free{x}
\newcommand\ix{i^*}
\newcommand\vett{\mathbf1}

\newcommand\ax{a}
\newcommand\byx{b^*}
\newcommand\bx{b}
\newcommand\bz{b_0} 
\newcommand\ooz{_0^\infty}

\newcommand\vqq{v^*}
\newcommand\maxx{\cM}

\newcommand{\sumjN}{\sum_{j=1}^N}

\newcommand\mean{\chi}
\newcommand\WW{W^0}
\newcommand\subsex[1]{\medskip\noindent\textbf{#1.}}
\newcommand\subsexx[1]{\medskip\noindent\textit{#1.}}
\newcommand\pis{\pi}
\newcommand\PP{\Pi}
\newcommand\qn{\bar n}
\newcommand\fS{\mathfrak S}
\newcommand\psix{\bar\psi}

\newcommand\az{a_*}
\newcommand\cEao{\cE_{\le \az}}
\newcommand\iao{[0,\az]}

\begin{document}

\begin{abstract} 
We prove some basic results for a dynamical system given by a piecewise
linear and contractive map on the unit interval
that takes two possible values at a point of discontinuity. We prove that
there exists a universal limit cycle in the non-exceptional cases, 
and that 
the exceptional parameter set is very tiny in terms of gauge functions. The
exceptional two-dimensional parameter is shown to have Hausdorff-dimension
one. We also study the invariant sets and the limit sets; these are sometimes
different and there are several cases to consider. 
In addition, we give a
thorough investigation of the dynamics; studying the cases of rational and
irrational rotation numbers separately, and we show the existence of a
unique invariant measure.
We apply some of our results to a
combinatorial problem involving an election method suggested by Phragm\'en
and show that the proportion of elected seats for each party
converges to a limit, which is a rational
number except for a very small exceptional set of parameters.
This is in contrast to a related election method suggested
by Thiele, which we study at the end of this paper,  for which the limit
can be irrational also in typical cases and hence there is no typical
ultimate periodicity as in the case of Phragm\'en's method.
\end{abstract}

\maketitle
\tableofcontents

\section{Introduction}\label{S:intro}\noindent
The purpose of this paper is to study the dynamical system
$f_\pm:\oi\to\oi$ given by the multi-valued function
$x\mapsto\set{f_-(x),f_+(x)}$, where 
\begin{equation}\label{f}
  f_-(x)=\frax{ax+b},
\end{equation}
where 
$a$ and $b$ are given constants with $0<a<1$ and $0\le b<1$, 
$\frax\cdot$ denotes the usual fractional part taking values in $\oio$,
and where $f_+(x)$ 
takes the
value $1$ instead of $0$ for $x$ such that $ax+b$ is an integer, but
otherwise equals $f_-(x)$. We  write $f_+(x)=\{ax+b \}_+$.

The dynamical system given by $f_-:[0,1)\to [0,1)$ has been 
studied from time to time and looks deceptively simple; 
it is locally contractive, but it has (typically) a discontinuity which 
makes the behaviour non-trivial.
It has been studied in a variety of contexts, see, e.g., 
\cite{veer}, \cite{gamb}, \cite{bugeaudCR}, \cite{conze}, \cite{brem}, and
\cite{cout}; furthermore, it is a special case of more general locally
contractive dynamical systems in one or several dimensions studied in
\cite{bruin} and \cite{cat}. The recent works by Nogueira and
Pires \cite{nog}, Nogueira, Pires and Rosales \cite{nog2}, and, especially, 
that of Laurent and Nogueira \cite{laurent-nogueira}, are close to 
our investigation.

We study the dynamical system given by the multi-valued function $f_\pm$
instead of just $f_-$,
both in order to obtain complete (and symmetric)
results concerning the invariant set and the limit set, and 
because we need $f_\pm$
for our application to an election method in \refS{Sphr}. 
The study of the dynamics given by $f_\pm$ becomes somewhat
more complicated than for $f_-$, 
for example when studying the possible orbits,
but
we are rewarded by clear and useful results;
see for example the results in Sections \ref{Srational} and \ref{Sirrational}.

Earlier studies of $f_-$, show  that 
(ignoring a few complications that disappear when considering $f_\pm$)
the limit set may be either a periodic orbit or a
Cantor set, 
and that these cases correspond to rational and irrational rotation numbers.
These results are easily extended to $f_\pm$; much of the extension is
straight-forward, but we also add some details and special features for
$f_\pm$ that make the picture more complete.

In Sections \ref{Sdef} and \ref{Speriodic} we make a preliminary investigation 
of the invariant set $\gL_\pm:=\bigcap_{n=0}^\infty f_\pm^n([0,1])$ and the
limit set $\omega_{f_\pm}(x)$ of  
$f_\pm$ for $x\in [0,1]$.
(See \refS{Sdef} for the definition of the limit set in this context.)
We also show that if there exists a periodic orbit, then it is a universal
limit cycle in the sense that every orbit converges to it.
In particular, there is at most one periodic orbit.
We further
give examples when $\omega_{f_\pm}(x)\subsetneq \gL_\pm$ for all $x\in [0,1]$,
and show that even if $f_\pm$  has a universal limit cycle, the
invariant set may be different from it, in analogy with the higher dimensional
case, see \cite{cat}.

In \refS{Sorbits+} we study all possibilities for
orbits of $f_\pm$, with different cases depending on whether a periodic
orbit exists or not, and also on whether the periodic orbit (if it exists)
contains 
the  point of discontinuity
(the point of two values) of $f_\pm$ or not.

Next, building on 
the work by Bugeaud \cite{bugeaudCR}, Bugeaud and Conze \cite{conze}, 
and  \citet{cout},
we study in Sections \ref{Srotation} and \ref{Sloc}
the rotation number of $f_\pm$, 
with special attention to whether the rotation number is rational or irrational.
Furthermore, we show in \refS{Srotation} that every orbit has a well-defined
average, and that this is related to the rotation number. 
In \refS{Sloc} we identify the set of
parameters $(a,b)$
that gives rise to a certain rotation number.

As shown by 
\citet{bugeaudCR} and \citet{conze}, 
the rotation number of this dynamical system is
typically rational;
the exceptional set of parameters $(a,b)$ such that the rotation number is
irrational has Lebesgue measure 0, and \citet{laurent-nogueira} showed,
furthermore, that
the set of exceptional $b$ for a fixed $a$  has Hausdorff dimension 0.
We improve this result on Hausdorff dimension somewhat in \refS{Shaus},
in that we specify a gauge function,  
$h(t)=1/|\log t|^2$, for which 
the Hausdorff measure of the exceptional parameter set 
is finite.
We also give  a  lower
bound showing that this exceptional set is not arbitrarily tiny,
by showing that the Hausdorff measure is positive for the gauge function 
$h(t)=1/|\log t|$. We further prove that the exceptional set of 
parameter pairs  $(a,b)$ (a subset of $\oio^2$)
has Hausdorff dimension  1. 
We prove in \refS{Shaus} 
also that the Hausdorff dimension of the invariant set $\gL_\pm$ is zero
and that its Hausdorff measure is finite for the gauge function
$h(t)=1/|\log t|$. We leave it as an open  
question whether this gauge function is best possible in some sense.

In \refS{Srational} we prove that the dynamical system given by
$f_\pm$ has a rational rotation number if and only if it has a universal
limit cycle. In \refS{Sirrational}, we study the case of an irrational
rotation number and classify the limit
sets for $f_-$, $f_+$ and $f_\pm$; we prove in particular that
the limit set  $\omega_{f_\pm}(x)$ (then a Cantor set) is equal to the invariant
set $\gL_\pm$ for all $x\in [0,1]$.  

In \refS{Sinvariant}, we show that the dynamical system $F_\pm$ 
has a unique invariant measure with support in the invariant set.
Furthermore, the empirical measure of any orbit converges to this invariant
measure. 

The dynamical system we consider, or rather the one given by $f_-$, has been
studied in several
applications, of which we here only mention a couple of interesting ones: 
the work by Feely and Chua \cite{feely} in signal theory, 
which inspired \cite{bugeaudCR} and \cite{conze}, and 
the paper by Coutinho {\em et al}.\ \cite{cout2} 
studying genetic regulatory networks.

\subsex{Two election methods}
We also have an application in mind, and this was our original motivation
for the present work. We wanted to understand a curious  
behaviour 
recently found by  Mora and Oliver \cite{MoraO} 
of an election method that was suggested in 1894
by the Swedish mathematician Edvard Phragm\'en 
\cite{Phragmen1894}.

As a background, 
consider election methods where a given number $n\ge1$ of persons are to be
elected 
from some list of candidates without any formal parties, and
each voter votes for set of candidates (without ranking), 
where the set may be
chosen arbitrarily (except that possibly its size is restricted).
One such method is simple \emph{plurality}, where the $n$ persons with the
largest 
number of votes are elected. (In this case, usually each voter is
restricted to vote for at most $n$ candidates;
this system is also called \emph{block vote}. 
The version where a voter may
vote for any number of candidates is called \emph{appoval voting}.)
This method has been widely used, and it is still widely used in \eg{}
associations and societies without (formal or informal) parties. However,
for general elections with political parties, it will typically lead to the
largest party getting all seats; hence this method has for such purposes in
most places been replaced by other methods that tend to give
representation also to smaller parties, for example proportional methods 
based on parties with separate lists such as \emph{D'Hondt's method} 
\cite{DHondt1878,DHondt1882} 
or \emph{Sainte-Lagu\"e's method} \cite{StL}. 
(Many different
election methods are and have been used, or proposed; see \eg{}
\cite{BY}, \cite{electoralSystems} and \cite{Pukelsheim} for 
discussions of several important ones,
including also practical and political aspects.)

Another way to achieve some kind of proportionality is to keep the system
above, where each voter votes with a ballot containing an arbitrary set of
candidates, but elect the $n$ persons sequentially and reduce the voting
power of the ballots where some candidates already have been elected.
Two different such systems were proposed in 1894 and 1895 by the Swedish
mathematican 
Edvard Phragm{\'e}n (1863--1937) \cite{Phragmen1894,Phragmen1895}
and the Danish astronomer and
mathematician Thorvald Nicolai Thiele (1838--1910) \cite{Thiele},
respectively; see also
\cite{Phragmen1896,Phragmen1899}
and \cite{SJV9}.
Both methods can be seen as generalizations of D'Hondt's method to a
situation without formal parties \cite{Phragmen1895,SJV9}.
(To be precise, in a situation with organized parties, if every voter votes
for a party list, then both methods yield the same result as D'Hondt's
method.
This is easy to see
from the descriptions in Sections \ref{SSphr-algo} and \ref{SSthiele-algo}.)

We describe \phragmen's and Thiele's methods in Sections
\ref{Sphr} and \ref{Sthiele}; see \cite{SJV9} for further discussion.

\subsexx{A party version}
\citet{MoraO} recently considered an extension of \phragmen's me\-th\-od, where
the individual candidates are replaced by (disjoint) groups of candidates;
these groups are called \emph{candidatures} in \cite{MoraO}, but we shall call
them \emph{parties}. Mathematically, the difference is that a party may get
several members elected; the seats are allocated to the parties one
by one as in the original method, 
but we allow repetitions so a party may be selected several
times. We assume in this paper (unlike \cite{MoraO}) 
that the parties are sufficiently large (with potentially infinite lists of
candidates) so 
that they do not run out of persons to fill their seats.
We  also consider the same extension of Thiele's
method.

\begin{remark}\label{Rparty}
  We have presented the party version 
as an extension of the original method, but it can also be considered as
a special case. Consider the original method with individual candidates
and assume that there are parties consisting of disjoint sets of candidates
that are regarded 
as equivalent by all voters (and by us), so that each voter votes for either
all candidates from a party or for none of them, for each party. In other
words, each voter votes for the union of some set of parties. It is then
easy to see,   for both \phragmen's and Thiele's methods,
that the result is the same for the party version and for the original version
(with the party representatives chosen \eg{} by lot, since all from the same
party will tie each time).
\end{remark}

We consider an election using the party version of either \phragmen's or
Thiele's method,
with some  set of parties and some set of votes (where each vote thus is for one
or several parties).
We let $n\ge1$ seats be distributed in the election, 
and let $n_i$ be the number of
seats given to a party $i$ and $p_{in}:=n_i/n$ the corresponding proportion of
seats. 
Our main interest is in the asymptotics of these proportions as the number
$n$ of elected seats tends to infinity, for a fixed set of votes. 
(This makes sense for the party version, but not for the original version.)
In the
case when each voter 
votes for exactly one party, both methods reduce to D'Hondt's method, as said
above, 
and it is 
well-known and 
easy to see
that the proportion $p_{in}$
of elected seats for a party then converges to the proportion of
votes for that party. 
(For more precise results, see \cite{SJ262}.)

\citet[Section 7.7]{MoraO} studied in particular
the party version of \phragmen's method
in the case with only two
parties, $A$ and $B$, and found numerically 
that the proportions $n_A/n$ and $n_B/n=1-n_A/n$ of elected seats for each party
do converge;
however, the limit has
an unexpected singular  \lq Devil's staircase' structure as a function of the
proportions of votes 
for different ballots: it seemed that the limit is always a rational number and
that each rational number in $(0,1)$ is the limit for some range of the vote
proportions.  
We show that this is indeed the case, with the modification that irrational
limits exist but only for a null set of the parameters,
by interpreting the party version of Phragm\'en's method as a dynamical
system, which in the case of two parties can be 
transformed to a dynamical system of the type considered in the present
paper.
This leads to the following theorem, which is one of our main results.
The proof is given in \refS{Sphr}.
Recall that in the present context each vote is either for party $A$, party
$B$ or the set $\set{A,B}$, which we denote by $AB$. 

\begin{theorem}\label{Tphr}
Consider the party version of \phragmen's election method, 
with two parties A and B, 
and let the proportions of votes on $A$, $B$ and $AB$
be
$\ga,\gb$ and $\gab=1-\ga-\gb$, 
respectively, and assume that $\ga+\gb>0$. 
Let $n_A$ and $n_B$ be the numbers of seats given to the two
parties when $n$ seats have been distributed; then
the fractions $n_A/n$ and $n_B/n$ of seats given to the two parties converge to some
limits $p_A$ and $p_B=1-p_A$, respectively, as \ntoo.
Furthermore, the following holds.
\begin{romenumerate}

\item \label{TphrO}
$n_A=p_An+O(1)$ and $n_B=p_Bn+O(1)$.

\item \label{TphrpB}
If $\ga\ge\gb>0$, 
then
\begin{equation}\label{pB}
p_B=\frac{1}{2+\bz+\rho},    
\end{equation}
where $\rho$ is the rotation number of the dynamical system
\begin{equation}\label{faxbx}
  f_\pm(x)=\set{\frax{\ax x+\bx},\fraxx{\ax x+\bx}}
\end{equation}
and we define
\begin{align}\label{tax}
\ax&:
=\frac{\ga\gb}{(\ga+\gab)(\gb+\gab)}
=\frac{\ga\gb}{(1-\ga)(1-\gb)}
\in (0,1],
\\ \label{tbyx}
  \byx &:
=
\frac{\ga-\gb}\gb+\frac{\ga(1-\ga-\gb)}{(1-\ga)(1-\gb)},
\\\label{tbx}
\bx&:=\frax{\byx},
\\\label{tbz}
\bz&:=\floor{\byx}.
\end{align}
We have $a<1\iff\zeta>0$.

\item \label{TphrQ}
If the rotation number $\rho$ is rational, and furthermore $\gab>0$,
then the sequence of awarded seats is eventually periodic.
\end{romenumerate}
\end{theorem}

Furthermore,  \eqref{pB} can be combined with
Theorem \ref{TLb} or Theorems \ref{THb}--\ref{THab},
which all imply that the rotation number, and thus
$p_B$, is rational for almost all values of the parameters $\ga,\gb$, and
that each rational number in $(0,1)$ is attained for some set of $(\ga,\gb)$
with a non-empty interior, verifying the observed Devil's staircase behaviour.
The reader can compare \cite[Figure 1]{conze} and \cite[Figura 2]{MoraO},
which show this phenomenon from two different points of view, connected by
our \refT{Tphr}.
\begin{remark}
In particular,  as shown by \citet{laurent-nogueira}, see \refT{Talgebraic}
below,
the rotation number is rational whenever $a$ and $b$ are rational (or even
algebraic) numbers;
hence \refT{Tphr} shows that $p_B$ is rational whenever $\ga$ and $\gb$ are
rational (or algebraic),
which
explains why only rational limits were observed in \cite{MoraO}.
See further \refT{Tphr-alg}.
\end{remark}

\begin{problem}
  Consider the party version of \phragmen's method in a case with $N\ge3$
  parties, and given numbers of votes. Will the proportions of seats $n_i/n$
given to the different parties   converge as \ntoo?
What are the limits?
\end{problem}

In \refS{Sthiele}, we consider instead the party version of Thiele's method 
(with an arbitrary number of parties), and obtain very different results.
We show that, under weak hypoteses, the proportions of seats for each party
converge as
\ntoo{} for
Thiele's method too, 
but now each limit is a smooth function of the vote proportions;
moreover, 
the limits can be irrational numbers also in simple cases with integer
numbers of votes. 
 We do not know whether there is a quasi-periodic behaviour in this case.
In any case, we find this difference between the two election methods
interesting.

\begin{remark}
  \phragmen's and Thiele's methods were devised for a situation without
a completely developed party system, and for small constituencies.
Here, in contrast, we study the methods in the opposite situation with
well-organized parties and a very large number of seats.
The results are therefore not directly relevant for the original situation,
and our investigation is mainly for mathematical curiosity;
nevertheless, the results might give insight into some aspects of the
methods.

For small numbers of seats, Thiele's method sometimes yields undesirable 
results, while \phragmen's method seems more robust, as discussed with many
examples in the 1913
report of the Swedish Royal
Commission on the Proportional Election Method \cite{bet1913},
see also \cite{SJV9}.
For very large numbers of seats, our result indicate the opposite, with
a smoother behaviour of Theiele's method.
\end{remark}

\begin{historical}
Thiele's  method was  used in Swedish parliamentary
elections 1909--1920 
for the distribution of seats within parties 
(in combination with a special rule); it was in 1921 replaced by an ordered
version of \phragmen's method. This version of
\phragmen's method is still formally used but
nowadays in combination with a system of personal votes and in reality the
method has a very minor role.
See further \cite[Appendix D]{SJV9}.
\end{historical}

\subsex{Acknowledgements}
First of all, we would like to thank Mark Pollicott for helping us with this project. 
We are also grateful to Arnaldo Nogueira and Jean-Pierre Conze for valuable guidance, and to Anders Johansson
for many valuable discussions. The first author was supported in part by the Knut and Alice Wallenberg 
Foundation.

\section{Notation and some basic properties}\label{Sdef}\noindent
We assume throughout that $a$ and $b$ are given constants with $0<a<1$ and
$0\le b<1$.
(See \refR{Rac} for other parameter values.) 

We let, as usual, $\floor{x}$ and $\frax{x}$ denote the integer and
fractional parts of a real number $x$; thus $\floor{x}\in\bbZ$ and
$\frax{x}:=x-\floor{x}\in\oio$. 
Furthermore, $\ceil{x}:=-\floor{-x}$ is the smallest integer $\ge x$.
We further define $\fraxx{x}$ as the left-continuous version of $\frax{x}$; 
thus,
when $x\in\bbR\setminus\bbZ$, then
$\fraxx{x}=\frax{x}\in(0,1)$, but if $x\in\bbZ$, then $\frax{x}=0$ and
$\fraxx{x}=1$. 
(Equivalently, $\fraxx{x}:=1-\frax{-x}$.)

For a function $f$ defined on (a subset of) $\bbR$, let
$f(x-):=\lim_{y\upto x} f(y)$ and $f(x+):=\lim_{y\downto x} f(y)$, when
the limits exist.

The Lebesgue measure of a set $E\subseteq\bbR$ is denoted $\leb{E}$.

\subsection{The basic functions} 
Let us first dismiss a trivial case.

\begin{example}\label{E<1}
  Suppose that $a+b<1$. Then \eqref{f} is $f_-(x)=ax+b$ for all $x\in\oi$.
  This is a linear  contraction, and trivially $f_-^n(x)\to p_0$ as
\ntoo{} for every $x$, where $p_0:=b/(1-a)\in\oio$ is the (unique)
fixed point  of $f_-$.

If $b>0$, then $f_+=f_-$, and thus 
$f_\pm(x)^n\to p_0$ as \ntoo, for every $x$.
We return to the case $b=0$ in \refE{E0} below.
\end{example}

In the sequel we thus focus on the case $a+b\ge1$.

Let $\tau\in\oi$ be the point of discontinuity of $\frax{ax+b}$ in \oi, if any.
Thus, if $a+b\ge1$ (our main case), then 
$\tau=(1-b)/a$ is the solution of $ax+b=1$; note that in this case $\tau\in\ooi$.
In the exceptional case $b=0$, we have $\tau=0$, 
and in the trivial case $a+b<1$ with $b>0$ (see \refE{E<1}),
$\tau$ does not exist.

As said in the introduction, we allow an ambiguity at the discontinuity
point $\tau$, and we
thus define two versions of \eqref{f}, both for $x\in\oi$:
\begin{align}
  f_-(x)&:=\frax{ax+b}=
ax+b-\floor{ax+b},
\label{f1-}\\ 
\label{f1+}
  f_+(x)&:=\fraxx{ax+b}=
ax+b-(\ceil{ax+b}-1).
\end{align}
Thus, explicitly, in the case $a+b\ge1$, when $\tau>0$,
\begin{align}
  f_-(x)&=
  \begin{cases}
    ax+b, & 0\le x<\tau;\\
    ax+b-1, & \tau\le x\le1;
  \end{cases}
\label{f2-}\\ \label{f2+}
  f_+(x)&=
  \begin{cases}
    ax+b, & 0\le x\le\tau;\\
    ax+b-1, & \tau< x\le1.
  \end{cases}
\end{align}
If $\tau=b=0$, then \eqref{f2-}--\eqref{f2+} are modified by replacing $b$ by 1.
In the trivial case when $\tau$ does not exist, $f_-(x)=f_+(x)=ax+b$ for all
$x\in\oi$.

Note that $f_-(x)=f_+(x)$  except at the discontinuity $x=\tau$,
where $f_-(\tau)=0$ and $f_+(\tau)=1$. 
Note also that $f_-$ is right-continuous on $\oi$ 
and $f_+$ is left-continuous. 
Furthermore, $f_-:\oi\to\oio$ and $f_+:\oi\to\ooi$.

Finally, let $f_\pm(x)$ denote the multi-valued function
$x\mapsto\set{f_-(x),f_+(x)}$. 
Formally, this is a set-valued function, 
but we usually regard it as a function $\oi\to\oi$ that is indeterminate at
$\tau$, where we can choose freely between $f(\tau)=0$ and $f(\tau)=1$;
for $x\in\oi\setminus\set\tau$, 
$f_\pm(x)$ is a unique single value in $\oi$.

Note that $f_\pm$ is injective but not surjective, and that it has a
continuous single-valued inverse $f_\pm\qw:[0,a+b-1]\cup[b,1]\to\oi$ (when
$a+b>1$). 

\begin{remark}
  \label{Rac}
We thus assume  $0<a<1$ and $0\le b<1$.
The assumption $0\le b<1$ is without loss of generality, since only the
fractional part of $b$ matters. However, it is also possible to consider other
values of $a$. The main reason for our assumption $0<a<1$ is that we want the
dynamical system to be locally contractive, which rules out $|a|\ge1$.

The case $-1<a<0$ is locally contractive but decreasing instead of
increasing; this seems to be another interesting case, and we expect results
similar to the ones in the present paper, but this case will not be studied
here.  

Note also that the limiting cases $a=0$ and $a=1$ are trivial: when $a=0$,
$f$ is constant, and when $a=1$, $f_-(x)=\set{x+b}$ is just a translation
(rotation) on the circle group $\bbR/\bbZ$.
\end{remark}

\begin{remark}\label{Rreflection}
  The reflection $\gs(x):=1-x$ maps the dynamical system to another one of
  the same kind. More precisely,  indicating the parameters $a,b$ by
  subscripts, if we reflect the left-continuous $f_{a,b;+}$ we 
obtain the right-continuous
\begin{equation}
  \begin{split}
\gs\circ f_{a,b,+}\circ\gs(x)
&=
  1-f_{a,b;+}(1-x)=1-\fraxx{a-ax+b}
\\&
=\frax{-(a-ax+b)}
=\frax{ax-(a+b)}
=f_{a,\tilde b;-}(x),    
  \end{split}
\raisetag\baselineskip
\end{equation}
where
\begin{equation}\label{bref}
  \tilde b:= \frax{-(a+b)}.
\end{equation}
Similarly, the reflection of $f_{a,b,-}$ is $f_{a,\tilde b;+}$, and consequently
the reflection of $f_{a,b,\pm}$ is $f_{a,\tilde b;\pm}$.

If $a+b>1$ (the most interesting case),
\eqref{bref} yields $\tilde b=2-a-b$.
\end{remark}

\subsection{Orbits and periodic points}\label{SSdeforbit}

For the  single-valued function $f_-$, the \emph{orbit} of a
point $x\in\oi$ is, as usual, the sequence $(f_-^n(x))_{n=0}^\infty$, 
and similarly for $f_+$.
For the multi-valued $f_\pm$, we say that an orbit of $x\in\oi$ is any
sequence
$(x_n)_{0}^\infty$ such that $x_0=x$ and $x_{n+1}\in f_\pm(x_n)$, $n\ge0$.
In other words, an orbit is any possible sequence obtained by repeatedly
applying $f_\pm$, making arbitrary choices each time there is a choice (\ie,
when the orbit visits $\tau$).

A \emph{periodic orbit} 
is an orbit $(x_n)_0^\infty$ with $x_{n+q}=x_n$ for some $q\ge 1$ 
(the \emph{period}) and all $n\ge0$; in this case we also write the orbit as
$\set{x_0,\dots,x_{q-1}}$. If furthermore $x_0,\dots,x_{q-1}$ are distinct,
we say that this is a \emph{minimal periodic orbit}.
Note that, also for a multi-valued function such as $f_\pm$, a non-minimal
periodic orbit 
always can be seen as a combination of several minimal periodic orbits
(identical or not, and possibly with different initial points and inserted into
each other).

A periodic orbit with period 1 is the same as a fixed point.

A \emph{periodic point} is a point $x$ that has a periodic orbit.

We consider a few simple examples with a periodic orbit (for example a fixed
point), but where the multi-valuedness of $f_\pm$ causes complications
because $\tau$ is in the periodic orbit.
The general case is studied in
\refS{Speriodic}. 

\begin{example}
  \label{E1}
Suppose that $a+b=1$. Then $\tau=1$, and 1 is both a fixed point of $ax+b$
and a discontinuity point, since $f_\pm(1)=\setoi$. 
If $0\le x<1$, then $x$ has a unique orbit
$(x_n)\ooz=(f_\pm^n(x))\ooz=(f_-^n(x))\ooz=(f_+^n(x))\ooz$ 
with, by induction, $x_n=1-a^n(1-x)$;
the orbit converges to the fixed point 1, but it never reaches 1 and thus
there is never any choice. 

However, if we start with $x=1$, then there is one periodic orbit $1$ with
period 1, but there are also infinitely many other orbits, starting with 1
repeated an arbitrary number of times followed by a jump to 0; from that
point the orbit follows the unique orbit starting at 0 and thus converges to
1 as said above.

Consequently, in this example, all possible orbits converge to the fixed
point 1. However, note that they do not converge uniformly, since an orbit
starting at 1 may reach 0 at any given later time.
\end{example}

\begin{example}
  \label{E0}
Suppose that $b=0$. This is a special case of \refE{E<1}, and $f_-(x)=ax$
which is a contraction with fixed point 0, so all orbits of $f_-$ converge to 0.

However, in this case (unlike the case $a+b<1$ with $b>0$), \refE{E<1} does
not give the full story for $f_\pm$, since $f_+(0)=1$. Hence, the fixed
point 0 is also 
the discontinuity point $\tau$, and 0 has infinitely many orbits, the
periodic orbit 
$0$ and orbits starting 0 repeated an arbitrary number of times followed by
1 and then converging back to 0, without ever reaching it.

The situation is as in \refE{E1}, with 0 and 1 interchanged; in fact, the two
examples are the mirror images of each other by the reflection discussed in
\refR{Rreflection}. 
\end{example}

\begin{example}\label{E2/3}
Consider $a=1/2$ and $b=2/3$, \ie, $f_-(x)=\frax{\frac12x+\frac{2}3}$.
Then $\tau=2/3$. Furthermore, $f_\pm(0)=2/3$,
and thus $\set{0,\frac{2}3}$ is a periodic orbit with period 2. 
But 0 and $2/3$ also have an infinite number of orbits that include
$f_+(2/3)=1$, for example $\frac23,1,\frac{1}6, \dots$. Each such orbit
continues from 1 along the unique orbit of 1, which is
$1,\frac{1}6,\frac{3}4,\frac{1}{24},\frac{11}{16},\dots$,
where $x_{2n}=(2+2^{-2n})/3$
and $x_{2n+1}=2^{-2n-1}/3$; hence each such orbit converges to the periodic
orbit $\set{0,\frac{2}3}$.
\end{example}

\subsection{The invariant set}

If $K\subseteq\oi$, then
\begin{equation}
  f_\pm(K)=f_+(K\cap[0,\tau])\cup f_-(K\cap[\tau,1]).
\end{equation}
Since $f_+$ is continuous on $[0,\tau]$ and $f_-$ on $[\tau,1]$, it follows
that if $K\subseteq\oi$ is compact, then $f_\pm(K)$ is compact.

Consequently (by induction), $f_\pm^n(\oi)$, $n\ge0$, is a decreasing
sequence of non-empty compact subsets of $\oi$, and thus
\begin{equation}\label{gLpm}
  \gL_\pm:=\bigcap_{n=0}^\infty f_\pm^n(\oi)
\end{equation}
is a non-empty compact set.

Note that $f_\pm(\gL_\pm)=\gL_\pm$ and (since $f_\pm\qw$ is single-valued)
$f_\pm\qw(\gL_\pm)=\gL_\pm$.
In particular, since $f_\pm(\tau)=\setoi$,
\begin{equation}\label{01tau}
  0\in\gL_\pm \iff \tau\in\gL_\pm \iff 1\in\gL_\pm.
\end{equation}
Moreover, if $0,\tau,1\notin\gL_\pm$, then $f_\pm$ is single-valued on
$\gL_\pm$, and thus $f_\pm:\gL_\pm\to\gL_\pm$ then is a homeomorphism. 
(We shall see in Sections \ref{Srational} and \ref{Sirrational} that this
happens only when $\gL_\pm$ is finite, 
\cf{} the general \cite[Theorem 3.1]{cat}.)

We can also define the corresponding sets for $f_-$ and $f_+$:
\begin{equation}
  \gL_-:=\bigcap_{n=0}^\infty f_-^n(\oi),
\qquad
  \gL_+:=\bigcap_{n=0}^\infty f_+^n(\oi).
\end{equation}
However, these may be empty, as seen by the following example (and its
mirror image \refE{E0});
furthermore, $\gL_-$ and $\gL_+$ are not always closed sets, see \refT{Tir}.
Hence $f_\pm$ and \eqref{gLpm} yield a more satisfactory definition. 
We describe the sets $\gL_\pm,\gL_-,\gL_+$ completely in Theorems \ref{TRb} and
\ref{Tir}.

\begin{example}\label{EbigA}
  Consider again \refE{E1} with $a+b=1$. Clearly the fixed point
  $1\in\gL_\pm$, and 
thus every orbit of 1 is contained in $\gL_{\pm}$; furthermore, by applying
$f_\pm\qw$ repeatedly, it is easily seen that no further points belong to
$\gL_\pm$. Thus
$\gL_\pm=\set{1-a^n:n\ge0}\cup\set1$.
It is also easily seen that $\gL_-=\emptyset$ and $\gL_+=\set1$.
\end{example}

\begin{remark}\label{RbigA}
  The invariant set is sometimes called the \emph{attractor}, see \cite{cat}
  (where our definition corresponds not to 
Definition 2.2 but to the version 
given  immediately afterwards;
  these are not always equivalent). 
However, in the present context, this name
  seems less appropriate. 
For example, in \refE{EbigA},  every 
orbit is attracted to 1, see \refE{E1}.
\end{remark}

\subsection{The limit set}
As in the higher-dimensional case (see \cite{cat}) the invariant set  $\gL_\pm$ 
for our multivalued $f_\pm$ can be quite large,
and too large for some purposes, see \refE{EbigA} and \refR{RbigA}.
It is convenient to introduce the notion of a 
{\em limit set} for $f_\pm$. 
For single-valued functions, we define the 
$\omega$-limit set as in, e.g.,
\cite{nog} and \cite{cat}: 
for a single-valued function $f$, we say that a point $p$ is an 
$\omega$-limit point of $x$ if there is a strictly increasing
sequence of positive integers $\{n_\ell\}$ such that
$\lim_{\ell \to \infty} f^{n_\ell}(x)=p$. The collection of all such limit points is 
the $\omega$-limit set of $x$, denoted by $\omega_f (x)$. 
Equivalently,
\begin{equation}\label{omegaf}
\omega_f (x)=\bigcap_{m\geq 0} \overline{\bigcup_{k\geq m}\{f^k(x) \}}.  
\end{equation}
We adjust this definition for the multi-valued function
$f_\pm$ with the convention that we follow a
specific orbit. 
More precisely, for   $f_\pm$, we say that $p$ is an 
$\omega$-limit point of $x$ if 
there exists an orbit $(x_n)_0^\infty$ of $x$ and
a subsequence $\{n_\ell\}_{\ell=1}^\infty$ of
positive integers such
that $x_{n_\ell}\to p$ as $\ell \to \infty$. 

\begin{remark}\label{Rdefoi}
  The function $f_-$ maps into $\oio$, so it may be regarded as a dynamical
  system either $f_i:\oio\to\oio$ or $f_i:\oi\to\oi$. 
(The difference is of course trivial, and usually does not matter.)
For definiteness, we interpret \eqref{omegaf}
in $\oi$, so $\go_{f_-}(x)$ is a closed subset of $\oi$, defined for all
$x\in\oi$. The same applies to $f_+$.
\end{remark}

For a specific periodic orbit $C=\set{y_0,\dots,y_{k-1}}$, we say that 
an orbit $(x_n)_{n=0}^\infty$  converges to $C$ 
if there exists $j$ such that
$x_n-y_{j+n \bmod k}\to0$ as \ntoo. 
We further say that $C$ is 
a \emph{limit cycle} of $x$
if every orbit starting at $x$ converges to $C$;
in this case we also say that \emph{$x$ is attracted to} $C$.
If $C$ is a limit cycle of $x$, then $\omega_{f_\pm}(x)=C$.
Conversely, 
using \refL{LA01} below, 
it is easy to see that if $C$ is a periodic orbit of $f_\pm$,
and $\omega_{f_\pm}(x)=C$, then $C$ is a limit cycle of $x$.

We say that $C$ is a \emph{universal limit cycle} if it is a limit cycle for 
every $x\in\oi$, 
or, equivalently, that $\omega_{f_\pm}(x)=C$ for every $x$.
In other words, every orbit with any initial point is
attracted to $C$. 

A related notion is that
$f_\pm$ is {\em asymptotically periodic}
if $\omega_{f_\pm}(x)$ is a periodic orbit of $f_\pm$ for every $x\in [0,1]$.
As shown in \refS{Speriodic} below, $f_\pm$ has at most one periodic orbit,
and thus $f_\pm$ is asymptotically periodic if and only if $f_\pm$ has a
universal limit cycle.
(Cf.\ \cite{bruin} and  \cite{nog}, where this notion is studied in
situations where several periodic orbits may occur.)

It is easy to see that
$\omega_{f_\pm}(x)\subseteq \gL_\pm$. 
We note that in \refE{E1} we have $\go_{f_\pm}(x)=\set1$ for every $x$, and
thus, see Example \ref{EbigA}, 
$\omega_{f_\pm}(x)\subsetneq \gL_\pm$ for every $x$. This is 
also the case in the following example, which illustrates 
one possible 
situation when there is a periodic orbit, 
see \refS{Sorbits+}.
See also Remarks \ref{RRb} and \ref{Rir} where 
the relation between the limit sets and the invariant sets
is studied further.

\begin{example}
Consider again \refE{E2/3} with $a=1/2$ and $b=2/3$.
Then the \gol{} set $\go_{f_\pm}(x)=\set{0,\frac{2}3}$ for every $x\in\oi$,
and thus the periodic orbit $\set{0,\frac{2}3}$ is a 
universal limit cycle with period 2. 
But  
$\tau=2/3$ is mapped to $0$ or $1$ and this makes it impossible to
get a uniform bound on the rate of convergence to the limit cycle.
This phenomenon will occur for any $f_\pm$ as soon as $\tau \in \gL_\pm$
and is in contrast to the uniform rates for $f_-$ and $f_+$ 
(see \cite[Theorem 2.2(2)]{brem}).
\end{example}

\begin{remark}
Another related notion, 
is the {\em non-wandering set} of $f_\pm$, 
as defined in e.g.\ \cite{cat}. 
In our case, it can be shown, \eg{} using Theorems \ref{TRb} and \ref{Tir},
 that the non-wandering set
is equal to the \gol{} set $\omega_{f_\pm}(x)$
for all $x\in [0,1]$.
We shall therefore not consider the non-wandering set further.
\end{remark}

\subsection{The lifts}

We define lifts 
$F_-,F_+: \mathbb R \to \mathbb R$ of $f_-$ and $f_+$
by
\begin{align}
F_-(x) &:= a\frax{x}+b  + \floor{x} \phantom:= ax + b + (1-a)\floor{x}, 
\label{F-}
\\
F_+(x) &:= F_-(x-)
=ax+b  - (1-a)\floor{1-x}. 
\label{F+}
\end{align}

Note that $F_-(x)=F_+(x)$ unless $x$ is an integer.

We collect some standard properties that follow immediately from the definition.

\begin{lemma}[{Cf.~\cite[p.\ 15]{cout}}] \label{LF}
Let $F_-,F_+: \mathbb R \to \mathbb R$ be the lifts defined in
\eqref{F-}--\eqref{F+}. 
Then
\begin{romenumerate}
\item  
$F_-(x + 1) = F_-(x) + 1$,  $F_+(x + 1) = F_+(x) + 1$.
\item\label{LFpi}  
$\pis_-\circ F_- = f_- \circ \pis_-$,
where $\pis_-: \mathbb R \to [0,1)$ is given by $\pis_-(x) = \frax{x}$;  
$\pis_+\circ F_+ = f_+ \circ \pis_+$,
where $\pis_+: \mathbb R \to \ooi$ is given by $\pis_+(x) = \fraxx{x}$. 
\item\label{LF<}  
$F_-$ and $F_+$ are strictly increasing.
\item $F_-$ and $F_+$ are continuous except at the integers;
$F_-$ is right-continuous and $F_+$ is left-continuous.
\end{romenumerate}
\end{lemma}
\begin{proof}
  Obvious.
\end{proof}

\subsection{The rotation number} 
\label{SSrotation}
It is well-known 
that the dynamical system $f_-$ has a well-defined 
\emph{rotation number},
see \eg{} 
\cite{bugeaudCR},
\cite{conze}, \cite{cout}, 
This is easily extended to $f_\pm$ 
in the following sense.
We give a proof in \refS{Srotation}.
\begin{lemma}\label{Lrho}  
There exists a number $\rho=\rho(f_\pm)\in\oio$,
called the rotation number of $f_\pm$,
such that, for any $x\in\bbR$, as \ntoo,
\begin{equation}\label{lrho}
  F_-^n(x)/n\to \rho, 
\qquad
  F_+^n(x)/n\to \rho.
\end{equation}
In fact,
\begin{equation}\label{lrhoO}
  F_-^n(x)= x+\rho n + O(1), 
\qquad
  F_+^n(x)= x+\rho n + O(1),
\end{equation}
uniformly in $x\in\bbR$ and $n\ge0$.
We have  
\begin{equation}\label{rho-ab}
  a+b-1\le\rho\le b.
\end{equation}
Furthermore, $\rho=0\iff a+b\le1$. 
\end{lemma}
We also use the notation $\rho(a,b)$.

The rotation number will be important in the sequel.
In particular, we shall see (in \refS{Srational}) 
that 
there exists a periodic orbit
if and only if the rotation number is rational;
moreover, in this case the periodic orbit is unique and is a universal limit
cycle, i.e., it attracts every orbit.

\subsection{Symbolic dynamics}\label{SSsymbolic} 
In the case $a+b\ge1$ (and thus $\tau>0$), 
we code an orbit $(x_i)_0^\infty$ for $f_\pm$
by a symbolic sequence $(\eps_i)_0^\infty$, where $\eps_i\in\setoi$ is defined by
\begin{equation}\label{eps}
  \eps_i:=
  \begin{cases}
0 & x_i\in[0,\tau) \text{ or } (x_i=\tau \text{ and } x_{i+1}=1),
\\   
1 & x_i\in(\tau,1] \text{ or } (x_i=\tau \text{ and } x_{i+1}=0).
  \end{cases}
\end{equation}
See e.g.\ \cite{DingHemmer}, \cite{feely}, \cite{cout}  
for equivalent versions (in the single-valued case); 
see also
\cite{gamb} for deep study of symbolic dynamics in a more general
situation. 

By \eqref{f2-}--\eqref{f2+}, we have
\begin{equation}\label{eps2}
  \eps_i=ax_i+b-x_{i+1}.
\end{equation}

For completeness,  
we define $\eps_i$ by \eqref{eps2} also when $a+b<1$,
although this case is not very interesting: if $a+b<1$ and $b>0$, then
$\eps_i=0$ for all $i$, and if $b=0$, then $\eps_i=0$ except possibly for
one $i$, where we have $\eps_i=-1$.

The proportion of 1's in the symbolic sequence converges for any orbit, 
and the limit equals the rotation number.
This was shown for $f_-$ by \citet{cout};
we extend this to $f_\pm$
in the next theorem;
the proof is given in \refSS{SSrotation-symbolic}.

\begin{theorem}\label{Teps}
  For any orbit $(x_i)_0^\infty$ for $f_\pm$, the corresponding symbolic
  sequence
$(\eps_i)_0^\infty$ satisfies
\begin{equation}
\sum_{i=0}^n \eps_i= \rho n+O(1),   
\end{equation}
where $\rho$ is the rotation number of $f_\pm$.
In particular,
$\sum_{i=0}^{n-1} \eps_i/n\to\rho$ as \ntoo.
\end{theorem}

\section{Periodic points}\label{Speriodic}\noindent
Recall the definition of periodic points in \refSS{SSdeforbit}.
\begin{lemma}
  \label{LA01}
$0$ and $1$ cannot both be periodic points of $f_\pm$.
\end{lemma}

\begin{proof}
  Suppose that 0 is a periodic point, and consider a minimal periodic orbit
  $x_0,\dots,x_{k-1}$ with $x_0=0$. Recall that $f_\pm\qw$ is single-valued,
  and $f_\pm\qw(0)=\tau$. Thus $x_{k-1}=\tau$. Furthermore, if $x_i=1$ for
  some $i\le k-1$, then $i>0$ and
$x_{i-1}=f_\pm\qw(1)=\tau=x_{k-1}$, which is impossible
  since this   periodic orbit is minimal. Consequently, the backwards orbit 
$Q:=\set{f_\pm^{-n}(\tau):n\ge0}=\set{x_j:0\le j<k}$ contains 0 but not 1.

Similarly, if 1 is a periodic point, then $Q$ contains 1 but not 0.

Thus these two events exclude each other. 
\end{proof}
Note that the proof is valid also when $\tau\in\setoi$, which occurs
precisely in the simple cases in Examples \ref{E1} and \ref{E0},
and when $\tau$ does not exist (then 0 and 1 are not in the image of
$f_\pm$, and thus certainly not periodic points). 

\begin{lemma}\label{LA}
Suppose that $p\in\oi$ is a periodic point of $f_\pm$.
Then $p$ is a periodic point of $f_-$ or $f_+$ (or both).
\end{lemma}

\begin{proof}
By assumption, there exists $k\ge1$ and a periodic orbit
$C=\set{p_0,\allowbreak\dots,p_{k-1}}$ with $p_0=p$. 
By \refL{LA01}, 0 and 1 cannot both appear in $C$.
If $0\notin C$, then $C$ is a periodic orbit of $f_+$, and 
if $1\notin C$, then $C$ is a periodic orbit of $f_-$.
\end{proof}

\begin{theorem}\label{T1}
Suppose that $f_\pm$ has a periodic orbit $C$. 
Then $f_\pm$ is asymptotically periodic and $C$ is the universal limit cycle for
$f_\pm$.
\end{theorem}

\begin{proof}
By assumption, there exists a periodic orbit $C=\set{p_0,\dots,p_{q-1}}$ of
$f_\pm$.

Suppose first that 1 is not a periodic point of $f_\pm$.
Then $p_i<1$ for every $i$, and it follows,
as in the proof of \refL{LA},
that $C$ is a periodic orbit of $f_-$.
We may assume that the orbit is minimal,
so $p_0,\dots,p_{q-1}$ are distinct.
We consider first only the action of $f_-$. 

Let $\xi_0,\dots,\xi_{q-1}$ be $p_0,\dots,p_{q-1}$ arranged in increasing
order; thus $0\le \xi_0<\dots<\xi_{q-1}<1$. Extend this to a doubly infinite
increasing sequence $\Xi=\set{\xi_n}_{-\infty}^\infty$ by
\begin{equation}\label{abc}
  \xi_{mq+i}:=\xi_i+m,
\qquad 0\le i<q,\; m\in\bbZ.
\end{equation}
It follows, using \refL{LF}, that $F_-$ maps the set $\Xi$ into itself.
Moreover, if $0\le i<q$, then 
$\pis_-\circ F_-^q(p_i)=f_-^q\circ\pis_-(p_i)=f_-^q(p_i)=p_i$ and thus
$F_-^q(p_i)=p_i+r_i$ for some $r_i\in\bbZ$.
It follows, using \refL{LF} again, that $F_-^q(\Xi)=\Xi$, and thus
$F_-:\Xi\to\Xi$ is onto. Since $F_-$ is strictly increasing, it follows that
there exists an integer $r$ such that
\begin{equation}\label{euclid}
  F_-(\xi_n)=\xi_{n+r},
\qquad n\in\bbZ.
\end{equation}
In particular, this implies that, recalling \eqref{abc},  
\begin{equation}\label{archimedes}
  F_-^q(\xi_n)=\xi_{n+qr}=\xi_n+r,
\qquad n\in\bbZ.
\end{equation}

Let $I_i:=(\xi_{i},\xi_{i+1}]$ and $\bI_i:=[\xi_{i},\xi_{i+1}]$,
for $i\in\bbZ$.
Since $F_-$ is strictly increasing, \eqref{euclid} implies  that
$F_-(\bI_i)\subseteq \bI_{i+r}$. Moreover, if $I_i\cap \bbZ=\emptyset$, then $F_-$
is linear (and thus continuous) on $\bI_i$, and  $F_-(\bI_i)=\bI_{i+r}$; since
$F_-$ has contraction factor $a$, this implies $|\bI_{i+r}|=a|\bI_i|$.

Suppose that none of the $q$ intervals $I_i, I_{i+r}, \dots, I_{i+(q-1)r}$
contains an integer. Then $F_-$ is a linear contraction
$\bI_{i+jr}\to\bI_{i+(j+1)r}$ for each $j$, and in particular
$|\bI_{i+(j+1)r}|=a|\bI_{i+jr}|$. Hence, 
$|\bI_{i+qr}|=a^q|\bI_{i}|$, which is a contradiction, since
$\bI_{i+qr}=\bI_i+r$ by \eqref{abc}.

Consequently, for each $i$, at least one of the
$q$ intervals $I_i, I_{i+r}, \dots, I_{i+(q-1)r}$ contains an integer.
Taking $i=i_0,\dots,i_0+r-1 $ for some $i_0$, 
we see that the $rq$ disjoint intervals $I_j$, 
$i_0\le j<i_0+rq$, 
contain at least $r$ integers. On the other hand, the union of
these intervals is $(\xi_{i_0},\xi_{i_0+rq}]=(\xi_{i_0},\xi_{i_0}+r]$, which contains
exactly $r$ integers. It follows that for every $i\in\bbZ$, 
exactly one of the
$q$ intervals $I_i, I_{i+r}, \dots, I_{i+(q-1)r}$ contains an integer.
(Also, no $I_i$ contains two integers.)

Suppose that 
$j\in\bbZ$ is such that
$I_j$ contains an integer $\ell_j$.
Then $F_-$ is linear on
$I_j':=[\xi_j,\ell_j)$ and on $I_j'':=[\ell_j,\xi_{j+1}]$, and maps both
intervals into $\bI_{j+r}$. 
Since there is no integer in any of
$I_{j+r},\dots,I_{j+(q-1)r}$ by the argument above,
we can apply $F_-$ repeatedly and see that $F_-^m$ is linear on $I_j'$ and
$I_j''$ for $1\le m\le q$. In particular, $F_-^q$ is linear on $I_j'$ and
$I_j''$. Since $F_-^q(\xi_j)=\xi_j+r$ and $F_-^q(\xi_{j+1})=\xi_{j+1}+r$ by
\eqref{archimedes}, and $F_-^q$ has contraction factor $a^q<1$, it follows that
$F_-^q:I_j'\to I_j'+r$ and $F_-^q:I_j''\to I_j''+r$, and we can thus iterate
further.
Consequently, if $x\in I_j'$ then,
for every $n\ge0$,
\begin{equation}\label{puh}
F_-^n(x)-F_-^n(\xi_j)=a^n(x-\xi_j).
\end{equation}
It follows also, for example by \eqref{F-} and \eqref{puh} for $n$ and $n+1$,
that
$\floor{F_-^n(x)}=\floor{F_-^n(\xi_j)}$, and thus, using 
\eqref{puh} again and
\refL{LF}\ref{LFpi},  
\begin{equation}\label{puh2}
f_-^n(\frax{x})-f_-^n(\frax{\xi_j})
=\frax{F_-^n(x)}-\frax{F_-^n(\xi_j)}=a^n(x-\xi_j).
\end{equation}
Hence,
$f_-^n(\frax{x})-f_-^n(\frax{\xi_j})\to0$ as \ntoo,
and since $\frax{\xi_j}=\xi_{j\bmod q}\in C$, $\frax{x}$ is attracted to the
periodic orbit $C$ by $f_-$.
Similarly, if $x\in I_j''$, then
$f_-^n(\frax{\xi_{j+1}})-f_-^n(\frax{x})\to0$ as \ntoo,
and again $\frax{x}$ is attracted to $C$.
We have shown that if $x\in \bI_j$ and 
$I_j\cap\bbZ\neq\emptyset$, then $\frax x\in\oio$ is
attraced to $C$ by $f_-$.

Now let $x\in \bI_j$ with $j$ arbitrary. 
Then there exists $m$ with $0\le m<q$ such that $I_{j+mr}\cap\bbZ\neq\emptyset$.
Furthermore, $F_-^{m}(x)\in \bI_{j+mr}$, and thus the argument above 
applies to $F_-^m(x)$, and shows that
$\frax{F_-^m(x)}=f_-^m(\frax{x})$ is attracted to $C$ by $f_-$; consequently
also $\frax{x}$ is attracted to $C$.

This shows 
that every $x\in\oio$ is attracted to the periodic orbit $C$ by $f_-$.
Moreover, $f_-(1)\in\oio$, and thus it follows that $1$ too is attracted to
$C$ by $f_-$.

It remains to show that every point is attracted to $C$ also by $f_\pm$,
\ie, even when we allow $\tau\to f_+(\tau)=1$ instead of $\tau\to f_-(\tau)=0$.
If  $\set{x_n}$ is  an orbit that makes
the transition $\tau\mapsto 1$ only once, 
then the development after this is  by $f_-$,
and thus the sequence is attracted to $C$.
The only possible problem is thus when we make the transition $\tau\mapsto
1$ at least twice, but 
then $1$ appears at least twice in the orbit $\set{x_n}$, and thus there is
a periodic orbit containing 1, contradicting our assumption.

This completes the proof that if 1 is not a periodic point, then every orbit
is attracted to $C$. 

If 0 is not a periodic point, the same conclusion holds
by mirror symmetry, see \refR{Rreflection}, or by repeating the proof above
with $F_+$ instead of $F_-$, mutatis mutandis.

Since either 0 or 1 is not a periodic point by \refL{LA01}, this completes
the proof.
\end{proof}

\begin{corollary}  \label{C1}
The dynamical system $f_\pm$ has at most one periodic orbit.
\nopf
\end{corollary}

It follows from \eqref{archimedes} in the proof above that if $f_\pm$
has a periodic orbit, then the rotation number is rational
($r/q$ in the notation above). 
In fact, the converse holds too; we return to this in \refT{TR}.

\section{A classification of orbits}\label{Sorbits+}\noindent
We now clarify what the possibilities are for orbits of $f_\pm$.

If $x\in\oi$ has an orbit for $f_\pm$ that does not contain $\tau$, then
there is 
never any choice, and this orbit is simultaneously the orbit of $x$ for both
$f_-$ and $f_+$, and the unique orbit for $f_\pm$.
Hence, our consideration of the multi-valued $f_\pm$ lead to complications
only when $x$ has an orbit containing $\tau$, 
\ie, when $x$ is in the countable (or finite) set
$\AAA:=\set{f_\pm^{-n}(\tau):n\ge0}$. 

Consider first the case when $\tau$ does not belong to any periodic orbit. Then
no orbit can contain $\tau$ more than once; hence if $x$ has an orbit
containing $\tau$, then $\tau$ will not appear again, which means that there
are no further choices.
Consequently, if $x\in\AAA$, then $x$ has exactly two orbits for $f_\pm$, one is
its orbit 
for $f_-$ and the other is its orbit for $f_+$; furthermore, both orbits agree
until they reach $\tau$, and then they follow the unique orbits of 0 and 1 (for
$f_-$, $f_+$ or $f_\pm$). Hence, for the asymptotical behaviour of the
orbits, it does not matter whether we consider $f_-$, $f_+$ or $f_\pm$.

On the other hand, if $\tau$ belongs to a periodic orbit $C$, and $x\in\AAA$,
then $x$ has an infinite number of orbits for $f_\pm$: the orbit is unique
until we reach $\tau$, but then we can either continue along the periodic
orbit $C$ repeatedly for ever, or we can go around $C$
$N$ times, where $N=0,1,2,\dots$, and then make the other choice at $\tau$;
this brings us to either 0 or $1\notin C$, and then we cannot come back to
$\tau$, by 
\refL{LA01}, so the orbit continues with the unique orbit of 0 or 1.

This leads to the following possibilities for the orbits of an arbitrary
$x\in\oi$. 

\case{Case 1. There exists a periodic orbit $C$.}
By \refT{T1} (and \refC{C1}), 
$C$ is the only periodic orbit, and every orbit is asymptotic
to $C$.
We distinguish two subcases.

\case{Case 1a. $\tau\notin C$.}
Then $\tau$ does not belong to any periodic orbit, and thus
no orbit can contain $\tau$ more than once.
Hence, starting at an arbitrary $x\in\oi$, either there is a unique orbit
for $f_\pm$ ($x\notin\AAA$), or there are two orbits ($x\in\AAA$), 
one (the orbit for $f_-$) containing 0 and one (the orbit for $f_+$)
containing 1. All orbits 
are asymptotic to $C$.
Hence, $\go_{f_\pm}(x)=C$ for every $x\in\oi$.
Furthermore, it follows from the proof of \refT{T1} (see \eqref{puh2}) 
that the orbits converge uniformly 
to $C$, and thus $\gL_\pm=\gL_-=\gL_+=C$.

\case{Case 1b. $\tau\in C$.}
Then either $0\in C$ or $1\in C$, but not both (\refL{LA01}).
Suppose that $0\in C$. (The case $1\in C$ is symmetric, with $0$ and $1$
and the indices $+$ and $-$ interchanged below.)

If $x\notin\AAA$, then $x$ has a unique orbit, which by \refT{T1} is
asymptotic to $C$. If $x\in\AAA$, then $x$ has an infinite number of orbits,
as described above; one follows eventually $C$ for ever (this is the orbit
for $f_-$), while all others eventually follow the unique orbit of $1$. 
Each orbit is asymptotic to $C$, and
$\go_{f_\pm}(x)=C$ for every $x\in\oi$. However, for $x\in\AAA$, the orbits
do not converge to $C$ uniformly. It follows easily that if $O_1$ is the
(unique) orbit of $1$, then
$\gL_\pm=C\cup O_1$, $\gL_-=C$ and $\gL_+=\emptyset$.

\case{Case 2. There is no periodic orbit of $f_\pm$.}
As in Case 1a, any $x\in\oi$ has either one or two orbits.
$\gL_\pm$ is infinite, and
we shall see in \refS{Sirrational}
that $\go_{f_\pm}(x)=\gL_\pm$ for every $x\in\oi$.
Furthermore, the orbits converge to $\gL_\pm$ uniformly.

\section{The rotation number}\label{Srotation}\noindent
For completeness, we supply a simple proof of the existence of a rotation 
number in our context 
(\refL{Lrho}), based on earlier proofs for $f_-$,
see \eg{} \cite{conze}, \cite{cout}.
We also, again for completeness, prove the simple consequence 
\refT{Teps}, that the proportion of $1$'s in the symbolic sequence converges
to the rotation number; see again \cite{cout} for $f_-$.
Finally, we use this to show that every orbit has an asymptotic average,
which is independent of the orbit.

\subsection{Existence of the rotation number}
\begin{proof}[Proof of \refL{Lrho}]
We first observe that for any $x\in\bbR$ and any $n\ge0$,
  \begin{equation}\label{p1}
	\bigabs{F_-^n(x)-x-F_-^n(0)}<1.
  \end{equation}
In fact, $F_-^n(x)-x$ has period 1, so it suffices to consider $x\in[0,1)$,
  and then
  \begin{equation}
F_-^n(x)-x\le	F_-^n(x)< F_-^n(1)=F_-^n(0)+1
  \end{equation}
and
\begin{equation}
  F_-^n(x)-x>F_-^n(x)-1\ge F_-^n(0)-1,
\end{equation}
which verifies \eqref{p1}.

Taking $x=F_-^m(0)$ in \eqref{p1} we obtain, for $m,n\ge0$,
\begin{equation}\label{fmn0}
\bigabs{ F_-^{m+n}(0)-F_-^m(0)-F_-^n(0)}<1.
\end{equation}
Consequently,
\begin{equation}
  F_-^{m+n}(0)+1 \le \bigpar{F_-^m(0)+1}+\bigpar{F_-^n(0)+1},
\end{equation}
i.e., the sequence $F_-^n(0)+1$ is subadditive. 
As is well-known, this implies the existence of the limit
\begin{equation}\label{rhoo2}
\rho=
 \lim_{\ntoo}\frac{F_-^n(0)+1}n =   \inf_{n\ge1}\frac{F_-^n(0)+1}n \ge-\infty.
\end{equation}
We thus have $F_-^n(0)/n\to\rho$ as \ntoo, and 
it follows from \eqref{p1} that $F_-^n(x)/n\to\rho$ for any $x\in\bbR$.

The corresponding result for $F_+$ then holds too since, if
$x'<x<x''$, then, by \refL{LF}\ref{LF<} and $F_+(x)=F_-(x-)$, 
$F_-(x')<F_+(x)<F_-(x'')$, 
and thus by induction $F_-^n(x')<F_+^n(x)<F^n_-(x'')$. 
Hence, \eqref{lrho} holds.

Furthermore, \eqref{fmn0} implies similarly that sequence $F_-^n(0)-1$ is
superadditive, and thus also 
\begin{equation}\label{rhoo3}
\rho=
 \lim_{\ntoo}\frac{F_-^n(0)-1}n =   \sup_{n\ge1}\frac{F_-^n(0)-1}n.
\end{equation}
By \eqref{rhoo2} and \eqref{rhoo3}, 
 $n\rho \le F_-^n(0)+1$ and
 $n\rho \ge F_-^n(0)-1$. 
Consequently, 
\begin{equation}\label{er}
\rho n-1 \le F_-^n(0)\le \rho n+1,
\qquad n\ge0.
\end{equation}
It follows from \eqref{er} and \eqref{p1} that for any real $x$,
\begin{equation}
\rho n+x-2<  F_-^n(x)< \rho n+x+2,
\qquad n\ge0,
\end{equation}
showing \eqref{lrhoO}.

If $x\ge0$, then by \eqref{F-}, $F_-(x)\ge F_-(0)=b\ge0$, and thus
by induction $F_-^n(0)\ge0$ for all $n\ge1$; hence $\rho\ge0$.
Similarly, \eqref{F-} implies 
$F_-(x)-x=b-(1-a)\frax{x}\in[b+a-1,b]$, and 
hence by induction $n(a+b-1)\le F_-^n(0)\le nb$.
Consequently, $a+b-1\le \rho\le b<1$, showing both \eqref{rho-ab} and
$\rho\in\oio$. 

Finally, if $a+b\le1$, then $x\in\oio$ implies by \eqref{F-} 
$F_-(x)= ax+b<a+b\le1$ and thus $F_-(x)\in\oio$; 
hence $F_-^n(0)\in\oio$, and $\rho=\lim_\ntoo F^n(0)/n=0$.
The converse follows by \eqref{rho-ab}.
\end{proof}

\subsection{Proof of \refT{Teps}}\label{SSrotation-symbolic}
Suppose first that the orbit does not contain 1; then
  $x_{n+1}=f_-(x_n)=\frax{ax_n+b}$ for $n\ge0$,
and it follows from \eqref{F-} and \eqref{eps2} by induction that
\begin{equation}
F_-^n(x_0)=x_n+\sum_{i=0}^{n-1}\eps_i.  
\end{equation}
Hence, $\sum_{i=0}^{n-1}\eps_i=F_-^n(x_0)+O(1)=n\rho+O(1)$ by \eqref{lrhoO},
and the result follows.

If the orbit contains only a finite number of 1's, then the result follows
by considering the part of the orbit after the last 1.

Similarly, if the orbit does not contain 0, then 
\begin{equation}
F_+^n(x_0)=x_n+\sum_{i=0}^{n-1}\eps_i,  
\end{equation}
and the conclusion follows by  \eqref{lrhoO}.
Again, this extends to any orbit with a finite number of 0's.

The only remaining case is thus an orbit that contains an infinite number of
0's and an infinite number of 1's. However, no such orbit can exist; in
fact, if there were an orbit with both 0 and 1 ocurring more than once, then
both 0 and 1 would be periodic points, but that is impossible by
\refL{LA01}.
\qed

\subsection{The average of an orbit}\label{SSmean}

The following theorem shows that every orbit has an average, in the sense of
the limit of the average of the $n$ first points; 
furthermore, this limit is independent of the orbit, and we provide an
explicit formula.

\begin{theorem}\label{Tmean}
  Let $(x_n)\ooz$ be any orbit of $f_\pm$, with  any initial point
  $x_0\in\oi$.
Then, as \ntoo,
\begin{equation}\label{tmean}
\frac{1}{n}\sum_{i=0}^{n-1} x_i \to 
\mean:=\frac{b-\rho}{1-a}.
\end{equation}
\end{theorem}

\begin{proof}
  Let $S_n:=\sum_{i=0}^{n-1} x_i$. Then, using \eqref{eps2} and \refT{Teps},
  \begin{equation}
    \begin{split}
aS_n+nb&
=\sum_{i=0}^{n-1}(ax_i+b)    
=\sum_{i=0}^{n-1}(x_{i+1}+\eps_i)    
=\sum_{i=1}^{n}x_{i}+\sum_{i=0}^{n-1}\eps_i
\\&
=S_n+x_n-x_0+\rho n+O(1)
=S_n+n\rho+O(1).
    \end{split}
  \end{equation}
Consequently,
\begin{equation}
  S_n=n\frac{b-\rho}{1-a}+O(1).
\end{equation}
This implies \eqref{tmean}.
\end{proof}

In particular, if there exists a periodic orbit $(x_n)_0^{k-1}$, then the
average of the points in the orbit is $\chi$.
For an example, see \refE{E2/3}, where $\rho=1/2$ and $\mean=1/3$.

For a more trivial example, suppose that 
 there is a fixed point $p_0$. Then $\rho=0$, and \eqref{tmean} implies
that $p_0=\chi=b/(1-a)$, as is immediately seen directly.

\section{Location of the rotation number}\label{Sloc}\noindent
The dependency of the rotation number $\rho(a,b)$ on $a$ and $b$ was 
investigated by 
\citet{DingHemmer},
\citet{bugeaudCR},
\citet{conze} 
and 
\citet{cout}.
We use and combine some of their ideas and develop them further.
There are large overlaps with the results of the references just mentioned;
we nevertheless give full proofs. 

In this  section, 
$\rho$ denotes an arbitrary real number. We do not assume that $\rho$ equals
the rotation number $\rho(a,b)=\rho(f_\pm)$ unless explicitly said so; on the
contrary, our aim is to let $\rho$ vary freely in order to eventually
derive conditions
for the equality $\rho=\rho(f_\pm)$.

We define, 
following \citet{cout},
for $\rho\in\bbR$ and $x\in\bbR$,
  \begin{equation}\label{phirho}
  \phi_\rho(x) = 
  \phi_{\rho,a,b}(x) := 
\frac{b}{1-a}+ (1-a)\sum_{j=0}^\infty a^j\floor{x - (j+1)\rho}.
  \end{equation}
The sum obviously converges absolutely, so each $\phi_\rho$ is a function
$\bbR\to\bbR$. 

It follows from \eqref{phirho} that
\begin{equation}
   \label{pr-1}
\phirho(x+1)=\phirho(x)+1,
\qquad x\in\bbR.
\end{equation}

We state some further simple properties of
the function $\phirho$.
\begin{lemma}\label{Lpr}
For any $\rho\in\bbR$, $\phi_\rho:\bbR\to\bbR$ has the following properties.
  \begin{romenumerate}
\item \label{pr-mono1}
$\phirho$ is weakly increasing: if $x\le y$, then $\phirho(x)\le\phirho(y)$. 

\item \label{pr-mono2}
If $\rho$ is irrational, then $\phi_\rho$ is strictly increasing, while
if $\rho$ is rational, with denominator $q$, then $\phi_\rho$ is constant on
each interval $[\frac{k}q,\frac{k+1}q)$.

\item \label{pr-disco}
The set of discontinuity points of $\phirho$ is 
\begin{equation}\label{drho}
\Drho:=\set{n + m\rho:m \in \bbZgto, n \in \bbZ},
\end{equation}
and $\phirho$ has a jump discontinuity at each $x\in\Drho$.
In particular, if $\rho$ is irrational, then the set of discontinuity points
is dense in $\bbR$.   

\item\label{pr-rc}
$\phirho(x)$ is right-continuous.

  \end{romenumerate}
\end{lemma}

\begin{proof}

\pfitemref{pr-mono1}
This is clear from \eqref{phirho}, because each
$\floor{x-(j+1)\rho}$ is weakly 
increasing and the coefficients in \eqref{phirho} are
positive. 

\pfitemref{pr-disco}
First, note that each discontinuity is a jump
discontinuity by \ref{pr-mono1}.

Let $y\in\Drho$, so $y=n+m\rho$ with $m\ge1$. 
Then, in the sum in \eqref{phirho}, 
 the term with $j=m-1$ has a positive jump at $x=y$. 
The sum of all other terms is a weakly increasing function of $x$, 
since each term is;
hence, the sum in \eqref{phirho} has a positive jump at $y$, and 
$\phirho(y)>\phirho(y-)$. 

Conversely, if $y\notin\Drho$, then every
term in the sum in \eqref{phirho}
is continuous at $x=y$.
Since the sum converges uniformly on bounded sets, it follows
that $\phirho(x)$ is continuous at $y$.

Finally, it is well known that if $\rho$ is irrational,
 then the sequence $(\frax{m\rho})_{m\ge1}$ is dense in $\oo$, and thus
$\Drho$ is dense in $\bbR$.

\pfitemref{pr-mono2}
 If $\rho=p/q$, 
and $x\in [\frac{k}q,\frac{k+1}q)$, 
then 
$\floor{x-(j+1)\rho} = \floor{\frac{k}q-(j+1)\rho}$ for every $j$
and thus
$\phi_\rho(x)=\phi_\rho(\frac{k}q)$.

On the other hand, if $\rho$ is irrational and $x<y$, then there exists
by \ref{pr-disco} a discontinuity point $z\in (x,y)$.
Hence $\phirho(x)\le\phirho(z-)<\phirho(z+)\le\phirho(y)$.

\pfitemref{pr-rc}
This follows because each
$\floor{x-(j+1)\rho}$ is right-continuous, and the sum in \eqref{phirho}
converges uniformly on compact intervals.
\end{proof}

In particular, it follows from \eqref{drho} that $0\in\Drho$ if and only if
$\rho$ is rational, and hence
\begin{equation}\label{jump}
  \begin{cases}
    \phirho(0)>\phirho(0-),& \text{if } \rho\in\bbQ,\\
    \phirho(0)=\phirho(0-),& \text{if } \rho\notin\bbQ.
  \end{cases}
\end{equation}

\begin{lemma}\label{LX}
Suppose that
\begin{equation}\label{good}
  \phirho(0-)\le 0\le\phirho(0).
\end{equation}
Then
  \begin{romenumerate}
\item \label{LXgood0}
If $\rho$ is irrational, or $\phirho(0-)<0$,
then, for all $x\in\bbR$,
\begin{align}\label{floorphirho}
  \floor{\phi_\rho(x)}&=\floor x, 
\\
  \frax{\phi_\rho(x)}&=\phi_\rho(\frax x), \label{fraxphirho} 
\intertext{and} 
\label{rhoF-}
  F_-(\phirho(x))&=\phirho(x+\rho), \\
  f_-(\frax{\phirho(x)})
&=\frax{\phirho(x+\rho)}
=\phirho(\frax{x+\rho}). 
\label{rhof-}
\end{align}

\item \label{LXgood1}
If $\rho$ is irrational, or $\phirho(0)>0$,
then, for all $x\in\bbR$,
\begin{equation}\label{ceilphirho}
  \ceil{\phi_\rho(x-)}=\ceil x, 
\end{equation}
and 
\begin{align}\label{rhoF+}
  F_+(\phirho(x-))&=\phirho((x+\rho)-),
\\
  f_+(\fraxx{\phirho(x-)})&=\fraxx{\phirho((x+\rho)-)}. 
\label{rhof+}
\end{align}
  \end{romenumerate}
\end{lemma}

Note that \eqref{jump} shows that if 
\eqref{good} holds, then at least one
of \ref{LXgood0} and \ref{LXgood1} applies. Furthermore, if $\rho$ is
irrational, then \eqref{good} holds if and only if $\phirho(0)=0$.

\begin{proof}

\pfitemref{LXgood0}
By  monotonicity and \eqref{good}, if $x\ge0$, then
$\phirho(x)\ge\phirho(0)\ge0$. 
Similarly, if $x<0$ and $\phirho(0-)<0$, then
$\phirho(x)\le\phirho(0-)<0$. 
Furthermore, if $x<0$ and $\rho\notin\bbQ$,
then $\phirho(x)<\phirho(x/2)\le\phirho(0-)\le0$ by \refL{Lpr}\ref{pr-mono2}
and \eqref{good}. Hence, in both cases, $x<0\implies \phirho(x)<0$, and it
follows from \eqref{pr-1} that
$x<1\implies \phirho(x)<1$. Consequently, $x\in\oio \implies \phirho(x)\in\oio$,
which yields \eqref{floorphirho} and \eqref{fraxphirho} by \eqref{pr-1}.

Moreover, by \eqref{F-}, \eqref{floorphirho} and  \eqref{phirho},
\begin{equation}\label{didiF-}
  \begin{split}
F_-(\phirho(x))&
=a\phirho(x)+b+(1-a)\floor{\phirho(x)}	
=a\phirho(x)+b+(1-a)\floor{x}	
\\&
= \frac{ab}{1-a}+(1-a)\sum_{j=0}^\infty a^{j+1}\floor{x - (j+1)\rho}
+b
+(1-a)\floor{x}	
\\&
=\frac{b}{1-a}+(1-a) \sum_{k=0}^\infty a^{k}\floor{x - k\rho}
=\phirho(x+\rho).
  \end{split}
\raisetag{1.5\baselineskip}
\end{equation}
Finally, \eqref{rhof-} follows from \eqref{didiF-} by \refL{LF} and
\eqref{fraxphirho}. 

\pfitemref{LXgood1}
In this case we similarly see that
$x\in\ooi \implies \phirho(x-)\in\ooi$,
and \eqref{ceilphirho} follows by \eqref{pr-1}.
Then, \eqref{rhoF+} follows as in \eqref{didiF-}.
(By \eqref{pr-1}, it suffices to consider  $x\in\ooi$.)
Finally, \refL{LF} yields \eqref{rhof+}.
\end{proof}

Let, for $\rho\in\bbR$,
\begin{equation}\label{psi}
  \begin{split}
  \psi(\rho):=\phirho(0)
&
=\frac{b}{1-a}+ (1-a)\sum_{j=0}^\infty a^j\floor{ - (j+1)\rho}
\\&
=\frac{b}{1-a}- (1-a)\sum_{j=0}^\infty a^j\ceil{(j+1)\rho}.
  \end{split}
\end{equation}

\begin{lemma}\label{Lpsi}
  \begin{thmenumerate}
  \item \label{Lpsi-lc}
$\psi(\rho)$ is left-continuous and strictly decreasing.
  \item \label{Lpsi-disco}
$\psi(\rho)$ is continuous at every irrational $\rho$ and has a jump at
	every rational $\rho$. 
  \item \label{Lpsi-}
The right limits are given by
\begin{equation}\label{psi+}
  \begin{split}
  \psi(\rho+)=\phirho(0-)
&
=\frac{b}{1-a}-1- (1-a)\sum_{j=0}^\infty a^j\floor{(j+1)\rho}.	
  \end{split}
\end{equation}
\item \label{Lpsi-01}
$\psi(0)\ge0$ and $\psi(1)<0$. Furthermore,  $\psi(0+)>0 \iff a+b>1$.
  \end{thmenumerate}
  \end{lemma}

\begin{proof}
 \pfitemref{Lpsi-lc} 
The left-continuity follows from \eqref{psi}, since each $\ceil{(j+1)\rho}$
is left-continuous, and the sum converges uniformly on bounded domains.

That $\psi(\rho)$ is weakly
decreasing follows also from \eqref{psi}. Furthermore,
if $\rho_1<\rho_2$,
then there exist $j$ such that $(j+1)(\rho_2-\rho_1)>1$ and then
$\ceil{(j+1)\rho_1}<\ceil{(j+1)\rho_2}$; hence $\psi(\rho_1)>\psi(\rho_2)$.
Thus $\psi$ is strictly decreasing.

\pfitemref{Lpsi-disco}
If $\rho$ is irrational, then every $\ceil{(j+1)\rho}$ is continuous at
$\rho$, and thus \eqref{psi} implies that $\psi$ is continuous at $\rho$,
again using the fact that the sum converges uniformly on bounded domains.

Conversely, if $\rho$ is rational, then $(j+1)\rho\in\bbZ$ for some $j$, and
then $\ceil{(j+1)\rho}$ has a jump at $\rho$. (There will be infinitely many
such $j$, but all jumps are in the same direction, so there is no cancellation.)

\pfitemref{Lpsi-}
For any $x,\rho\in\bbR$ and $j\ge0$, 
\begin{equation}
  \lim_{\rho'\downto\rho}\floor{x-(j+1)\rho'}
=
  \lim_{x'\upto x}\floor{x'-(j+1)\rho}.
\end{equation}
Hence, \eqref{phirho} yields, using local uniform convergence of the sums again,
\begin{equation}
\phi_{\rho+}(x):=  \lim_{\rho'\downto\rho}\phi_{\rho'}(x)
=
  \lim_{x'\upto x}\phi_\rho(x')=\phirho(x-).
\end{equation}
Now take $x=0$ 
to obtain $\psi(\rho+)=\phi_{\rho+}(0) =\phi_\rho(0-)$.
Finally, use \eqref{psi} and $\ceil{y+}=\floor{y}+1$.

\pfitemref{Lpsi-01}
Simple calculations using \eqref{psi} and \eqref{psi+} yield
\begin{align}
 \psi(0)&=\frac{b}{1-a}, \label{psi0}\\
\psi(1)&=\frac{b}{1-a}-\frac{1}{1-a}=-\frac{1-b}{1-a},\label{psi1}\\
 \psi(0+)&=\frac{b}{1-a}-1
=\frac{a+b-1}{1-a}.\label{psi0+}
\end{align}
\end{proof}

By \eqref{psi} and \eqref{psi+}, \eqref{good} is equivalent to
\begin{equation}
  \label{goodpsi}
\psi(\rho+)\le0\le\psi(\rho).
\end{equation}

\begin{lemma}
  \label{LY}
Let $\rho\in\bbR$. Then $\rho$ equals the rotation number 
$\rho(f_\pm)=\rho(a,b)$ 
of $f_\pm$ if and only
if \eqref{goodpsi} holds (or, equivalently, \eqref{good} holds).
\end{lemma}

\begin{proof}
Suppose first that \eqref{goodpsi} holds, and thus also \eqref{good}. 
As noted above, then
\refL{LX}\ref{LXgood0} or \ref{LXgood1} applies.
If \refL{LX}\ref{LXgood0} applies, then \eqref{floorphirho} implies
$|\phirho(x)-x|<1$, and thus by iterating \eqref{rhoF-},
\begin{equation}
  F_-^n(\phirho(0))=\phirho(n\rho)=n\rho+O(1), \qquad n\ge0;
\end{equation}
hence $F_-^n(\phirho(0))/n\to\rho$ as \ntoo, 
and thus the rotation number $\rho(f_\pm)=\rho$.

A similar argument works if \refL{LX}\ref{LXgood1} applies.

For the converse, let 
\begin{equation}\label{rhox}
  \rhox:=\sup\set{\rho:\psi(\rho)\ge0}.
\end{equation}
\refL{Lpsi} implies that $\rhox$ is well-defined, with $0\le\rhox\le1$;
furthermore, the left-continuity of $\psi$ implies 
$\psi(\rhox)\ge0$, so the supremum in \eqref{rhox} is attained (and is thus
a maximum). Furthermore, by \eqref{rhox}, $\psi(\rho)<0$ for $\rho>\rhox$, and
thus $\psi(\rhox+)\le0$.

Hence, $\psi(\rhox+)\le0\le\psi(\rhox)$, 
\ie{}
\eqref{goodpsi} holds for $\rho=\rhox$; 
as shown above this implies that
$\rhox$ equals the rotation number $\rho(f_\pm)$.
Consequently, \eqref{goodpsi} holds when $\rho=\rho(f_\pm)$.
\end{proof}

The rotation number $\rho(f_\pm)=\rho(a,b)$ depends on $a$ and $b$ in a
rather complicated way. 
Similarly, the function $\psi(\rho)$ depends on $a$ and $\rho$ in rather
complicated 
ways, but its dependency on $b$ is simple.

We define
\begin{align}\label{b-}
b_-(a, \rho) 
&= (1-a)^2 \sum_{j=0}^\infty a^j \ceil{(j+1)\rho},
\\
b_+(a, \rho)& 
=1-a+(1-a)^2 \sum_{j=0}^\infty a^j\floor{(j+1)\rho}.
\label{b+}
\end{align}
Then, by \eqref{psi} and \eqref{psi+},
\begin{align}
(1-a)  \psi(\rho) &= b-b_-(a,\rho),\label{wb-}
\\
(1-a)  \psi(\rho+) &= b-b_+(a,\rho).\label{wb+}
\end{align}
Note that $b_-(a,\rho)\le b_+(a,\rho)$, with equality if and only if $\rho$
is irrational, as is easily seen directly from \eqref{b-}--\eqref{b+}, or by 
\eqref{wb-}--\eqref{wb+} and \refL{Lpsi}\ref{Lpsi-disco}.
Furthermore, $b_-(a,\rho)$ and $b_+(a,\rho)$ are 
strictly increasing functions of $\rho$, and $b_+(a,\rho)=b_-(a,\rho+)$.

By \eqref{wb-} and \eqref{wb+},
\begin{align}
  \psi(\rho)\ge0 & \iff b \ge b_-(a,\rho),\label{ele+}
\\
  \psi(\rho+)\le0 & \iff b \le b_+(a,\rho),\label{ele0}
\end{align}

We can now rephrase and expand
\refL{LY}, regarding $a$ and $\rho$ as given and $b$
as varying.
This yields the following theorem, essentially due to \citet{bugeaudCR} 
(in a different form, see \refR{RBC} below), see also \citet{conze} and
\citet{DingHemmer}. 

\begin{theorem}\label{TLb}
Fix $a\in(0,1)$ and $\rho\in[0,1)$.
Then $0\le b_-(a,\rho)\le b_+(a,\rho)<1$.
Moreover, the rotation number $\rho(a,b)$ 
of $f_\pm$ equals $\rho$ if and only if
\begin{equation}\label{bbb}
b_-(a,\rho)\le b\le b_+(a,\rho).   
\end{equation}
Furthermore,
  \begin{romenumerate}
  \item \label{TLbi}
If $\rho\notin\bbQ$, then $b_-(a,\rho)=b_+(a,\rho)$. 
Hence there is a unique value of $b$
such that the rotation number $\rho(a,b)$  equals $\rho$. 
\item \label{TLbr<}
If $\rho\in\bbQ$, then $b_-(a,\rho)<b_+(a,\rho)$. 
Hence, there is an interval $I_{a,\rho}:=
[b_-(a,\rho), b_+(a,\rho)]$ of $b$ that give the same rotation number $\rho$
of $f_\pm$. 
If $\rho$ has denominator $q$ (in lowest terms), then $I_{a,\rho}$ has length 
\begin{equation}\label{tlbr<}
\leb{I_{a,\rho}}=
b_+(a,\rho)-b_-(a,\rho)
=
a^{q-1}(1-a)^2/(1-a^q).  
\end{equation}
  \end{romenumerate}
\end{theorem}

\begin{proof}
First, 
by \eqref{b-}, $b_-(a,0)=0$  and $b_-(a,1)=1$. Hence, $0\le\rho<1$  implies
$b_-(a,\rho)\ge0$ and $b_+(a,\rho)=b_-(a,\rho+)<1$. 

By \refL{LY}, $\rho=\rho(f_\pm)$ 
if and only if \eqref{goodpsi}
holds, which
by \eqref{ele+}--\eqref{ele0} is equivalent to \eqref{bbb}.

We have already remarked that $b_-(a,\rho)=b_+(a,\rho)$ if and only if
$\rho\notin\bbQ$. Hence it only remains to calculate $\leb{I_{a,\rho}}$.
We have, by \eqref{b-}--\eqref{b+},
\begin{equation}
  \begin{split}
b_+(a,\rho)-b_-(a,\rho)
=
(1-a)^2\sumjo a^j\Bigpar{1+\floor{(j+1)\rho}-\ceil{(j+1)\rho}}.
  \end{split}
\end{equation}
The big bracket in this sum is 0 or 1, and 1 if and only if $(j+1)\rho\in\bbZ$.
If $\rho=p/q$, this happens when $j=kq-1$ with $k\ge1$; hence
\begin{equation}
  \begin{split}
b_+(a,\rho)-b_-(a,\rho)
=
(1-a)^2\sumk a^{kq-1}=
(1-a)^2\frac{a^{q-1}}{1-a^q}.
  \end{split}
\end{equation}
\end{proof}

As remarked by \citet{DingHemmer} and \citet{conze},
it follows from \cite[Theorem 309]{HardyW} that
for any $a\in\ooio$,  
the sum of the lengths $\leb{I_{a,\rho}}$ for all rational $\rho\in\oio$ is,
considering only $p/q$ in lowest terms and letting $\gf$ be the Euler
totient function,
\begin{equation}\label{sum=1}
  \begin{split}
\biggabs{\bigcup_{\rho\in \bbQ\cap\oio} I_{a,\rho} }  
&=
\sum_{\rho\in \bbQ\cap\oio}\leb{ I_{a,\rho} }  
=
(1-a)^2 \sum_{p/q\in \bbQ\cap\oio}\frac{a^{q-1}}{1-a^q}
\\&
=
(1-a)^2 \sum_{q=1}^\infty \gf(q)\frac{a^{q-1}}{1-a^q}
=1
  \end{split}
\end{equation}
and hence for any fixed $a$, the rotation number is rational for almost
every $b\in\oio$. Furthermore, the exceptional set of $b$ has Hausdorff
dimension 0, see \cite{laurent-nogueira} and \refT{THb} below. 

\begin{remark}\label{RBC}
As simple consequences of \eqref{b-}--\eqref{b+}, we also have
\begin{align}\label{b-2}
b_-(a,\rho)&=(1-a)\sum_{j=0}^\infty a^j\bigpar{\ceil{(j+1)\rho}-\ceil{j\rho}}
\\
b_+(a,\rho)&=(1-a)
\Bigpar{1+\sum_{j=0}^\infty a^j\bigpar{\floor{(j+1)\rho}-\floor{j\rho}}}.
\label{b+2}
\end{align}
This shows that $b_-(a,\rho)$ and $b_+(a,\rho)$ coincide with the functions
defined (for the same purpose) by
\citet{bugeaudCR} and \citet{conze,BC-noise}.  
In their notation, our $b_-(a,\rho)$ is written $\tau_a(\rho)$ when $\rho$
is irrational, and $P^p_q(a)/(1+a+\dots+a^{q-1})$ when $\rho=p/q$ is rational;
$P^p_q(a)$ is a polynomial, and these polynomials are studied further in 
\cite{bugeaudCR,conze,BC-noise}.
\end{remark}

\begin{example}\label{Erho=1/2}
  For $\rho=1/2$, \eqref{b-} yields
  \begin{equation}
    b_-(a,\tfrac12)=(1-a)^2\sum_{k=0}^\infty\bigpar{a^{2k}+a^{2k+1}}\bigpar{k+1}
=(1-a)^2\frac{1+a}{(1-a^2)^2}
=\frac1{1+a}
  \end{equation}
and then \eqref{tlbr<} yields
\begin{equation}
  b_+(a,\tfrac12) =   b_-(a,\tfrac12)+\frac{a(1-a)^2}{1-a^2}
=\frac{1+a-a^2}{1+a}.
\end{equation}
Consequently, 
\begin{equation}\label{rho=1/2}
\rho(f_\pm)=\frac12
\iff
\frac{1}{1+a}\le b\le \frac{1+a-a^2}{1+a}.
\end{equation}
\end{example}

\section{Hausdorff dimension}\label{Shaus}\noindent
We use the results above to 
prove three theorems about the Hausdorff dimension of important sets.
The first two concern the exceptional set of parameters 
for which the  rotation number is irrational, 
and thus the invariant set of $f_\pm$ is a Cantor set;
in the third
theorem we study the invariant set itself. 

As said after \refT{TLb}, 
\citet{conze} showed that for any fixed $a$, 
the  exceptional  set of  $b$ 
that yield an irrational rotation number $\rho(a,b)$
has Lebesgue measure 0;
moreover, 
\citet[Theorem 2]{laurent-nogueira}
show the sharper result that this exceptional set has
Hausdorff dimension 0.
See also \cite{DingHemmer}.
We supply gauge functions 
that provide even finer information,
including both an upper and a lower bound on the \lq{size}' of the
exceptional set. 
Furthermore, we consider in \refT{THab}
the Hausdorff dimension of the two-dimensional parameter set
$(a,b)$ that yield irrational rotation numbers. 

Let $\cE$ be the exceptional set of all $(a,b)\in\ooio\times\oio$ such that
$f_{\pm,a,b}$ has irrational rotation number;
furthermore, for $a\in\ooio$, 
let $\cE_a$ be the set of $b\in\oio$ such that $(a,b)\in\cE$.
\begin{theorem}
  \label{THb}
For every $a\in\ooio$, the Hausdorff dimension of $\cE_a$ is $0$.
Moreover, the Hausdorff measure $\cH_h(\cE_a)<\infty$ 
for the gauge function $h(t)=1/|\log t|^2$,
but $\cH_h(\cE_a)>0$ 
for the gauge function $h(t)=1/|\log t|$,
\end{theorem}

\begin{proof}
  Fix $N>1$. There are less that $N^2$ intervals $I_{a,p/q}$ with
$q\le N$. (Here and throughout the proof we consider only $I_{a,p/q}$ with
  $p/q\in\oio$ and $p/q$ in lowest terms.)
Hence, their complement $A_N:=\ooio\setminus\bigcup_{q\le N} I_{a,p/q}$ 
is a union of at most $N^2$ (open) intervals. Each of these intervals has
length at most, recalling \eqref{sum=1},
\begin{equation}\label{haus1}
  \begin{split}
\leb{A_N}&
=1-\sum_{q\le N}\sum_p\leb{I_{a,p/q}}
=\sum_{q> N}\sum_p\leb{I_{a,p/q}}
\le \sum_{q> N} q (1-a)^2 \frac{a^{q-1}}{1-a^n}
\\&
\le  (1-a) \sum_{q> N} q a^{q-1}
=\bigpar{N+(1-a)\qw}a^N.
  \end{split}
\end{equation}
Since $\cE_a\subset A_N$, it follows that,
for any gauge function $h$
\begin{equation}\label{hausdorff2}
  \cH_h(\cE_a)\le \liminf_{\Ntoo} \bigpar{N^2h(2Na^N)}.
\end{equation}
Taking $h(t)=t^\ga$, we find $\cH_\ga(\cE_a)=0$ for every $\ga>0$, and
thus the Hausdorff dimension is 0.

Furthermore, taking $h(t)=1/|\log t|^2$ in \eqref{hausdorff} we obtain
$  \cH_h(\gL_\pm)<\infty$.

For the lower bound 
for the gauge function $h(t)=1/|\log t|$,
suppose that we have a covering
\begin{equation}
  \label{hausx1}
\cE_a\subseteq\bigcup_{k=1}^\infty I_k,
\end{equation}
where $I_k=[b_k',b_k'']\subseteq\oi$.

Let $J_k:=[\rho(a,b_k'),\rho(a,b_k'')]$. Then, every irrational
$\rho\in\ooio$ equals $\rho(a,b)$ for some $b\in\cE_a$; thus $b\in I_k$ for
some $k$ and then $\rho\in J_k$. Consequently, $\bigcup_k
J_k\supseteq\ooio\setminus\bbQ$, and taking the Lebesgue measure we obtain
\begin{equation}\label{hausx2}
  \sum_k |J_k|\ge1.
\end{equation}
We shrink each $J_k$ to $[\rho_k',\rho_k'']\subseteq J_k$ with
$\rho_k',\rho_k''$ irrational and $\rho_k''-\rho_k'\ge \frac12|J_k|$.
(Ignore  $J_k$ with $|J_k|=0$, if any.)
Then $b_-(a,\rho_k'),b_-(a,\rho_k'')\in I_k$.

Let $j_k:=\floor{(\rho_k''-\rho_k')\qw}\le 2|J_k|\qw$.
Then $(j_k+1)\rho_k''\ge (j_k+1)\rho_k'+1$, and thus
$\ceil{(j_k+1)\rho_k''}\ge \ceil{(j_k+1)\rho_k'}+1$.
Hence, \eqref{b-} implies
\begin{equation}\label{hausx3}
  |I_k|\ge b_-(a,\rho_k'')-b_-(a,\rho_k')\ge (1-a)^2 a^{j_k}.
\end{equation}
If $|I_k|\ge(1-a)^4$, then \eqref{hausx3} implies $a^{j_k}\le (1-a)^2$, and
thus by \eqref{hausx3} again, $|I_k|\ge a^{2j_k}$ and
\begin{equation}
  \frac{1}{\log(1/|I_k|)}
\ge \frac{1}{2j_k\log(1/a)}
\ge \frac{|J_k|}{4\log(1/a)}.
\end{equation}
Hence, for any covering \eqref{hausx1} with $\sup|I_k|\le(1-a)^4$,
using \eqref{hausx2},
\begin{equation}
\sum_k  \frac{1}{\log(1/|I_k|)}
\ge\sum_k \frac{|J_k|}{4\log(1/a)}
\ge \frac{1}{4\log(1/a)}.
\end{equation}
Consequently, with the gauge function $h(t)=1/|\log t|$ we have
\begin{equation}
\cH_h(\cE_a)\ge 1/(4\log(1/a)).  
\end{equation}
\end{proof}

For each fixed $\rho\in\oi$, the functions $b_-(a,\rho)$ and $b_+(a,\rho)$
defined in \eqref{b-}--\eqref{b+} are analytic functions of $a\in\ooio$, and
by
\refT{TLb}, for every irrational $\rho\in\ooio$, the set 
$(a,b)\in\ooio\times\oio$ such that 
$f_{\pm,a,b}$ has rotation number $\rho$ is the smooth curve
$\gG_\rho:=\set{(a,b_-(a,\rho)):a\in\ooio}$. Hence
$\cE=\bigcup_{\rho\in\ooio\setminus\bbQ} \gG_\rho$ is an uncountable union
of these smooth curves. Each curve $\gG_\rho$ obviously has Hausdorff
dimension $1$. We show that the same holds for their union $\cE$.

\begin{theorem}
  \label{THab}
The Hausdorff dimension of $\cE$ is $1$.
\end{theorem}

\begin{proof}
We develop the argument in the proof of \refT{THb} further, taking into
account the dependence on $a$.

Let $\az\in\ooio$ and consider only $a\in(0,\az]$;
let $\cEao:= E\cap\bigpar{(0,\az]\times\oio}$.  
We let $C$ denote unspecified constants that may depend on $\az$ (but not on
$N$ below).

Let $N>1$, and let $Q_N:=\set{\frac{p}{q}\in\bbQ\cap\oi:1\le q\le N}$.
Order the elements of $Q_N$ 
as $0=r_1<\dots<r_M=1$, where $M:=|Q_N|\le N^2$. 
(This is the well-known Farey series
\cite{HardyW}.) 

By \refT{TLb},
if $b_-(a,r_j)\le b\le b_+(a,r_j)$, then $\rho(a,b)=r_j\in\bbQ$.
Hence, recalling that $b_-(a,0)=0$ and $b_-(a,1)=1$, 
\begin{equation}\label{hausE}
  \cE\subset 
\bigcup_{j=1}^{M-1}\bigset{(a,b)\in\ooio\times\oio:b_+(a,r_j)<b<b_-(a,r_{j+1})}.
\end{equation}
For any $a\le \az$, and any $i<M$, \eqref{haus1} shows that
\begin{equation}\label{haus2}
0<b_-(a,r_{j+1})- b_+(a,r_j) \le \bigpar{N+(1-a)\qw}a^N
\le (N+C)\az^N 
.
\end{equation}

Let $\gd_N:=N\az^N$, 
$M':=\ceil{\az/\gd_N}$,
and $a_i:=i\az/M'$, $i=0,\dots,M'$; thus $a_{i}-a_{i-1}=\az/M'\le \gd_N$.
Let 
\begin{equation}\label{hausEij}
  E_{i,j}:=\bigset{(a,b)\in(a_{i-1},a_{i}]\times\oio:
b_+(a,r_j)<b<b_-(a,r_{j+1})}.
\end{equation}
Then, by \eqref{hausE},
\begin{equation}\label{haus3}
  \cEao\subseteq\bigcup_{\substack{1\le i\le M'\\1\le j<M}} E_{i,j}.
\end{equation}

It follows from \eqref{b-}--\eqref{b+} that 
\begin{equation}
 \lrabs{ \frac{\partial}{\partial a}b_-(a,\rho)},
 \lrabs{ \frac{\partial}{\partial a}b_+(a,\rho)}
\le C,
\end{equation}
uniformly for all $a\in\iao$ and $\rho\in\oi$.
Consequently, if $a\in(a_{i-1},a_{i}]$, then 
$|b_-(a,\rho)-b_-(a_i,\rho)|\le C\gd_N$ and $|b_+(a,\rho)-b_+(a_i,\rho)|\le
C\gd_N$ for every $\rho\in\oi$, and it follows from \eqref{hausEij} and
\eqref{haus2} that every set $E_{i,j}$ has diameter at most
$(N+C)\az^N+C\gd_N\le C N \az^N$.
By \eqref{haus3}, $\cEao$ is covered by less than 
$MM'\le C N^2/\gd_N=C N \az^{-N}$
such sets.
Consequently, for any $\ga>1$,
\begin{equation}
  \cH_\ga(\cEao)\le \liminf_{\Ntoo}C N \az^{-N}\bigpar{C N \az^N}^\ga
=0.
\end{equation}

Finally, $\cE=\bigcup_n \cE_{\le 1-1/n}$, and thus $\cH_\ga(\cE)=0$ for
every $\ga>1$.
\end{proof}

Our final theorem on Hausdorff dimension concerns the invariant set
$\gL_\pm$ (or, equivalently, the \gol{} set $\go_{f_\pm}(x)$ for any
$x\in\oi$, see \refT{Tir}). 
In the case of a rational rotation number, this set is finite or countably
infinite, see \refT{TRb} below, so it has trivially
Hausdorff  dimension  0. 
We prove that the same holds also in the irrational case,
and prove a sharper result using
the gauge function $h(t)=1/|\log t|$.

\begin{theorem}\label{THa}
The set $\gL_\pm$ has Hausdorff dimension $0$.
Moreover, the Hausdorff measure $\cH_h(\gL_\pm)$ is finite for the gauge
function $h(t)=1/|\log t|$.
\end{theorem}

\begin{proof}
We claim that for each $n\ge0$, $f_\pm^n(\oi)$ is the union of at most $n+1$ 
disjoint closed intervals (possibly of length 0) of total length $a^n$. 
In fact, this is true for
$n=0$. Suppose that it holds for some $n$, with $f_\pm^n(\oi)=\bigcup_{j=1}^{n+1}
I_j$, where some of the intervals $I_j$ may be empty.
Then $\tau$ belongs to at most one interval $I_k=[x_k,y_k]$, and then
$f_\pm^{n+1}(I_k)= f_+([x_k,\tau])\cup f_-([\tau,y_k])$ is the union of two
disjoint closed intervals; all other intervals are mapped to single
intervals.
Since furthermore, $f_\pm$ is injective, and contracts measures by $a$, the
claim follows by induction. 

Hence, $\gL_\pm$ can for each $n$ be covered by $n+1$ intervals of lengths
$a^n$, and thus, since $a^n\to 0$, for any gauge function $h$
\begin{equation}\label{hausdorff}
  \cH_h(\gL_\pm)\le \liminf_{\ntoo} \bigpar{(n+1)h(a^n)}.
\end{equation}
Taking $h(t)=t^\ga$, we find $\cH_\ga(\gL_\pm)=0$ for every $\ga>0$, and
thus the Hausdorff dimension is 0.

Furthermore, taking $h(t)=1/|\log t|$ in \eqref{hausdorff} we obtain
$  \cH_h(\gL_\pm)\le 1/|\log(a)|<\infty$.
\end{proof}
Alternatively, we can argue as in the proof \refT{THb}, using
\eqref{gap} below. 

\smallskip  
Unlike in \refT{THb}, 
we do not know any lower bound in \refT{THa}, 
in the sense of a certain Hausdorff measure being positive.
We state this as an open problem.

\begin{problem}
Find a gauge function $h(t)$  such that $\cH_h(\gL_\pm)>0$,
at least for some $(a,b)$.
\end{problem}

In particular, we do not know whether the gauge function $1/|\log t|$ is
best possible in \refT{THa}.
We suspect that the answer might depend on the parameters;
it seems possible that
$1/|\log t|$ is best possible in \refT{THa}
if, for example, $\rho=1/\sqrt2$ or
$(\sqrt5-1)/2$, but not if $\rho$ is a Liouville number.

Similarly, we  do not know whether the gauge functions in \refT{THb} are best
possible. 
\begin{problem}
Improve, if possible, 
one or both of the gauge functions 
$1/|\log t|^2$ 
and $1/|\log t|$ 
in \refT{THb}.
\end{problem}
Again, it seems possible that the answer depends on $a$.

\section{Rational rotation number}\label{Srational}\noindent
We return to the study of orbits.
We first use the results of \refS{Sloc} to
show that $f_\pm$ has a periodic orbit if and only if the rotation number is
rational, as claimed at the end of \refS{Speriodic}.

\begin{theorem}\label{TR}
  \begin{thmenumerate}
  \item \label{TR1}  
Suppose that the rotation number $\rho=\rho(f_\pm)$ of $f_\pm$ is rational, 
say $\rho=p/q$  
(in lowest terms).
Then $f_\pm$ has a periodic orbit $C$ of length exactly $q$.
Furthermore,
$C=\set{\phirho(k/q):k=0,\dots,q-1}$.
In particular,
\begin{align}
\min C &= \phirho(0)=\psi(\rho),  \label{minC}
\\
\max C &= \phirho((q-1)/q)=\phirho(1-)=1+\psi(\rho+).\label{maxC}
\end{align}
\item \label{TR2}
Conversely, if $f_\pm$ has a periodic orbit, then the rotation number is
rational. 
Moreover,
if the periodic orbit is minimal and has length $q$, then $\rho(f_\pm)$ has
denominator $q$ in lowest terms.
  \end{thmenumerate}
\end{theorem}

\begin{proof}
\pfitemref{TR1}
By \refL{LY} and \eqref{goodpsi}, $\psi(\rho+)\le0\le\psi(\rho)$.
Define $x_k:=\phirho(k/q)$, $k\in\bbZ$, and note that, by \eqref{pr-1},
\begin{equation}\label{xk+q}
  x_{k+q}=\phirho(k/q+1)=x_k+1.
\end{equation}
By \refL{Lpr}\ref{pr-disco}, $x_k<x_{k+1}$.
Furthermore, 
\begin{equation}
  \label{x0}
x_0=\phirho(0)=\psi(\rho)\ge0,
\end{equation}
and, recalling \refL{Lpr}\ref{pr-mono2},
\begin{equation}
  \label{xq-}
x_{q-1}=\phirho((q-1)/q)=\phirho(1-)=1+\phirho(0-)=1+\psi(\rho+)\le1.
\end{equation}

Suppose first that $\psi(\rho+)<0$. Then, recalling \eqref{psi+},
\refL{LX}\ref{LXgood0} applies, and \eqref{rhoF-} holds. Consequently, for
any $k\in\bbZ$,
\begin{equation}
F_-(x_k)=F_-\Bigpar{\phirho\Bigparfrac{k}q}=\phirho\Bigpar{\frac{k}q+\rho}
=\phirho\Bigpar{\frac{k}q+\frac{p}q}
=x_{k+p}  .
\end{equation}
This implies, by \refL{LF}\ref{LFpi},
$f_-(\frax{x_k})=\frax{F_-(x_k)}=\frax{x_{k+p}}$, and thus by iteration
$f_-^n(\frax{x_k})=\frax{x_{k+np}}$ for any $n\ge0$.
Taking $n=q$ we find, using \eqref{xk+q},
$f_-^q(\frax{x_k})=\frax{x_{k}+p}=\frax{x_k}$, so $\frax{x_k}$ lies in a
periodic orbit $C$ of $f_-$. Moreover, it is easy to see that
\begin{equation}\label{C=}
C=\bigset{\frax{x_k}}_{k\in\bbZ}=\bigset{\frax{x_k}}_{k=0}^{q-1}
=\set{x_k}_{k=0}^{q-1},
\end{equation}
using the fact that $x_k\in\oio$ for $0\le k\le q-1$ by 
\eqref{x0}--\eqref{xq-}. We thus have $\min C = x_0$ and $\max C = x_{q-1}$;
hence \eqref{x0}--\eqref{xq-} yield \eqref{minC}--\eqref{maxC}.

If $\psi(\rho+)=0$, then necessarily $\psi(\rho)=\phirho(0)>0$, see
\eqref{jump}. 
In this case,
\refL{LX}\ref{LXgood1} applies, and \eqref{rhoF+} holds.
By \refL{Lpr}, $\phirho$ is constant on the interval
$[\frac{k}q,\frac{k+1}q)$, and thus, using \eqref{rhoF+},
\begin{equation}
  \begin{split}
F_+(x_k)
&=F_+\Bigpar{\phirho\Bigpar{\frac{k}q}}    
=F_+\Bigpar{\phirho\Bigpar{\frac{k+1}q-}}    
=\phirho\Bigpar{\Bigpar{\frac{k+1}q+\rho}-}    
\\&
=\phirho\Bigpar{{\frac{k+1+p}q}-}    
=\phirho\Bigpar{{\frac{k+p}q}} 
=x_{k+p}      . 
  \end{split}
\raisetag{\baselineskip}
  \end{equation}
We can now repeat the arguments above, using $f_+$, $F_+$ and $\fraxx\cdot$
instead of $f_-$, $F_-$  and $\frax\cdot$; this shows that 
$C=\set{x_k}_{k=0}^{q-1}$
now is a periodic
orbit for $f_+$. Note that in the present case, $C\subset\ooi$.

\pfitemref{TR2} 
Suppose that $f_\pm$ has a periodic orbit. By \refL{LA}, 
either $f_-$ or $f_+$ has a periodic orbit; let us assume that $f_-$ has
one.
Then, for some $x\in\oio$ and some $q\ge1$, $f_-^q(x)=x$, which by \refL{LF}
implies $F_-^q(x)=x+p$ for some integer $p$. Consequently, $F_-^{nq}(x)=x+np$
for every $n\ge0$, and thus $F_-^{nq}(x)/n\to p/q$; hence the rotation number
is $p/q$.

If $q$ is minimal, then $p$ and $q$ are coprime, as a consequence of
\ref{TR1} and \refC{C1} (or by a simple direct argument which we omit). 
\end{proof}

By \refT{T1}, $f_\pm$ has a universal limit cycle.
Combining these results, we obtain the following.

\begin{theorem}\label{TRb}
  Suppose that $a\in\ooio$ and $\rho\in\oio$ with $\rho$ rational.
Then $f_\pm$ has rotation number $\rho(f_\pm)=\rho$ 
if and only if one of the following
three cases holds. 
\begin{romenumerate}
  \item \label{TRb-}
$b=b_-(a,\rho)$.
Then $f_\pm$ has a unique periodic orbit $C$, with $0\in C$ but $1\notin C$.
$C$ is also a periodic orbit of $f_-$, but $f_+$ has no periodic orbit.

Furthermore, $\gL_-=C$, while $\gL_+=\emptyset$ and $\gL_\pm =
C\cup O_1$, where $O_1$ is the orbit of $1$.

  \item \label{TRb0}
$b_-(a,\rho)<b<b_+(a,\rho)$.
Then $f_\pm$ has a unique periodic orbit $C$, with $0,1\notin C$.
Furthermore, $\gL_\pm=\gL_+=\gL_-=C$.

\item \label{TRb+}
$b=b_+(a,\rho)$.
As in \ref{TRb-}, interchanging $0$ and $1$ and indices $+$ and $-$.
  \end{romenumerate}
In all three cases,
every orbit of $f_\pm$ converges to $C$, so 
$\go_{f_\pm}(x)=\go_{f_-}(x)=\go_{f_+}(x)=C$ for every
$x\in\oi$.
\end{theorem}

\begin{proof}
  The rotation number $\rho(f_\pm)$  equals $\rho$ if and only if 
$b_-(a,\rho)\le b\le b_+(a,\rho)$ 
by \refT{TLb}. In this case, $f_\pm$ has a periodic orbit $C$ by
  \refT{TR}. Furthermore, $C$ is unique by \refC{C1}, and
by \eqref{minC}--\eqref{maxC} and \eqref{wb-}--\eqref{wb+},
$0\in C\iff \psi(\rho)=0\iff b=b_-(a,\rho)$
and
$1\in C\iff \psi(\rho+)=0\iff b=b_+(a,\rho)$.
Hence, $\tau\in C$ if and only if $b=b_-(a,\rho)$ or $b=b_+(a,\rho)$.
In other words, we are in Case 1a in \refS{Sorbits+} in \ref{TRb0},
and in Case 1b in \ref{TRb-} and \ref{TRb+}.
\end{proof}

\begin{remark}\label{RRb}
\refT{TRb} shows that 
if $\rho(f_\pm)$ is rational, then
$\go_{f_\pm}(x)\subseteq \gL_\pm$ for all $x$,
with equality in Case \ref{TRb0}, but strict inclusion in \ref{TRb-} and
\ref{TRb+}.
 
In contrast, we have $\go_{f_-}(x)\supseteq \gL_-$ for all $x$,
with equality in Cases \ref{TRb-} and \ref{TRb0}, but strict inclusion in 
\ref{TRb+}, when $\gL_-=\emptyset$, and similarly for $f_+$.
\end{remark}

\begin{theorem}\label{TRbc}
  If the dynamical system $f_\pm$ has a rational rotation number, then
$f_\pm$ has a universal limit cycle $C$.
Thus every orbit of $f_\pm$ converges to $C$.
Furthermore, the symbolic sequence of every orbit is eventually periodic.
\end{theorem}

\begin{proof}
  The first statement follows from \refT{TRb}, and it implies the second by
  definition.
Thus, again by the definitions, if $(x_n)_0^\infty$ is any orbit, there
exists a periodic orbit $(y_n)_0^\infty$ (started at a suitable point
$y_0\in C$) such that $x_n-y_n\to0$ as $\ntoo$.
By \eqref{eps2} this implies, with obvious notation,
$\eps_n^x-\eps_n^y\to0$,
and thus $\eps_n^x=\eps_n^y$ for all large $n$ since
$\eps_n^x,\eps_n^y\in\setoi$. 
Consequently, the symbolic sequence for the 
  orbit $(x_n)_0^\infty$ equals from some point on the symbolic sequence for
  $(y_n)_0^\infty$, which is periodic.
\end{proof}

\begin{example} \label{E00}  
By \refT{TLb}, or \refL{LY} and \eqref{psi0}--\eqref{psi0+}, 
the rotation number is 0 if and only if $0\le
b\le 1-a$, \ie, if and only if $a+b\le 1$.
This is the simple case studied already in Examples \ref{E<1}, \ref{E1} and
\ref{E0}. We see from Theorems \ref{TR} and \ref{TRb}, or directly as in
these examples, that in this case (and only in this case) there is a fixed
point, \ie, a periodic cycle of length 1, and that every orbit converges to
the fixed point. 
The cases $b=0$ and $b=1-a$ discussed in Examples \ref{E0} and \ref{E1}
are the cases \ref{TRb-} and \ref{TRb+} in \refT{TRb}.

\refT{Teps} shows that when $\rho=0$, at most a finite number of the symbols
$\eps_i$ are non-zero. In fact, it is easy to see that there can be at most
one non-zero symbol.
\end{example}

\subsection{A sufficient condition for a rational rotation number}
By Theorems \ref{THb} and \ref{THab}, 
or  by the earlier  results by 
\citet{conze} and  \citet{laurent-nogueira}
discussed in \refS{Shaus},
the rotation number is rational for `most'
values of the parameters $(a,b)$.
Explicit examples with a rational rotation number can easily be produced
using \refT{TLb}.
Another large class of parameter values with a rational rotation number is
given by 
the following 
theorem by
Laurent and Nogueira
\cite[Theorem 3]{laurent-nogueira}, 
which we quote for later reference;
their proof is based on a number theoretic result by 
\citet[Theorem 7]{lox}, 
combined with results by \cite{conze}
(our \eqref{b-}--\eqref{b+} and \refT{TLb}).

\begin{theorem}[\citet{laurent-nogueira}]\label{Talgebraic}
  If $a$ and $b$ are algebraic numbers, then 
the dynamical system $f_\pm$  has a rational rotation number.
\qed
\end{theorem}

\section{Irrational rotation number}\label{Sirrational}\noindent
We now consider the case when $f_\pm$ has an irrational rotation number
$\rho=\rho(f_\pm)$. 
By \refT{TR}\ref{TR2}, $f_\pm$ has no periodic orbit.
Hence, this is Case 2 in \refS{Sorbits+}; we proceed to verify the claims
there. 

By \refL{LY} and \eqref{jump}, $\phirho(0)=\psi(\rho)=0$,
and thus, see \eqref{pr-1}, $\phirho(1)=1$.
Moreover, $\phirho$ is strictly increasing, by \refL{Lpr}, and thus 
$\phirho$ gives a bijection of $\oio$ onto $\glo:=\phirho(\oio)\subset\oio$.

It follows from \eqref{rhof-} that $f_-(\glo)=\glo$, and that $f_-$
restricted to $\glo$ is a bijection, which is conjugated by $\phirho$ to the
rotation $x\mapsto \frax{x+\rho}$ on $\oio$. 

By \refL{Lpr}\ref{pr-disco}, 
the set of discontinuities of $\phirho$ in $\oi$ is 
\begin{equation}
\Drho\cap\oi=\bigset{ \frax{m\rho}:m\ge1}.
\end{equation}
This set is countably infinite, and dense in \oi; note also that
$0,1\notin\Drho$. 
Let $x_i:=\frax{i\rho}$, so $\Drho\cap\oi=\set{x_i}_1^\infty$,
and let $\xi_i:=\phirho(x_i-)$ and $\eta_i:=\phirho(x_i)$.
Since $\phirho$ is strictly increasing and right-continuous (\refL{Lpr}), 
it follows that
\begin{equation}\label{glo}
  \glo=\phirho(\oio)=\oio\setminus\bigcup_{i=1}^\infty [\xi_i,\eta_i)
\end{equation}
and
\begin{equation}\label{bglo}
  \bglo=\oi\setminus\bigcup_{i=1}^\infty (\xi_i,\eta_i)
= \set{\phirho(x),\phirho(x-):x\in\oi}. 
\end{equation}
It follows from \eqref{phirho} that the gap $(\xi_i,\eta_i)$ has length
\begin{equation}\label{gap}
  \eta_i-\xi_i=(1-a)a^{i-1},
\qquad i=1,2,\dots
\end{equation}
Hence, the sum of the lengths of the gaps is 1, so $\bglo$ has Lebesgue
measure 0. In fact, it has Hausdorff dimension 0, see \refT{THa}.

Note also that \eqref{rhof-} implies $f_-(\phirho(1-\rho))=0$, and thus
$\tau=\phirho(1-\rho)$. In particular, $\tau\in \glo$; furthermore,
$\tau\neq\eta_i$ for $i\ge1$, and consequently, $\tau\notin[\xi_i,\eta_i]$.
Since $f_-(\eta_i)=\eta_{i+1}$, 
by  \eqref{rhof-} again, 
it follows
that for every $i\ge1$, 
$f_-$ maps $[\xi_i,\eta_i]$ linearly onto $[\xi_{i+1},\eta_{i+1}]$;
furthermore, $f_\pm=f_+=f_-$ on each such interval.
Finally, \eqref{rhof-} and \eqref{rhof+} 
(with $x=0$)
imply 
\begin{equation}\label{f01}
f_\pm(0)=\eta_1\qquad\text{and}\qquad f_\pm(1)=\xi_1. 
\end{equation}
This describes the dynamics of $f_\pm$ on $\oi\setminus\glo$ completely.
It follows easily, by induction, that
\begin{align}
  f_-^n(\oio)&=\oio\setminus\bigcup_{i=1}^n[\xi_i,\eta_i),\label{fn-}\\
  f_+^n(\ooi)&=\ooi\setminus\bigcup_{i=1}^n(\xi_i,\eta_i],\label{fn+}\\
  f_\pm^n(\oi)&=\oi\setminus\bigcup_{i=1}^n(\xi_i,\eta_i).\label{fn+-}
\end{align}

\begin{remark}
As shown above, $\tau\in\gL_\pm$, and thus also $0,1\in\gL_\pm$ whenever
$\rho(f_\pm)$ is irrational, see \eqref{01tau}.
\end{remark}

\begin{theorem}\label{Tir}
Suppose that $f_\pm$ has an irrational rotation number
$\rho=\rho(f_\pm)$. 
Then
\begin{align}
\gL_\pm&=\bglo= \set{\phirho(x),\phirho(x-):x\in\oi}, \label{tirgl+-}\\
\gL_-&=\glo= \set{\phirho(x):x\in\oio}, \label{tirgl-}\\
\gL_+&=\gli:= \set{\phirho(x-):x\in\ooi}
=\bglo\setminus\set{0,\eta_1,\eta_2,\dots}. \label{tirgl+}
\end{align}
Furthermore, the limit sets 
$\go_{f_\pm}(x)=\go_{f_-}(x)=\go_{f_+}(x)=\gL_\pm$ for every $x\in\oi$.

For any orbit $(x_n)\ooz$, the distance $d(x_n,\gL_\pm)\le a^n$ for every 
$n\ge0$;
hence the orbits converge to $\gL_\pm$ uniformly (and geometrically).
\end{theorem}

\begin{proof}
First, \eqref{tirgl+-}--\eqref{tirgl+} follow from \eqref{fn-}--\eqref{fn+-}
and \eqref{glo}--\eqref{bglo}.

For the limit sets, consider  first $f_-$. 
Suppose first that $x\in\glo$. Then $x=\phirho(t)$ for some $t\in\oio$, and 
thus $f_-^n(x)=f_-^n(\phirho(t))=\phirho(\frax{t+n\rho})\in\glo$.
Hence, $\go_{f_-}(x)\subseteq\bglo$. On the other hand, for any
$y=\phirho(u)\in\glo$, there exists a subsequence $\xpar{n_k}$ such that 
$t_{n_k}:=\frax{t+n_k\rho}\to u$ with $t_{n_k}\ge u$; since $\phirho$ is
right-continuous, this implies $f_-^{n_k}(x)\to\phirho(u)=y$.
Hence, $\go_{f_-}(x)\supseteq\glo$. 
Since $\go_{f_-}(x)$ is closed by \eqref{omegaf}, this implies
$\go_{f_-}(x)\supseteq\bglo$, and thus
$\go_{f_-}(x)=\bglo=\gL_\pm$.

On the other hand, if $x\in\oio\setminus\glo$, 
then $x\in[\xi_i,\eta_i)$ for some $i$.
Since $f_-$ is a linear contraction on each interval $[\xi_i,\eta_i]$, it
follows that 
\begin{equation}\label{vesp}
f_-^n(\eta_i)-f_-^n(x)=a^n(\eta_i-x)\to0   
\end{equation}
as \ntoo; hence the
orbit of $x$ is asymptotic to the orbit of $\eta_i\in\glo$, and thus 
$\go_{f_-}(x)=\go_{f_-}(\eta_i)=\bglo=\gL_\pm$ in this case too.

Finally, for $x=1$, recall from \eqref{f01} that $f_-(1)=\xi_1\in\oio$. Thus
$\go_{f_-}(1)=\go_{f_-}(\xi_1)=\gL_\pm$.
Hence $\go_{f_-}(x)=\gL_\pm$ for every $x\in\oi$.

By symmetry (\refR{Rreflection}), also $\go_{f_+}(x)=\gL_\pm$ for every
$x\in\oi$.

The description of the orbits in the beginning of \refS{Sorbits+}
shows that every orbit for $f_\pm$ is an orbit for $f_-$ or for $f_+$.
Hence, for any $x\in\oi$, $\go_{f_\pm}(x)=\go_{f_-}(x)\cup\go_{f_+}(x)=\gL_\pm$.

Now, let $(x_n)\ooz$ be an arbitrary orbit. If $x_0\in\gL_\pm$, then 
$x_n\in\gL_\pm$ for every $n$, and thus $d(x_n,\gL_\pm)=0$.
On the other hand, if $x_0\in\oi\setminus\gL_\pm\subset \oio\setminus\gL_0$,
then for every $n\ge1$, \eqref{vesp} implies
$d(x_n,\gL_\pm)\le d(x_n,f_-^n(\eta_i))\le a^n$.
\end{proof}

\begin{remark}
  \label{Rir}
In particular, if $\rho(f_\pm)$ is irrational, then, 
for any $x$,
$\go_{f_\pm}(x)= \gL_\pm$, 
while 
$\go_{f_-}(x)\supsetneq \gL_-$
and $\go_{f_+}(x)\supsetneq \gL_+$.
Cf.\ the case of a rational rotation number in \refR{RRb}.
\end{remark}

\begin{remark}
  It is easy to see that when $\rho(f_\pm)$ is irrational, $\gL_\pm$ is a
  Cantor set, i.e., a totally disconnected perfect compact set 
(and thus homeomorphic to the Cantor cube $\set{0,1}^\infty$).
In fact, $\gL_\pm$ is compact and non-empty, and totally disconnected since
it has measure 0 and thus does not contain any open interval.
Finally, if $x\in\gL_\pm$, then $x\in\go_{f_\pm}(x)$ by \refT{Tir}, 
so there exists an
orbit $(x_n)$ with $x_0=x$ and a subsequence $x_{n_k}\to x$. Then each
$x_n\in\gL_\pm$ since $\gL_\pm$ is invariant, and $x_n\neq x$ for $n\ge1$
since there is
no periodic orbit; hence $x$ is not isolated in $\gL_\pm$.
\end{remark}

\begin{remark}
  When $\rho$ is irrational, as shown above,
$0,1,\tau\in\gL_\pm=\go_{f_\pm}(x)$ for
any $x$. Hence, since each $x$ has at most two orbits, any orbit comes
arbitrarily close to the discontinuity point
$\tau$ (on both sides), as well as to $0$ and $1$,
infinitely often. 
\end{remark}

\section{The invariant measure}\label{Sinvariant}\noindent
If $\rho(f_\pm)$ is rational, so there exists a periodic orbit $C$ by
\refT{TR}, then there is an obvious invariant probability measure $\mu$
on $C$,
\viz{} the uniform measure with mass $1/|C|$ at each point.
This measure $\mu$ 
is invariant under $f_\pm$ in the sense that if $1\notin C$ it is
invariant under $f_-$ and if $0\notin C$ then it is invariant under $f_+$;
recall that at least one of these cases occurs, see \refT{TRb}.

Suppose now that $\rho(f_\pm)$ is irrational.
Then we construct an invariant probability measure $\mu$ as 
the image measure of the Lebesgue measure on $\oi$ under 
the map $\phi_\rho$, where $\rho:=\rho(f_\pm)$.
Then $\phi_\rho:\oi\to\gL_\pm$, see \eqref{tirgl+-}, and thus $\mu$ is a
probability measure on $\gL_\pm$. Since $\phi_\rho$ is strictly increasing
by \refL{Lpr}, $\mu$ is in this case a continuous measure, i.e., each point
has measure 0.
Moreover,
\eqref{good} holds by \refL{LY}, 
so \refL{LX} applies, and it follows from \eqref{rhof-}
that $\mu$ is invariant under $f_-$;
$\mu$ is invariant under $f_+$ too since $\mu$ has
no point mass at $\tau$.

\begin{theorem}\label{Tmu}
Let $(x_i)\ooz$ be an arbitrary orbit of $f_\pm$.
Then the empirical measure $\frac{1}n\sum_{i=0}^{n-1}\gd_{x_i}$ converges
weakly to the invariant $\mu$ as $\ntoo$.
\end{theorem}

\begin{proof}
If $\rho=\rho(f_\pm)$ is rational, this follows from the fact that the orbit
converges 
to the limit cycle $C$, see \refT{TRbc}.

Thus suppose that $\rho$ is irrational. 
Then the orbit visits 1 at most once, and if it does, it suffices to
consider the part of the orbit after 1. Hence, we may assume that
$x_0\in\oio$ and that $x_n=f_-^n(x_0)$.

If $x_0\in \gL_0$, so $x_0=\phi_\rho(t)$ for some $t\in\oio$ (see
\eqref{tirgl-}), then \eqref{rhof-} implies
$x_i=\phi_\rho(\frax{t+i\rho})$, and hence $\mu_n:=\frac{1}n\sum_0^{n-1}\gd_{x_i}$
is the image under $\phi_\rho$ of the measure
$\nu_n:=\frac{1}n\sum_0^{n-1}\gd_{\frax{t+i\rho}}$.
As \ntoo, the measures $\nu_n$ converge weakly to the uniform measure $\gl$
on $\oio$, and since $\phi_\rho$ is measurable and $\gl$-a.e.\ continuous
(by \refL{Lpr}), it follows that $\mu_n\to\mu$ weakly, see
\cite[Theorem 5.1]{Billingsley}. 

If $x_0\in \oio\setminus\gL_0$, 
then there exists as in the proof of \refT{Tir} an $\eta_i\in\gL_0$ such
that \eqref{vesp} holds. We have just shown that
the theorem holds for the orbit starting at $\eta_i$, 
and then \eqref{vesp} implies that the same holds
for the orbit starting at $x_0$.
\end{proof}

\begin{corollary}
  The invariant measure $\mu$ has center of mass
$\intoi x\dd\mu = \chi:=(b-\rho(f_\pm))/(1-a)$.
\end{corollary}

\begin{proof}
With $\mu_n$ as in the proof of \refT{Tmu},
$\intoi x\dd\mu_n\to\intoi x\dd\mu$ by \refT{Tmu}, and
$\intoi x\dd\mu_n\to\chi$ by \refT{Tmean}.
\end{proof}

\begin{theorem}\label{Tnu}
The measure $\mu$ is the only probability
measure on $\oi$ that is invariant under $f_-$ or $f_+$.
\end{theorem}

\begin{proof}
  Suppose that $\nu$ is such a probability measure, invariant under, say,
  $f_-$.
Let $X_0$ be a random point in $\oi$ with the distribution $\nu$, and let
$X_n:=f_-^n(X_0)$. Then $X_n$ is a sequence of random variables, each having
the same distribution $\nu$.

Let $h\in C\oi$ be an arbitrary continuous function on $\oi$.
Then \refT{Tmu} shows that
\begin{equation}
  \frac1n\sum_{i=0}^{n-1} h(X_i) \to \int h\dd\mu.
\end{equation}
The random variables on the left-hand side are uniformly bounded, so by
dominated convergence,
\begin{equation}
\frac1n\sum_{i=0}^{n-1}\E h(X_i)=
  \E\Bigpar{ \frac1n\sum_{i=0}^{n-1} h(X_i)} \to \int h\dd\mu.
\end{equation}
On the other hand, each $X_i$ has distribution $\nu$, so $\E h(X_i)=\int
h\dd\nu$. Consequently, $\int h\dd\nu=\int h\dd\mu$, which, since $h$ is
arbitrary, means $\nu=\mu$.
\end{proof}

\section{Phragm\'en's election method}\label{Sphr}\noindent

\subsection{Definition of Phragm\'en's method}\label{SSphr-def}

\phragmen's election method can be described in several different, but
equivalent, ways. 
For our purposes it is convenient to use the following, which is based on
\phragmen's original formulation (in French) in \cite{Phragmen1894}; 
see also
\cite{Phragmen1895,Phragmen1896,Phragmen1899}, \cite{SJV9} 
and \refSS{SSphr-algo} below
for different
formulations and motivations.

\begin{metod}{\phragmen's election method} 
Assume that each ballot has some \emph{voting power} $t$; this number is
the same for all ballots and will be determined later. 
A candidate needs total voting power $1$ in order to be elected.
The voting power of a
ballot may be used by the candidates on that ballot, and it may
be divided among several of the candidates on the ballot. 
During the
procedure described below, some of the voting power of a ballot may be already
assigned to already elected candidates; the remaining voting power of the
ballot
is free.

The seats are distributed one by one.

For each seat, each remaining candidate
may use all the free voting power of each ballot that includes
the candidate. (I.e., the full voting power $t$ except for the voting power
already assigned from that ballot to candidates already elected.)
The ballot voting power $t$ 
that would give the candidate voting power $1$ is
computed, and the candidate requiring the smallest voting power $t$ is
elected. All free (\ie, unassigned) 
voting power on the ballots that contain the elected
candidate is assigned to that candidate, and these assignments remain fixed
throughout the election. 

The computations are then repeated for the next seat 
for the remaining candidates
(resulting in a new voting power $t$), 
and so on.
\end{metod}

Ties are broken by lot or by some other method. The required voting power
$t$ increases for each seat, except in some cases of a tie where $t$ may
remain the same.

\subsection{An algorithmic version of \phragmen's method}
\label{SSphr-algo}

For any set $\gs$ of candidates (parties in the party version),
let $v_\gs$ be the number of votes for the set $\gs$.
Hence the total number of votes for  candidate (party) $i$ is
\begin{equation}\label{WWI}
  \WW_i:=\sum_{\gs\ni i} v_\gs.
\end{equation}

\phragmen's method is often formulated in the following algoritmical form,
where  $\WW_i$ is reduced to a \emph{reduced vote} $W_i$ when some
candidates on  ballots containing $i$ already have been elected:

For each set $\gs$ with $v_\gs>0$ (i.e., each group of identical ballots),
we assign dynamically a 
\emph{place number} $q_\gs$, which is a real non-negative number that can be
interpreted as the (fractional) number of seats elected so far by these
ballot; the sum of the place numbers is always equal to the number of seats
already allocated. The place numbers are assigned and the seats
are allocated recursively  by the following rules. 
\begin{romenumerate}
\item 
Initially all place numbers $q_i=0$.  
\item \label{phr-wi}
The reduced vote for candidate $i$ is defined as
\begin{equation}\label{wi}
  W_i:=\frac{\sum_{\sigma\ni i} v_\sigma}{1+\sum_{\sigma\ni i} q_\sigma},
\end{equation}
\ie, the total number of votes for the candidate divided by 1 + their total
place number.
\item 
The candidate $i$ with the largest $W_i$ is elected to the next seat, breaking
ties by lot or some other method.
(In the original version, only unelected candidates are considered.
In the party version, repetitions are allowed.)
\item \label{phr-update}
If $i$ is elected, then $q_\gs$ is updated for every $\gs\ni i$ (i.e., for
the ballots that contributed to the election of $i$); the new value is
\begin{equation}
  q_\gs':=\frac{v_\gs}{W_i}.
\end{equation}
$q_\gs$ remains unchanged when $\gs\not\ni i$.\\
Repeat from \ref{phr-wi}.
\end{romenumerate}
It is easily verified from \eqref{wi} that \ref{phr-update} 
increases $\sum_\gs q_\gs$
by 1, so by induction, $\sum_\gs q_\gs$  
equals the number of elected, as claimed above.

For a proof that this really yields the same result as the definition in
\refSS{SSphr-def}, see e.g.\ \cite{SJV9}; we remark here only that the
connection is that the voting power $t$ required to elect candidate $i$ in
the previous version equals $1/W_i$ with $W_i$ given by \eqref{wi}, and that
$q_\gs$ is the total voting power already assigned to previously elected on
the ballots of type $\gs$.

\subsection{\phragmen's method as a dynamical system}
\label{SSphr-dynamical}

\phragmen's method (in the party version)
can be regarded as a dynamical system as follows.

Let $\cP$ be the set of parties (or candidates, in the original version),
and
let as above $v_\gs$ be the number of votes for the set $\gs$ of parties.
(We regard these numbers as fixed.)
Define $\WW_i$ by \eqref{WWI}.
We may ignore parties that do not appear on any ballot,
and thus we assume that $\WW_i>0$ for every $i\in\cP$.
Let 
\begin{equation}\label{PP}
\PP:=\set{\gs\subseteq\cP:v_\gs>0 \text{ and } \gs\neq\emptyset},  
\end{equation}
the family of all nonempty sets of parties with at least one vote for the set.
(I.e., the different types of ballots that occur.
We ignore blank votes, i.e., $\gs=\emptyset$, since they do not affect the
outcome.) 

We use the formulation of \phragmen's method in \refSS{SSphr-def}, and let
$\free_\gs=\free_\gs(n)$ be the free voting power of each ballot $\gs$
when $n$ candidates have been elected.
Let $\vfree=\vfree(n)=(\free_\gs)_{\gs\in\PP}$ be the vector of free voting powers.
Let $\vett:=(1)_{\gs\in\PP}$ be the vector with all components 1.
The description in \refSS{SSphr-def} now can be formalized as follows:
\begin{romenumerate}

\item \label{PHRinitiate}
Initialize all $\free_\gs:=0$.

\item \label{PHRvp}
A party (candidate) $i$ can use a voting power
\begin{equation}\label{vix}
  V_i(\vfree)
=V_i((\free_\gs)_\gs)
:=\sum_{\gs\ni i}v_\gs \free_\gs.
\end{equation}
For each $i\in\cP$, find $\gD_i:=\gD_i(\vfree)$ such that 
$V_i(\vfree+\gD_i\vett)=1$, \ie,
\begin{equation}\label{gD_i}
\sum_{\gs\ni i}v_\gs (\free_\gs+\gD_i)
=1.  
\end{equation}

\item \label{PHRmin}
Find $\ix$ such that $\gD_{\ix}$ is minimal, \ie,
$\gD_{\ix}=\min_{i\in\cP} \gD_i$.\\
Output $\ix$ as the next elected.

\item \label{PHRupdate}
Update $\vfree$ to
\begin{equation}\label{phrupdate}
\free_\gs':=
  \begin{cases}
    \free_\gs+\gD_{\ix}, & \ix\notin\gs, \\
    0, & \ix\in\gs.
  \end{cases}
\end{equation}
Repeat from \ref{PHRvp}.
\end{romenumerate}

In the original version, candidates that are elected are not considered
further, but in the party version there is no such restriction.

We can regard \ref{PHRvp}--\ref{PHRupdate} as a function $f$,  
taking a vector $\vfree$ to a new vector $f(\vfree)=(\free_\gs')_\gs$; 
a natural state space is
\begin{equation}\label{K}
K:=
\bigset{\vfree=(\free_\gs)_\gs\in\ooo^\PP:V_i(\vfree)\le1 \; \forall  i\in \cP}.
\end{equation}
If $\vfree\in K$ and
$\gs\in\PP$, take any $i\in\gs$; then
$V_i(\vfree)\le1$ and thus  $x_\gs\le 1/v_\gs<\infty$ by \eqref{vix}.
Consequently,
$K$ is closed and bounded,
i.e.,  $K$ is a compact
subset of $\bbR^\PP$.
Note that the equation \eqref{gD_i} is a linear equation in $\gD_i$, with
positive coefficient $\WW_i$; thus the equation has a unique solution
$\gD_i(\vfree)$. Moreover, $\gD_i(\vfree)\ge0$ for $\vfree\in K$.

Ties are possible in \ref{PHRmin}; in that case we choose $\ix$ by lot or by
some other method. We regard the method as indeterminate in that case.
We formalize this by defining, for $i\in \cP$, 
\begin{equation}\label{Ki}
  K_i:=\bigset{\vfree\in K: \gD_i(\vfree)\le \gD_j(\vfree)\; \forall j\in \cP},
\end{equation}
\ie, the set of free voting powers where $i$ can be chosen as $\ix$.
Then \ref{PHRupdate} (with $\ix=i$) defines a function $f_i:K_i\to K$, and
$f$ is the union of these functions.
Note that $K=\bigcup_i K_i$, so $f$ is defined everywhere on $K$, but $f$
is multivalued at points in the intersection $K_i\cap K_j$ of two (or more)
domains. (Cf.\ \cite{cat}, where multivalued functions of this type are
studied in the case when each $f_i$ is a contraction.)

Note that the result is the same if all vote numbers $v_\gs$ are multiplied
by the same positive constant. We may thus divide by the total number of
votes and thus replace the numbers of votes by their proportions; we keep
the notation $v_\gs$ but may thus without loss of generality assume
$\sum_\gs v_\gs=1$. Moreover, we allow $v_\gs$ to  be arbitrary real numbers
in $\oi$ (with sum 1). (In a real election, the proportions are of course
rational numbers, but we may imagine that we have weighted votes, where
voters have different weigths that are arbitrary positive real numbers.)

The general case seems quite difficult to analyse, 
so we consider in the sequel the case of only two parties.

\begin{remark}
  The dynamical system just described is in general not locally contractive
for the standard Euclidean metric on $K\subset\ooo^\PP$
(or for the $\ell^1$ or $\ell^\infty$ metric, say), 
  not even for two parties; see \eqref{os2} below for a counterexample.
\end{remark}

\subsection{\phragmen's method for two parties}

With two parties $A$ and $B$, the possible votes are $A$, $B$ and $AB$
(and blank votes, but they  may be ignored as said above).
For convenience, we may assume as above that $v_\gs$ is the proportion of
votes on $\gs$, and thus that they sum to 1; furthermore
we change notation and denote these proportions by
$\ga:=v_A$, 
$\gb:=v_B$ and $\gab:=v_{AB}=1-\ga-\gb$.

By symmetry, we may assume $\ga\ge\gb\ge0$.
The cases  $\gb=0$ and $\ga=\gb$  are simple, see Examples \ref{Egb=0} and
\ref{E==}. 
We may thus assume $\ga>\gb>0$. 
We shall show that it then is possible to transform the
dynamical system in \refSS{SSphr-dynamical}
into the system $f_\pm=\set{\frax{ax+b},\fraxx{ax+b}}$
studied above, for some $a$ and $b$.

We do the transformation in several steps.
First, note that we do not use all of the set $K$ in \eqref{K}.
In fact, when $A$ is elected we put $x_A=x_{AB}=0$, and when $B$ is elected
we put $x_B=x_{AB}=0$. Hence, both $f_A$ and $f_B$ map $K$ into the subset,
with $\vfree=(x_A,x_B,x_{AB})$,
\begin{equation}
K':= K\cap\bigpar{\bigset{(x,0,0):x\ge0}
\cup
\bigset{(0,y,0):y\ge0}}
%
\end{equation}
and thus it suffices to consider the action of $f_A$ and $f_B$ on $K'$.

There are thus two cases:

\begin{xenumerate}
\item 
Suppose that
$\vfree=(x,0,0)$.
If the voting power of each ballot is increased by $\gD$, 
then $A$ has available voting power,
cf.\ \eqref{vix}--\eqref{gD_i},  
\begin{equation}
V_A(\vfree+\gD\vett)
=v_A(x+\gD) + v_{AB} \gD=(\ga+\gab)\gD+\ga x=(1-\gb)\gD+\ga x, 
\end{equation}
and thus
$A$ requires additional voting power
\begin{equation}
\gD_A = \frac{1-\ga x}{1-\gb}.  
\end{equation}

On the other hand, $B$ has available voting power 
\begin{equation}
  V_B(\vfree+\gD\vett)=
v_B\gD+v_{AB}\gD=(\gb+\gab)\gD=(1-\ga)\gD,
\end{equation}
so $B$ requires voting power
\begin{equation}
  \gD_B=
 \frac{1}{1-\ga}.  
\end{equation}

Since $\ga>\gb$ by assumption, $\gD_B>1/(1-\gb)\ge \gD_A$; hence the next
seat goes to $A$, updating $(x,0,0)$ to $(x',y',z')$ with $x'=z'=0$ and
\begin{equation}
  y'=\gD_A=\frac{1-\ga x}{1-\gb}.
\end{equation}

\item 
Suppose that
$\vfree=(0,y,0)$.
Arguing as above, we find that the additional voting power  required for the
two parties are 
\begin{align}
\gD_A &= \frac{1}{\ga+\gab}=\frac{1}{1-\gb},
\\
\gD_B &= \frac{1-\gb y}{\gb+\gab}= \frac{1-\gb y}{1-\ga}.  
\end{align}
Thus, there are two subcases: (In case of equality in \eqref{iia} and
\eqref{iib}, we are in the indeterminate case when both alternatives are
possible; the same applies to all transformations below.) 
\begin{alphenumerate}
\item 
$A$ is elected if
\begin{equation}\label{iia}
\frac{1}{1-\gb}\le  \frac{1-\gb y}{1-\ga},
\end{equation}
or, equivalently,
\begin{equation}
  \gb y \le 1-\frac{1-\ga}{1-\gb} = \frac{\ga-\gb}{1-\gb}.
\end{equation}
The free voting powers are updated to $(0,y',0)$ where 
\begin{equation}
  y':=y+\gD_A=y+\frac{1}{1-\gb}.
\end{equation}

\item 
$B$ is elected if
\begin{equation}\label{iib}
\frac{1}{1-\gb}\ge  \frac{1-\gb y}{1-\ga},
\end{equation}
or, equivalently,
\begin{equation}
  \gb y \ge 1-\frac{1-\ga}{1-\gb} = \frac{\ga-\gb}{1-\gb}.
\end{equation}
The free voting powers are updated to $(x',0,0)$ with
\begin{equation}
  x':=\gD_B= \frac{1-\gb y}{1-\ga}.
\end{equation}
\end{alphenumerate}
\end{xenumerate}

\subsubsection{First dynamical system}
Since $x_{AB}=0$ on $K'$, we may ignore $x_{AB}$ and write the elements of
$K'$ as $(x_A,x_B)$.
\phragmen's method can thus be formulated as a dynamical system, 
operating on vectors
$(x,y)\in(\ooo\times\set0)\cup(\set0\times\ooo)$
by the function $(x,y)\mapsto \fx_1(x,y)$ given by
\begin{romenumerate}
\item \label{ios1}
If $y=0$, then output $A$ and let
\begin{equation}\label{os1}
  \fx_1(x,0):=\Bigpar{0,\frac{1-\ga x}{1-\gb}}.
\end{equation}
\item[(iia)] 
If $x=0$ and $\gb y \le \frac{\ga-\gb}{1-\gb}$, then output $A$ and let
\begin{equation}\label{os2}
  \fx_1(0,y):=\Bigpar{0,y+\frac{1}{1-\gb}}.
\end{equation}
\item [(iib)] 
If $x=0$ and $\gb y \ge \frac{\ga-\gb}{1-\gb}$, then output $B$ and let
\begin{equation}\label{os3}
  \fx_1(0,y):=\Bigpar{\frac{1-\gb y}{1-\ga},0}.
\end{equation}
\end{romenumerate}
The system starts in $(0,0)$, and thus begins with \ref{ios1} or (iia)
which both give the same result when $x=y=0$.

\subsubsection{Second dynamical system}

We can simplify the analysis by noting that an election of $B$, by
\eqref{os3} always gives case \ref{ios1} and thus
election of $A$ for the next seat. Let us consider these two seat
assignments as a combined move. 
The combination thus start as in (iib) above with $\vfree=(0,y)$,
where $\gb y \ge \xpqfrac{\ga-\gb}{1-\gb}$. First $B$ is elected,
leaving by \eqref{os3} 
each ballot $A$ with
a free voting power $x'=\xpqfrac{1-\gb y}{1-\ga}$. 
Secondly, $A$ is elected, leaving by \eqref{os1} 
each ballot $B$ with
a free voting power 
\begin{equation}
  y''=\frac{1-\ga x'}{1-\gb}
=\frac{1-\ga-\ga(1-\gb y)}{(1-\ga)(1-\gb)}
=\frac{1-2\ga+\ga\gb y}{(1-\ga)(1-\gb)}.
\end{equation}

Using this combination instead of (iib) above, each case yields a vector of
the form $(0,y)$.
We can thus simplify the dynamical system to the following, acting on a single
variable $y\ge0$ (starting with $y=0$) by the function $\fx_2$ given by
\begin{romenumerate}
\item 
If $\gb y \ge \frac{\ga-\gb}{1-\gb}$, then output $BA$ and let
\begin{equation}\label{qu2}
  \fx_2(y):=\frac{1-2\ga+\ga\gb y}{(1-\ga)(1-\gb)}.
\end{equation}
\item 
If $\gb y \le \frac{\ga-\gb}{1-\gb}$, then output $A$ and let
\begin{equation}\label{ya2}
  \fx_2(y):=y+\frac{1}{1-\gb}.
\end{equation}
\end{romenumerate}

\subsubsection{Third dynamical system}

We simplify further by replacing $y$ by $z:=(1-\gb)y$, noting that
\begin{equation*}
\gb y \ge \frac{\ga-\gb}{1-\gb}  
\iff 
\gb z \ge \ga-\gb
\iff
z\ge \frac{\ga}\gb-1.
\end{equation*}
This yields an equivalent dynamical system acting on 
a variable $z\ge0$ (starting with $z=0$) by the function $\fx_3$ given by
\begin{romenumerate}
\item 
If $z \ge \frac{\ga}{\gb}-1$, then output $BA$ and let
\begin{equation}
  \fx_3(z):=\frac{1-2\ga}{1-\ga}+\frac{\ga\gb }{(1-\ga)(1-\gb)}z.
\end{equation}
\item 
If $z \le \frac{\ga}{\gb}-1$, then output $A$ and let
\begin{equation}
  \fx_3(z):=z+1.
\end{equation}
\end{romenumerate}

\subsubsection{Fourth dynamical system}\label{SSS4}
We replace $z$ by $w:=\ga/\gb-z$ and obtain the dynamical system
(starting with $w=\ga/\gb$)
given by the function $\fx_4$
defined by:
\begin{romenumerate}
\item \label{d4a}
If $w \le 1$, then output $BA$ and let
\begin{equation}\label{w41}
  \begin{split}
\fx_4(w)&:=\frac{\ga}\gb-\frac{1-2\ga}{1-\ga}
  -\frac{\ga\gb}{(1-\ga)(1-\gb)}\Bigpar{\frac{\ga}\gb-w}
\\&\phantom:
=
\frac{\ga}\gb+\frac{\ga}{1-\ga}-1-\frac{\ga^2}{(1-\ga)(1-\gb)}
+\frac{\ga\gb}{(1-\ga)(1-\gb)}w
.	
  \end{split}
\end{equation}
\item \label{d4b}
If $w \ge 1$, then output $A$ and let
\begin{equation}\label{w42}
  \fx_4(w):=w-1.
\end{equation}
\end{romenumerate}
In other words,
\begin{subnumcases}{\fx_4(w)=}  
	\ax w+\byx, & $w\le 1$,\label{fx4<}
\\
w-1, & $w\ge 1$,\label{fx4>}
  \end{subnumcases}
where
\begin{align}\label{ax}
\ax&
=\frac{\ga\gb}{(1-\ga)(1-\gb)}
=\frac{\ga\gb}{(\ga+\gab)(\gb+\gab)}
\in (0,1],
\\ \label{byx}
  \byx &= 
\frac{\ga}\gb+\frac{\ga}{1-\ga}-1-\frac{\ga^2}{(1-\ga)(1-\gb)}
=
\frac{\ga-\gb}\gb+\frac{\ga(1-\ga-\gb)}{(1-\ga)(1-\gb)}>0.
\end{align}

Note that $\ax<1$ unless $\gab=0$ (in which case \phragmen's method
reduces to D'Hondt's, as said above).
On the other hand, $\byx$ can be
arbitrarily large; we define $\bx:=\frax{\byx}\in\oio$ and
$\bz:=\floor\byx$.

Note also that 
$0<\fx_4(0)=\byx<\fx_4(1-)=\ax+\byx$ and that
\begin{equation}\label{gam+gd}
  \ax+\byx = 
\frac{\ga-\gb}\gb+\frac{\ga(1-\ga-\gb)+\ga\gb}{(1-\ga)(1-\gb)}
=
\frac{\ga}{\gb(1-\gb)}-1.
\end{equation}

\subsubsection{Final (fifth) dynamical system}

We can reformulate the dynamical system once more by combining 
each $BA$ move \eqref{fx4<} 
with all following $A$ moves 
\eqref{fx4>}. 
This yield the
dynamical system acting on $w\in\oi$ by the function $\fx_5:\oi\to\oi$ 
given by
\begin{align}\label{fx5}
  \fx_5(w):=\frax{\fx_4(w)}
=\frax{\ax w+\byx}
=\frax{\ax w+\bx}
\end{align}
with the output $BA^k$ where
\begin{equation}
  \label{fx5k}
k:=1+\floor{\fx_4(w)}=1+\floor{\ax w+\byx}
=1+\bz+\floor{\ax w+\bx},
\end{equation}
except that
in the indeterminate case
when $\ax w+\byx$ is an integer, we also allow 
$\fx_5(w)=\fraxx{\ax w+\bx}=1$ with $k:=\ax w+\byx$.

Thus $f_5(w)=f_\pm(w)$, the multi-valued function studied in the present
paper, with $a$ and $b:=\frax{\byx}$ given by \eqref{ax}--\eqref{byx}.
Furthermore, \eqref{fx5k} can be written, using \eqref{fx5}
and defining the symbol $\eps\in\setoi$ as in \eqref{eps2},
\begin{equation}\label{keps}
  k:= 1+\bz +\ax w+\bx -f_5(w)
= 1+\bz +\eps.
\end{equation}
Note that this includes both possibilities in the indeterminate case. 

The dynamical system really starts with $w=\ga/\gb$, which outputs $A$
$\floor{\ga/\gb}$ times  
before the first $B$
(or possibly one less, if $\ga/\gb$ is an integer), 
so in the version using $\fx_5$, we start with an initial
output $A^\ell$ with $\ell:=\floor{\ga/\gb}$ and then run the dynamical
system $f_\pm$ starting with $w=w_0:=\frax{\ga/\gb}$
(possibly modified if $\ga/\gb$ is an integer); the output is by
\eqref{keps} given by $BA^{1+\bz+\eps_i}$ for each symbol $\eps_i$ in the
symbolic sequence.
In other words, after the initial $A$'s, the output is obtain from the
symbolic sequence by the substitutions
\begin{equation}\label{epstoAB}
  0 \to BA^{\bz+1},
\qquad
1\to B A^{\bz+2}.
\end{equation}

\begin{example}\label{Egb=0}
  The case $\ga>\gb=0$ was excluded above. In this case, it is easily seen
  that every seat goes to $A$. Thus $n_A=n$ for any $n$. In particular,
  $n_A/n\to p_A=1$. (This can be seen as \eqref{pB} with $\bz=\infty$.)
\end{example}

\begin{example}\label{E==}
  The case $\ga=\gb$ was also excluded above. In this case, if $\ga=\gb>0$
  and $\gab>0$,
it is easily seen that the first seat goes to either $A$ or $B$, and all
following seats alternate between the two parties; hence 
$|n_A-n_B|\le1$.
In particular,   $n_A/n\to p_A=1/2$. 

In the extreme case $\ga=\gb=1/2$ and $\gab=0$, there is a tie at every
second seat; the first two seats go to either $AB$ or $BA$, and the same
holds for each following pair of seats; however, the order within each pair
is arbitrary. Hence \refT{Tphr}\ref{TphrQ} does not hold if, for example, ties
are resolved by lot. (However, it holds if ties always are resolved in
favour of, say, $A$.)
Nevertheless, in any case
we still have $|n_A-n_B|\le1$.

In the opposite extreme case $\ga=\gb=0$, so all votes are for $AB$ (and
thus $\gab=1$), 
every seat is a tie. If the ties are resolved by lot, then almost surely the
proportion $n_A/n\to p_A=1/2$, but other resolution rules may give \eg{} all
seats to $A$ (or $B$).
\end{example}

\begin{example}  \label{Egab=0}
The case $\gab=0$ is not excluded above; if  $\ga>\gb>0$ and $\gab=0$, then
\phragmen's method is still described by the dynamical system $f_5$ and
\eqref{epstoAB}. 
However, in this case \eqref{ax} yields $\ax=1$, and thus
$f_\pm(x)=\frax{x+\bx}$ (or $\fraxx{x+\bx}$), which is the
limiting case of a rotation on the circle mentioned in \refR{Rac}.
Our results in the preceding sections do not include this (simple) case,
but it is easy to see from \eqref{eps2} that \refT{Teps} still holds, with
the rotation number $\rho=\bx$.

Furthermore, since now $\ga+\gb=1$, \eqref{tbyx} yields
\begin{equation}\label{b00}
  \byx
=
\frac{\ga-\gb}\gb+\frac{\ga(1-\ga-\gb)}{\gb\ga}
=\frac{1-2\gb}{\gb}  
=\frac{1}{\gb}  -2.
\end{equation}
and thus $\bx=\frax{\byx}=\frax{1/\gb}$.
Since the dynamical system starts with
$w=\frax{\ga/\gb}=\frax{(1-\gb)/\gb}=\frax{1/\gb}$, it follows that 
$f_\pm^n(w)=\frax{(n+1)/\gb}$ or $\fraxx{(n+1)/\gb}$; hence,
if $\gb=p/q$ is rational, then there is a choice at each $p$:th
iteration. Hence, if e.g.\ the choices are made by lot, the orbit is \as{}
\emph{not} periodic. (We are in an orbit that is periodic except that each
$p$:th term is either 0 or 1, but these may be chosen arbitrarily.)
This is in stark contrast to the case $\ax<1$ studied in the present paper,
see for example \refL{LA01} and
\refT{TRbc}, and we see that
\refT{Tphr}\ref{TphrQ} does not hold when $\gab=0$. (Note that in this
case, $\rho=\bx\in\bbQ\iff\gb\in\bbQ$ by \eqref{b00}.)

Note that the same behaviour was found for $\gab=0$ and $\ga=\gb$ in \refE{E==}.
\end{example}

\subsection{Proof of \refT{Tphr}}
We consider several cases, and begin with the main case.
By symmetry, it suffices to consider $\ga\ge\gb$.

\pfcase{$\ga>\gb>0$ and $\gab>0$}
In this case, 
\phragmen's 
election method is described by the dynamical system $f_5=f_\pm$ as
described above.
Note that $\ax<1$ by \eqref{ax}.
Let $S_m:=\sum_{i=0}^{m-1}\eps_i$, where $\eps_i$ is the symbolic sequence
defined in \refSS{SSsymbolic}.
Let $m\ge0$ and suppose that at some stage of the election, $n_B=m$. This
means that we are in the $m$th iteration of the dynamical system; in other
words, we have so far made $m$ substitutions \eqref{epstoAB}, except that
the last may be incomplete. Taking into account also the initial string of
$A$'s, we obtain
\begin{equation}
  n_A=\sum_{i=0}^{m-1} (\bz+1+\eps_i)+O(1)=(\bz+1)m+S_m+O(1).
\end{equation}
Consequently, letting $\rho=\rho(f_\pm)$ be the rotation number of
\eqref{faxbx},
\refT{Teps} yields
\begin{equation}
  n_A=(\bz+1)m+\rho m+O(1),
\end{equation}
 which together with our assumption $n_B=m$ yields
 \begin{equation}
   n=n_A+n_B=(2+\bz+\rho) m + O(1)
 \end{equation}
and thus
 \begin{equation}\label{nBO}
n_B=m = \frac{n}{2+\bz+\rho}+O(1).
 \end{equation}
Consequently,
\begin{equation}
  \frac{n_B}{n}=\frac{1}{2+\bz+\rho} + O\Bigparfrac1{n},
\end{equation}
which shows both the existence of the limit $p_B$ as \ntoo, and its value
\eqref{pB} in \ref{TphrpB}.
Furthermore, obviously $n_A/n\to p_A:=1-p_B$,

\ref{TphrO} follows from \eqref{nBO}.

Finally, if $\rho$ is rational, then the symbolic sequence is eventually
periodic by \refT{TRbc}, and thus so is the sequence of awarded seats by
\eqref{epstoAB}, showing \ref{TphrQ}.

This completes the proof in \refpfcase.

\pfcase{$\ga>\gb>0$ and $\gab=0$}
As said in \refE{Egab=0},
we can use the dynamical system $f_5$ above in this case too; the only
difference from the preceding case is that now 
\eqref{tax} yields $\ax=1$, but \refT{Teps} still holds and \ref{TphrO} and
\ref{TphrpB} follow as above.
However, as noted in \refE{Egab=0}, \ref{TphrQ} does not always hold.

In this case, all votes are for $A$ or $B$, and \phragmen's method reduces
to D'Hondt's. The results can also easily be shown directly, see \eg{}
\cite{SJ262}.
Note that in this case, $\rho=\bx$ and thus, by \eqref{tbx}--\eqref{tbz} and
\eqref{b00},
$2+\bz+\rho=2+\byx=\gb\qw$; hence \eqref{pB} yields $p_B=\gb$.
In other words, when $\gab=0$,
the proportion of seats for a party converges to its
proportion of the votes, as said earlier.

\pfcase{$\ga>\gb=0$}
Trivial by \refE{Egb=0}, with $p_A=1$ and $p_B=0$.

\pfcase{$\ga=\gb>0$}
By \refE{E==}, \ref{TphrO}  holds, with $p_A=p_B=1/2$,
and if $\gab>0$, then also \ref{TphrQ} holds.
Furthermore, \eqref{gam+gd} yields
\begin{equation}
  \ax+\byx 
=\frac{1}{1-\gb}-1
=\frac{\gb}{1-\gb}
=\frac{\ga}{\ga+\gab}\le1.
\end{equation}
In particular, $\byx<1$ and thus $\bz=0$. Furthermore, 
$\ax+\bx\le1$, and thus the rotation number $\rho=0$, see
\refE{E00}.
Consequently, \eqref{pB} holds too.
\qed

\subsection{Further results}

We combine \refT{Tphr} with the result by 
\citet{laurent-nogueira} on rational rotation numbers
quoted above as \refT{Talgebraic}, and obtain the following.

\begin{theorem}\label{Tphr-alg}
  Consider the party version of \phragmen's election method with two parties.
If, with notation as in \refT{Tphr}, the proportions $\ga,\gb,\gab$
are algebraic numbers (in particular, if they are rational), and $0<\gab<1$,
then the sequence of awarded seats is eventually periodic.
In particular, the proportions $n_A/n$ and $n_B/n$ of seats given to each
party converge to rational numbers.
\end{theorem}

\begin{proof}
  By symmetry, we may assume $\ga\ge\gb$. The case $\gb=0$ is trivial by
  \refE{Egb=0} (all seats go to $A$); hence we may assume $\ga\ge\gb>0$, so
  \refT{Tphr}\ref{TphrpB} applies. The numbers $a$ and $\byx$ in
  \eqref{tax}--\eqref{tbyx} are algebraic, and thus so is $b$ by
  \eqref{tbx}. Furthermore, $0<a<1$ since $\gab>0$.
Hence, \refT{Talgebraic} 
applies and shows that
$\rho$ is rational. The proof is completed by \refT{Tphr}\ref{TphrQ}.
\end{proof}

\begin{remark}
Of course, in a real election, with integer numbers of votes,
the proportions of votes are always rational. (Unless
votes are weighted, 
and even then the proportions are rational or algebraic unless some weight
is transcendental.)
However, we are studying an idealized mathematical situation
(where we may let \ntoo), and then it is
natural to allow arbitrary real numbers $\ga$ and $\gb$ (with $\ga,\gb\ge0$
and $\ga+\gb\le1$).
\end{remark}

\begin{example}
  \label{E1/2}
When is $p_A=p_B=1/2$?
By symmetry we may assume $\ga\ge\gb$. 
Then
$\gb>0$ is necessary by \refE{Egb=0}, and thus
 \eqref{pB} shows that
$p_B=1/2$ if and only if 
$\bz+\rho=0$, i.e., 
if and only if
$\bz=0$ and $\rho=0$. 
By \refE{E00}, $\rho=0\iff \ax+\bx\le1$, and thus, 
using also \eqref{tbx}--\eqref{tbz} and \eqref{gam+gd}, 
\begin{equation}\label{pB1/2a}
  p_B=\frac12
\iff \bz=0 \text{ and }\ax+\bx\le1
\iff \ax+\byx\le1
\iff \ga\le 2\gb(1-\gb).
\end{equation}
By symmetry, if $\ga\le\gb$, then $p_B=1/2\iff \gb\le2\ga(1-\ga)$.

We may note that if $\ga\ge\gb$, then either $\ga\le\frac12$ and then
$\gb\le\ga\le2\ga(1-\ga)$, or $\ga\ge\frac12$ and then 
$\gb\le1-\ga\le2\ga(1-\ga)$; 
thus $\gb\le2\ga(1-\ga)$ always holds when
$\ga\le\gb$.
Hence, using symmetry again, we see that 
\begin{equation}\label{pB1/2}
p_B=\frac12 \iff
 \ga\le 2\gb(1-\gb) \text{ and }\gb\le2\ga(1-\ga), 
\end{equation}
as always excluding the case
$\ga=\gb=0$.

Given $\zeta$ with $0\le\zeta<1$, a simple calculation using \eqref{pB1/2a}
shows that
\begin{equation}\label{pB1/2z}
p_B=\frac12 \iff
\frac{3-\sqrt{1+8\gab}}4\le \ga \le \frac{1-4\gab+\sqrt{1+8\gab}}4.
\end{equation}
\end{example}

If $p_B=\frac12$ and $\zeta>0$, then
the sequence of awarded seats is eventually periodic by
\refT{Tphr}; furthermore, \eqref{epstoAB} shows that the sequence is
eventually alternating between the two parties.
In fact, in this simple special case,
the sequence alternates from the beginning.

\begin{theorem}\label{T1/2}
  Consider the party version of \phragmen's election method with two
  parties,
with the notations in \refT{Tphr}. If the conditions in \eqref{pB1/2} hold
and $0<\zeta<1$, then 
the seats are awarded alternatingly to $A$ and $B$ (starting with $A$ if
$\ga>\gb$, and with $B$ if $\gb>\ga$).  
\end{theorem}

\begin{proof}
The assumptions imply $\ga,\gb>0$, and the case $\ga=\gb$ follows by
\refE{E==}; hence we may, again using symmetry, 
assume $\ga>\gb>0$.
Then \phragmen's method 
is described by the dynamical system $f_5=f_\pm$ above,
starting at $w_0:=\frax{\ga/\gb}$,
after an initial $A^{\floor{\ga/\gb}}$.
We have, using \eqref{pB1/2},
$\gb<\ga\le2\gb(1-\gb)<2\gb$ and thus $1<\ga/\gb<2$.
Hence $\floor{\ga/\gb}=1$, and 
$0<w_0<1$.
Thus the first seat goes to $A$, and then we run $f_5$ starting at $w_0$.
We have $0<a<1$, and $a+b\le1$ since $\rho=0$ (see Example \ref{E00} or
\ref{E1/2}).
In the case, at most one symbol $\eps_i\neq0$, 
see Examples \ref{E<1}, \ref{E1}, \ref{E0} and \refSS{SSsymbolic}; 
furthermore, it is
easy to see that a non-zero $\eps_i$ can occur only in an orbit starting at
1 (if $a+b=1$) or 0 (if $b=0$), but this is not the case here since
$0<w_0<1$.
Thus, $\eps_i=0$ for all $i$, and thus \eqref{epstoAB} shows that the output
sequence is $A(BA)^\infty$. 
\end{proof}

If $p_B=1/2$ and $\zeta=0$, then \eqref{pB1/2} (or \refE{Egab=0})
implies that $\ga=\gb=1/2$; this case is treated in
\refE{E==}. As shown there, the sequence of elected seats is not
necessarily periodic in this case, because of ties.
Hence, \refT{T1/2} does not extend to $\zeta=0$.

\begin{remark}\label{Ebad}
  The  result in \refT{T1/2} is both surprising and unsatisfactory from the
  point of view of applications.
For example, if 40\% of the votes are for $A$, 30\% for $B$ and 30\% for
$AB$,
then \refT{T1/2} applies and shows that the seats are awarded $ABAB\dots$;
hence, for any even number of seats, $A$ and $B$ get equally many, in spite
of the fact that $A$ has substantially more votes than $B$.
\end{remark}

\begin{example}
  \label{E1/3}
When is $p_B=1/3$?
This cannot happen if $\gb>\ga$ or if $\gb=0$; thus $\ga\ge\gb>0$.
Hence, \eqref{pB} yields $\bz+\rho=1$, and thus (recalling that $\bz$ is an
integer),
$\bz=1$ and $\rho=0$.
Again, by \refE{E00}, $\rho=0\iff a+b\le1$. Furthermore, by
\eqref{tbx}--\eqref{tbz}, 
$\byx=\bz+b$, and thus,
using \eqref{tbyx} and \eqref{gam+gd},
for $\ga\ge\gb$,  
\begin{equation}\label{pB1/3}
  \begin{split}
  p_B=\frac13
&\iff \bz=1 \text{ and }\ax+\bx\le1
\iff \byx\ge1 \text{ and }\ax+\byx\le2
\\&
\iff \ga-2\gb-\ga^2+2\ga\gb+2\gb^2-3\ga\gb^2\ge0 \text{ and } \ga\le 3\gb(1-\gb)
. \end{split}
\end{equation}
\end{example}

\begin{example}
  \label{E2/5}
When is $p_B=2/5$?
We need $\ga>\gb>0$. Furthermore, \eqref{pB} yields $\bz+\rho=1/2$, i.e.,
$\bz=0$ and $\rho=1/2$.
Assume $\gab>0$, so $0<a<1$.
Using  \eqref{rho=1/2} in \refE{Erho=1/2}, we obtain,
assuming  $\ga\ge\gb$,
\begin{equation}
  \rho=\frac{2}5
\iff
\frac{1}{1+a}\le \byx\le \frac{1+a-a^2}{1+a}
\iff
1\le (1+a)\byx\le 1+a-a^2,
\end{equation}
with $a$ and $\byx$ given by \eqref{tax} and \eqref{tbyx}.
This can be expressed as two polynomial inequalities in $\ga$ and $\gb$, 
with one polynomial of degree 5 and one of degree 4; 
we omit the details.


\end{example}

Similarly, for any given rational $p\in(0,\frac12)$, 
one can see that
$p_B=p$ is equivalent to a few polynomial inequalities in $\ga$ and $\gb$,
but it seems that the degrees of the polynomials increase with the
denominator of $p$.

\section{Thiele's method}\label{Sthiele}\noindent

\subsection{Definition of Thiele's method}\label{SSthiele-algo}
Thiele's election method has a simple (and rather intuitive) formulation:
\begin{metod}{Thiele's election method}
Seats are awarded sequentially, and in each round, each ballot is counted as
$1/(\qn+1)$ for each name on it, where $\qn$ is the number of candidates on that
ballot that already have been elected.
\end{metod}

As with \phragmen's method, we consider the party version, where each ballot
contains a set of parties, and each party may get an arbitrary number of
seats;
then $\qn$ is counted with repetitions, \ie, $\qn$ is the number of seats that
so far have been awarded to the parties on the ballot.

We can rephrase  Thiele's method in  
the following form,
similar to the formulation of \phragmen's method in \refSS{SSphr-algo}.
As above, 
let $v_\gs$ be the number of votes for the set $\gs$ of candidates (parties).
The numbers $n_\gs$ defined below
will be the numbers of already elected on the different 
ballots (denoted $\qn$ in the description above).
\begin{romenumerate}
\item 
Initially all  $n_\gs=0$.  
\item \label{th-wi}
The reduced vote for candidate $i$ is defined as
\begin{equation}\label{thwi}
  W_i:=\sum_{\sigma\ni i}\frac{v_\sigma}{1+n_\sigma}.
\end{equation}
\item 
The candidate $i$ with the largest $W_i$ is elected to the next seat, breaking
ties by lot or some other method.
(In the original version, only unelected candidates are considered.
In the party version, repetitions are allowed.)
\item \label{th-update}
If $i$ is elected, then $n_\gs$ is updated for every $\gs\ni i$ (i.e., for
the ballots that contributed to the election of $i$); the new value is
\begin{equation}
  n_\gs':=n_\gs+1.
\end{equation}
$n_\gs$ remains unchanged when $\gs\not\ni i$.\\
Repeat from \ref{phr-wi}.
\end{romenumerate}

The difference from \phragmen's method is thus that the reduction of votes
in \eqref{thwi} is done in a different way.

\begin{remark}\label{RThall}
  A ballot voting for all parties will give the same contribution to
  everyone, and thus does not influence the result. In other words, with
  Thiele's method, ballots containing all parties can be ignored, just as
  blank votes.
  \end{remark}

\subsection{Main results for Thiele's method}

We assume as in \refSS{SSphr-dynamical} that we are given a set $\cP$ of
parties, and some numbers $v_\gs$ of votes on the sets $\gs\subseteq\cP$.
We let $n\ge1$ seats by distributed by Thiele's method, 
and let $n_i$ be the number of seats received by party $i\in\cP$. 
We also let
$p_{i}:=n_i/n$, the fraction of the seats received by $i$, and we define for
a set $\gs\subseteq\cP$ the sums
\begin{equation}
  n_\gs:=\sum_{i\in\gs}n_i,
\qquad
  p_\gs:=\sum_{i\in\gs}p_i = n_\gs/n.
\end{equation}
(These quantities all depend on $n$, but we do not show this in the
notation.)

We let $N=|\cP|$, the number of parties, and assume for notational
convenience
that $\cP=\set{1,\dots,N}$.
We let
$\vp=\vp_n:=(p_{1},\dots,p_{N})$, the vector of proportions of seats given
to the different parties.
Note that $\vp$ belongs to the simplex
\begin{equation}
\fS=\fS_N:=\Bigset{(x_1,\dots,x_N): x_i\ge0 \text{ and } \sum_{i=1}^N x_i=1}.  
\end{equation}
Let $\fSo:=\bigset{(x_1,\dots,x_N)\in\fS: x_i>0 \text{ for all $i$}}$,
the corresponding open simplex.

The following  two
theorems give conditions that guarantee that the vector $\vp$
converges, and provide a method to find the limit by solving a system of
(non-linear) equations.
\refT{TT} is more general, but its condition may be less easy to verify;
\refT{Tvi} has a simple condition that
still covers most cases of interest.
Furthermore, we give the even more general \refT{TT2} below, with a
different characterization of the limit.
The proofs of the results below are given in the next subsection.

\begin{theorem}
  \label{TT}
Consider Thiele's method for a set 
$\cP=\set{1,\dots,N}$ of $N$ parties 
with some given
numbers of votes $\set{v_\gs}_{\gs\subseteq\cP}$.
For a vector $(x_1,\dots,x_N)$,  define
\begin{equation}\label{xpi}
x_\sigma:=\sum_{i\in\sigma} x_i,
\qquad \gs\subseteq\cP.
\end{equation}
If, using \eqref{xpi},  the system of $N-1$ equations
\begin{gather}
  \sum_{\sigma\ni 1}\frac{v_\sigma}{x_\sigma}
=
  \sum_{\sigma\ni 2}\frac{v_\sigma}{x_\sigma}
=  \dots
=
  \sum_{\sigma\ni N}\frac{v_\sigma}{x_\sigma}
\label{tt1}
\end{gather}
has a unique solution $\vx_0$ in the open simplex $\fSo$,
then $\vp_n\to\vx_0$ as \ntoo.
\end{theorem}

Note that if $\vx\in\fSo$, or more generally $\vx\in\fS$, then
\begin{equation}
\label{tt2}
\sumiN x_i = 1,
\end{equation}
which together with \eqref{tt1} yields a system of $N$ non-linear equations
in the $N$ unknowns $x_i$.

\begin{theorem}\label{Tvi}
Consider Thiele's method for a set 
$\cP$ of $N$ parties 
with some given
numbers of votes $\set{v_\gs}_{\gs\subseteq\cP}$.
  Suppose that every party gets some individual vote, i.e.,
  \begin{equation}
	\label{vi}
v_{\set i}>0 \text{ for  every } i\in\cP.
  \end{equation}
Then the system
\eqref{tt1} has a unique
  solution $\vx_0$ in $\fSo$,
and $\vp_n\to\vx_0$ as \ntoo.
Moreover, $\vx_0$ is a smooth function of the vote numbers $v_\sigma$ as long
as \eqref{vi} holds.
\end{theorem}

The limit $\vx_0$ in these theorems
can also be characterized as the solution to an
optimization problem, which futhermore allows for a more general result.

Using the notations \eqref{xpi} and \eqref{PP},
and the standard convention $0^0=1$,
define the function, for $x_1,\dots,x_N\ge0$,
\begin{align}
  \Psi(x_1,\dots,x_N)&
:=\prod_{\sigma\neq\emptyset} x_\sigma^{v_\sigma}
=\prod_{\sigma\in\PP} x_\sigma^{v_\sigma}. \label{TPsi}
\end{align}
It is immediate that $\Psi$ is a
continuous function $\ooo^N\to\ooo$.
Let $M$ be the maximum of $\psi$ on the compact set $\fS$, and let
\begin{equation}\label{cMM}
  \cM:=\set{\vx\in\fS:\Psi(\vx)=M}
\end{equation}
be the set where the maximum is attained.

\begin{lemma}
  \label{LM}
  \begin{thmenumerate}
  \item \label{LM1}
$\cM$ is a non-empty compact convex subset of $\fS$.
\item \label{LM2}
 If $\vx\in\fSo$, then
\begin{equation}\label{linne}
  \vx\in\cM\iff\text{\eqref{tt1} holds}.
\end{equation}
\item \label{LM3}
If \eqref{tt1} has a unique solution $\vx_0$ in $\fSo$, then
$\cM=\set{\vx_0}$, i.e., $\vx_0$ is the only point in $\fS$ where the
maximum of $\Psi$ is attained.
  \end{thmenumerate}
\end{lemma}

The limit $\vx_0$ in Theorems \ref{TT}--\ref{Tvi} is thus the unique
maximum point of $\Psi$ on $\fS$.
The following, more general, theorem gives a (weaker) result
 also in the case when the maximum point is not unique.

\begin{theorem}\label{TT2}
Consider Thiele's method for $N$ parties $1,\dots,N$, with some given
numbers of votes $v_\gs$, for $\gs\subseteq\cP=\set{1,\dots,N}$.
Then,
  as \ntoo, $\vp_n\to\cM$, in the sense that the 
(Euclidean) distance $d(\vp_n,\cM)\to0$.
In particular, if $\cM$ consists of a single point, 
i.e.,  $\cM=\set{\vx_0}$ for some $\vx_0\in\fS$, then
$\vp_n\to\vx_0$.
\end{theorem}

\subsection{Proofs}
We consider the votes $\set{v_\sigma}$ as fixed.
Explicit and implicit constants below generally depend on $\set{v_\sigma}$.
%

We define for $x_1,\dots,x_N\ge0$,
recalling \eqref{TPsi} and
with the convention $0\cdot\infty=0$,
\begin{align}
  \psi(x_1,\dots,x_N)&:=
\log\Psi(x_1,\dots,x_N)
=\sum_{\sigma\in\PP} v_\sigma\log x_{\sigma}. \label{Tpsi}
\end{align}
Note that $\psi$ may take the value $-\infty$. 
Since $\Psi$ is a
continuous function $\ooo^N\to\ooo$, 
$\psi=\log \Psi$ is a
continuous function $\ooo^N\to[-\infty,\infty)$ (with the standard topology);
furthermore, $\psi$ is concave.

 The partial derivatives of $\psi$ are 
\begin{equation}\label{dipsi}
\ddi \psi:=  \frac{\partial \psi}{\partial x_i} 
= \sum_{\sigma\ni i} \frac{v_\sigma}{x_\sigma}.
\end{equation}
(If $\psi(\vx)=-\infty$, we regard the sum in \eqref{dipsi} as a definition
of $\ddi\psi(\vx)$.)
These derivatives are finite (and smooth) in $(0,\infty)^N$, 
but may be infinite
on the boundary; more precisely, 
$ \ddi \psi =+\infty$ when $x_\sigma=0$ for some
$\sigma\in\PP$ with $i\in\sigma$.

We are mainly interested in the behaviour of $\psi$ on 
the simplex $\fS$.
However, 
the partial derivatives $\ddi$ are along directions poining out of $\fS$;
we thus also consider directional derivatives in $\fS$. Let $e_i$,
$i=1,\dots,N$, be the unit vectors and define, for
$\vx=(x_1,\dots,x_N)\in\fS$, 
\begin{equation}
 e^*_i=e^*_i(\vx):=e_i-\sumjN x_j e_j
\end{equation}
which is parallel to $\fS$ and can be seen as a projection of $e_i$ to the
hyperplane $H$ of vectors tangent to $\fS$,
and the corresponding directional derivative
\begin{equation}\label{ddxi}
  \ddxi:=\partial_{e^*_i}=\partial_i-\sumjN x_j\partial_j.
\end{equation}
Equivalently, for any differentiable function $f$ on $(0,\infty)^N$ and
$\vx\in\fS$,
\begin{equation}\label{ddxi2}
  \ddxi f(\vx)=\ddi f(\vx)- \frac{d f(t\vx)}{d t}\Bigr|_{t=1}.
\end{equation}
Note that the vectors $e^*_i$ span the hyperplane $H$, and thus the
operators $\ddxi$ span the $(N-1)$-dimensional space of directional
derivatives parallel to $\fS$; moreover, they satsfy the linear relation
\begin{equation}\label{sumddxi}
  \sumiN x_i\ddxi=0.
\end{equation}

As said above, $\ddi\psi(\vx)$ may be $+\infty$ (but not $-\infty$).
Furthermore, it follows from \eqref{dipsi} that $x_i\ddi\psi(\vx)=O(1)$.
Hence $\ddxi\psi(\vx)$ is well-defined by \eqref{ddxi} for every
$\vx\in\fS$, with
\begin{equation}\label{ddxipsi}
  \ddxi\psi(\vx)=\ddi\psi(\vx)+O(1)\in(-\infty,\infty].
\end{equation}
Moreover, if $\vqq:=\sum_\sigma v_\sigma$, then 
by \eqref{Tpsi}, for any $t>0$, 
\begin{equation}
\psi(t\vx)=\psi(\vx)+\vqq\log t.  
\end{equation}
Hence, by \eqref{ddxi2}, for $\vx\in\fSo$, 
\begin{equation}\label{ddxi3}
\ddxi\psi(\vx)=\ddi\psi(\vx)-\vqq.  
\end{equation}
(More generally, \eqref{ddxi3} holds for all $\vx\in\fS$ with
$\psi(\vx)>-\infty$, but not necessarily everywhere on the boundary of $\fS$.)

Let $m:=\log M=\max_{\fS}\psi$.
Then \eqref{cMM} can be written
\begin{equation}\label{cM}
  \maxx=\set{\vx\in\fS:\psi(\vx)=m}.
\end{equation}
Since the function $\psi$ is concave (and the set $\fS$ convex),
if $\vx\in\fSo$, then
\begin{equation}\label{cMiff}
  \vx\in\cM \iff
\ddxi\psi(\vx) =0 \text{ for every $i$}.
\end{equation}
For $x$ on the boundary $\partial\fS$, we still have an implication 
$\Leftarrow$, but not necessarily in the opposite direction, see
\refE{Eexc1} below.

\begin{proof}[Proof of \refL{LM}]
\pfitemref{LM1}
$\cM$ is non-empty and compact by the definition \eqref{cMM}
because $\Psi$ is continuous and $\fS$ is compact.
Furthermore, $\cM$ is convex by \eqref{cM} because $\psi$ is a concave function.

\pfitemref{LM2}
If \eqref{tt1} holds, then by \eqref{dipsi},
  $\partial_1\psi(\vx)=\dots=\partial_N\psi(\vx)$. 
If also $\vx\in\fSo$, so \eqref{tt2} holds, then  \eqref{ddxi}
yields $\ddxi\psi(\vx)=0$ for every $i$, and 
thus  $\vx\in\cM$ by \eqref{cMiff}.

Conversely, if $\vx\in\cM\cap\fSo$, then \eqref{cMiff}
and \eqref{ddxi3} 
yield $\ddi\psi(\vx)=\vqq$ for all $i$, and thus
\eqref{dipsi} shows that
\eqref{tt1} holds.

\pfitemref{LM3}
If \eqref{tt1} has a unique solution $\vx_0$ in $\fSo$, 
then \ref{LM2} shows that 
$\cM\cap\fSo=\set{\vx_0}$.
Since $\cM$ is convex by \ref{LM1}, it follows that $\cM=\set{\vx_0}$.
\end{proof}

An important link between the seat assignments by Thiele's
method and the function $\psi$ is given by the following lemma.
We define
\begin{align}\label{nx}
  \nx&:=\min_{\sigma\in\PP} n_\sigma,
\\
  \px&:=\min_{\sigma\in\PP} p_\sigma=\nx/n. \label{px}
\end{align}
\begin{lemma}\label{LW}
For every party $i$,
  \begin{equation}\label{lw1}
W_i
=\frac{1}{n}\ddi\psi(\vp)+O\Bigpar{\frac{1}{\nx^2}}.
  \end{equation}
 
Moreover
\begin{equation}\label{wp3}
  \psi(\vp_{n+1})-\psi(\vp_n) 
= \frac{1}{n}\max_i \ddxi\psi(\vp_n)+ O\Bigpar{\frac{1}{\nx^2}}.
\end{equation}
\end{lemma}

\begin{proof}
First, for any $i$, \eqref{thwi} yields
\begin{equation}
W_i 
= \sum_{\sigma\ni i}\Bigpar{\frac{v_\sigma}{n_\sigma}
  +O\Bigpar{\frac{v_\sigma}{n_\sigma^2}}}  
=\sum_{\sigma\ni i}\frac{v_\sigma}{ n p_\sigma}+O\Bigpar{\frac{1}{\nx^2}}  
,
\end{equation}
which by  \eqref{dipsi}  shows \eqref{lw1}.

Suppose that the $(n+1)$:th seat goes to party $\ell$.
  Let, with 
$\vn:=(n_{1},\dots,n_{N})$,
  \begin{equation}\label{gDp}
\gD\vp:=\vp_{n+1}-\vp_n=\frac{\vn+e_\ell}{n+1}-\frac{\vn}{n}
=\frac{e_\ell-\vp_n}{n+1}	
  \end{equation}
and note that $\abs{\gD\vp}=O(1/n)$.

It follows from \eqref{dipsi} that
\begin{align}
  \ddi\psi(\vx)&=O\parfrac{1}{\min_{\sigma\in\PP} x_\sigma}, \label{snp1}
\\
  \ddi^2\psi(\vx)&=O\parfrac{1}{(\min_{\sigma\in\PP} x_\sigma)^2}.\label{snp2}
\end{align}
Thus, for $\vx$ on the line segment between $\vp_n$ and $\vp_{n+1}$,
since $\vp_{n+1}\ge \frac{n}{n+1}\vp_n$,
\begin{equation}
  \ddi^2\psi(\vx)=O\parfrac{1}{\px^2}
=O\parfrac{n^2}{\nx^2}.
\end{equation}
Hence, a Taylor expansion yields,
using \eqref{gDp} and \eqref{ddxi},
\begin{equation}\label{lx}
  \begin{split}
\psi(\vp_{n+1})-\psi(\vp_n)	
&=
\psi(\vp+\gD\vp)-\psi(\vp)	
=\gD\vp\cdot\nabla\psi(\vp)+O\lrpar{\frac{n^2}{\nx^2}|\gD\vp|^2}
\\&
=\frac{\ddq_\ell\psi(\vp)-\sumjN p_j\partial_j\psi(\vp)}{n+1}
+O\parfrac{1}{\nx^2}
\\&
=\frac{\ddx_\ell\psi(\vp)}{n+1}
+O\parfrac{1}{\nx^2}.
  \end{split}
\raisetag\baselineskip
\end{equation}
Furthermore, by \eqref{snp1},
\begin{equation}
  \frac{\ddx_\ell\psi(\vp)}{n+1}-
\frac{\ddx_\ell\psi(\vp)}{n}
= - \frac{\ddx_\ell\psi(\vp)}{n(n+1)}
=O\parfrac{1}{n^2\px}  
=O\parfrac{1}{\nx^2}  .
\end{equation}
Consequently, 
\eqref{lx} yields
\begin{equation}\label{lx2}
\psi(\vp_{n+1})-\psi(\vp_n)	
=\frac{1}n \ddx_\ell\psi(\vp)
+O\parfrac{1}{\nx^2}.
\end{equation}
Furthermore, by the definition of Thiele's method, $W_\ell=\max_i W_i$, and
thus \eqref{lw1} yields
\begin{equation}
 \frac{1}n \ddx_\ell\psi(\vp) 
=\max_i W_i +O\Bigpar{\frac{1}{\nx^2}}
=\max_i \frac{1}{n}\ddxi\psi(\vp) +O\Bigpar{\frac{1}{\nx^2}},
\end{equation}
which yields \eqref{wp3} by \eqref{lx2}.
\end{proof}

\begin{lemma}
  \label{Ldiff}
Let $U\subset\fS$ be an open neighbourhood of $\cM$.
Then there exists $\cc>0\ccdef\ccdiff$ 
such that for every $\vx\in\fS\setminus U$, there
exists $i$ with $\ddxi\psi(\vx)\ge \ccx$. 
\end{lemma}

\begin{proof}
  Let 
  \begin{equation}
g(\vx):=\max_{1\le i\le N} {\ddxi\psi(\vx)} .	
  \end{equation}
The assertion is equivalent to $g(\vx)\ge \ccx$ for $x\notin U$.
We first show $g(\vx)>0$.

Suppose that $\vx\in\fS$ with $g(\vx)\le0$. 
Then $\ddxi\psi(\vx)\le0$ for every $i$.
It follows from \eqref{sumddxi} that then $x_i\ddxi\psi(\vx)=0$ for every
$i$, so  
$\ddxi\psi(\vx)=0$ for every $i$ such that $x_i>0$.

Let $\vy:=(y_1,\dots,y_N)$ be any point in $\fS$, and let
$h(t):=\psi(\vx+t(\vy-\vx))$. Then $h$ is a concave function on $\oi$, 
and, using \eqref{ddxi} and $\sum_i(y_i-x_i)=1-1=0$,
\begin{equation}\label{h'0}
h'(0)=
  \sumiN(y_i-x_i)\ddi\psi(\vx)
=   \sumiN(y_i-x_i)\ddxi\psi(\vx).
\end{equation}
If $x_i>0$, then $\ddxi\psi(\vx)=0$ as just seen. Furthermore,
if $x_i=0$, then $y_i-x_i\ge0$ and $\ddxi\psi(\vx)\le0$.
It follows that every term in the final sum in \eqref{h'0} is $\le0$,
and thus $h'(0)\le0$. Since $h$ is concave, this implies $\psi(\vy)=h(1)\le
h(0)=\psi(\vx)$.  

We have shown that if $\vx\in\fS$ and $g(\vx)\le0$, then $\psi(\vx)\ge\psi(\vy)$
for every $\vy\in\fS$, and thus $\vx\in\cM$.
Equivalently, if $\vx\notin\cM$, then $g(\vx)>0$.

To complete the proof, it suffices to show that $g$ is continuous on $\fS$
(with values in $[0,\infty]$). This is not quite trivial, since the
individual $\ddxi\psi$ in general are not, because $x_j\partial_j\psi(\vx)$
by \eqref{dipsi} is discontinuous at $x_j=0$ if $v_{\set{j}}>0$.
We let $\vx\in\fS$ and consider two cases.
\begin{romenumerate}
\item 
If $\ddi\psi(\vx)<\infty$ for every $i$, then 
by \eqref{dipsi} and \eqref{ddxi},
this holds in a
neighbourhood $V$ of $\vx$, and in $V$ furthermore every $\ddi\psi$ and every
$\ddxi\psi$ is continuous. Hence, $g$ is continuous at $\vx$.
\item 
If $\ddi\psi(\vx)=\infty$ for some $i$, suppose that $\vy\to\vx$
with $\vy\in\fS$. Then $\ddi\psi(\vy)\to\infty$ and thus, using \eqref{ddxipsi},
\begin{equation}
  g(\vy)\ge\ddxi\psi(\vy)=\ddi\psi(\vy)+O(1)\to\infty=g(\vx).
\end{equation}
Hence, $g$ is continuous at $\vx$ in this case too.
\end{romenumerate}
 Consequently, $g$ is continuous everywhere in $\fS$, and since we have
 shown that $g>0$ on the compact set $\fS\setminus U\subseteq
 \fS\setminus\cM$, the result follows.
\end{proof}

\begin{lemma}
  \label{Lnx}
As \ntoo, $\nx\to\infty$.
\end{lemma}
\begin{proof}
  Suppose not. Then, since each $n_\sigma$ is non-decreasing,  there exists
  $\sigma\in\PP$ such that $n_\sigma=O(1)$.
Let
\begin{align}
  \PP_0&:=\set{\sigma\in\PP:n_\sigma=O(1)},
\\
  \cE&:=
\bigcup_{\sigma\in \PP_0}\sigma=
\bigset{i:\exists \sigma\in\PP_0 \text{ with } i\in\sigma}.
\end{align}
Then $\cE$ is a non-empty set of parties, and if $i\in\cE$, then there
exists $\sigma$ with $i\in\sigma\in\PP_0$ and thus $n_i\le n_\sigma=O(1)$.
In other words, after some time, no further seat goes to any party in $\cE$.

On the other hand, if $i\in\cE$, take again $\sigma\in\PP_0$ with
$i\in\sigma$. Then by \eqref{thwi},
\begin{equation}\label{wc}
  W_i \ge \frac{v_\sigma}{1+n_\sigma} \ge c,
\end{equation}
for some $c>0$. 
On the other hand, if $i\notin\cE$, then $n_\sigma\to\infty$ for every
$\sigma\in\PP$ such that $i\in\sigma$, and thus \eqref{thwi} yields $W_i\to0$.
This implies that if $n$ is large enough, then $W_i<c$ for every
$i\notin\cE$,
so by \eqref{wc}, the party $i$ with the largest $W_i$ is a party in $\cE$,
and thus every seat, for large $n$, goes to a party in $\cE$. 
This contradiction proves the lemma.
\end{proof}

\begin{lemma}
  \label{Linc}
Let $U\subset\fS$ be an open neighbourhood of $\cM$. Then there exists
$n_0$ and $\cc>0\ccdef\ccinc$ 
such that for all $n\ge n_0$, either $\vp_n\in U$ or 
\begin{equation}
\label{linc}
\psi(\vp_{n+1})-\psi(\vp_n)\ge \ccx/n.  
\end{equation}
\end{lemma}

\begin{proof}\newpfcase
  Let $\ccdiff$ be as in \refL{Ldiff} and let $\ccx:=\ccdiff/2$.
We assume $\vp=\vp_n\notin U$ and use \eqref{wp3}. We consider two cases.

\pfcase{$\nx\ge n^{3/4}$}
By \eqref{wp3} and \refL{Ldiff},
\begin{equation}
  \psi(\vp_{n+1})-\psi(\vp_n)\ge \frac{\ccdiff}{n}+ O(\nx^{-2})
= \frac{\ccdiff}{n}+ O\bigpar{n^{-3/2}},
\end{equation}
which is larger than $\ccx/n$ for large $n$.

\pfcase{$\nx<n^{3/4}$}
Let $\sigma\in\PP$ with $n_\sigma=\nx$. By \eqref{dipsi}, for any $i\in\sigma$,
\begin{equation}
  \ddi\psi(\vp)\ge \frac{v_\sigma}{p_\sigma} = n\frac{v_\sigma}{n_\sigma}
\end{equation}
and thus by \eqref{wp3} and \eqref{ddxipsi},
using also \refL{Lnx},
\begin{equation}
  \begin{split}
  \psi(\vp_{n+1})-\psi(\vp_n)
&\ge
\frac{1}n \ddxi\psi(\vp) + O\bigpar{\nx\qww}
=
\frac{1}n \ddi\psi(\vp) + O\bigpar{n\qw}+ O\bigpar{\nx\qww}
\\&
\ge \frac{v_\sigma}{n_\sigma} + O\bigpar{n\qw}+ O\bigpar{\nx\qww}	
= 
 \frac{v_\sigma}{\nx} + o\bigpar{\nx\qw}.
  \end{split}
\raisetag\baselineskip
\end{equation}
For large $n$, the \rhs{} is at least $v_\sigma/(2\nx)\ge \ccx/n$.
\end{proof}

We can now prove the theorems showing convergence
of the proportions $\vp_n$ for Thiele's method.

\begin{proof}[Proof of \refT{TT2}]
  Let $\eps>0$, and let $U:=\set{\vx\in\fS:d(\vx,\cM)<\eps}$.
If $\vx\notin U$, then $\vx\notin\cM$ and thus $\psi(\vx)<m$; 
hence, by compactness, there exists $\gd>0$ such that $\psi(\vx)\le m-\gd$ 
for $x\in\fS\setminus U$.

Let $U_1:=\set{\vx\in\fS:\psi(\vx)>m-\gd}$ and 
$U_2:=\set{\vx\in\fS:\psi(\vx)>m-\gd/2}$.
Then $U_1$ and $U_2$ are open in $\fS$ and
\begin{equation}
  \cM \subset U_2\subset \bU_2 \subset U_1 \subseteq U.
\end{equation}
In particular, the two compact sets $\bU_2$ and $\fS\setminus U_1$ are
disjoint, and thus have a positive distance $\eta$. In other words, $\eta>0$
and if $\vx\in\fS$ with $d(\vx,\bU_2)<\eta$, then $\vx\in U_1$.
We apply \refL{Linc} to $U_2$.

First, we claim that $\vp_n\in U_2$ for infinitely many $n$.
In fact, if this is false, then by \refL{Linc}, 
\eqref{linc} holds for all large $n$. Since 
$\sum_n \ccinc/n =\infty$, this would imply $\psi(\vp_n)\to\infty$, which is a
contradiction because $\psi(\vx)\le0$ when $\vx\in\fS$.

Next, suppose that $n\ge n_0$ and that $\vp_n\in U_1$. There are two cases.
\begin{romenumerate}
\item 
If $\vp_n\notin U_2$, then \eqref{linc} holds, and thus
\begin{equation}
  \psi(\vp_{n+1}) > \psi(\vp_n) > m-\gd;
\end{equation}
hence $\vp_{n+1}\in U_1$. 
\item 
If $\vp_n\in U_2$, then we use \eqref{gDp} which implies $|\gD\vp|\to0$.
Hence, provided $n$ is large enough,  
$|\gD\vp|<\eta$, which together with $\vp_n\in U_2$ and the definition of
$\eta$ implies $\vp_{n+1}=\vp_n+\gD\vp\in U_1$.
\end{romenumerate}

We have thus shown that, in any case, if $n$ is large enough and $\vp_n\in
U_1$, then $\vp_{n+1}\in U_1$. Since we also have shown that $\vp_n\in
U_2\subset U_1$ for arbitrarily large $n$, it follows that for all
sufficiently large $n$, $\vp_n\in U_1\subset U$, and thus $d(\vp_n,\cM)<\eps$.
\end{proof}

\begin{proof}[Proof of \refT{TT}]
The assumption that $\vx_0$ is a unique solution of \eqref{tt1} in $\fSo$
implies $\cM=\set{\vx_0}$ by \refL{LM}\ref{LM3}.
Hence, the result follows from \refT{TT2}. 
\end{proof}

\begin{proof}[Proof of \refT{Tvi}]
  When \eqref{vi} holds, $\Psi(\vx)=0$ as soon as some $x_i=0$;
hence the maximum $M$ of $\Psi$ can not be attained on the boundary of $\fS$ so
$\cM\subset\fSo$. Furthermore, along any straight line in $\fSo$, each term
in the sum \eqref{Tpsi} is smooth and concave, and at least one of the terms
has a strictly negative second derivative. (For example the term with
$\sigma=\set i$ for any $i$ such that $x_i$ varies along the line.)
It follows that $\psi$ is strictly concave in $\fSo$,
and thus the maximum set $\cM$ cannot contain more than one point.
Hence $\cM=\set{\vx_0}$ for some point $\vx_o\in\fSo$.
It follows by \refL{LM}\ref{LM2} that $\vx_0$ is the unique solution of
\eqref{tt1} 
in $\fSo$, and $\vp_n\to\vx_0$ follows by \refT{TT2} (or \refT{TT}).

Finally, 
use $x_1,\dots,x_{N-1}$ as coordinates on $\fSo$ and 
write $\psix(x_1,\dots,x_{N-1}):=\psi(x_1,\dots,x_{N-1},1-x_1-\dots-x_N)$.
Then the maximum point $\vx_0$ is given by
\begin{equation}\label{disk}
  D\psix :=\Bigpar{\frac{\partial\psix}{\partial x_i}}_{i=1}^{N-1} =0.
\end{equation}
Moreover, the function $\psix$ is concave in $\fSo$, with a strictly
negative second derivative along any line  as shown above; in other words,
the Hessian matrix
$\bigpar{\frac{\partial^2\psix}{\partial x_i\partial x_j}}_{i,j=1}^{N-1} $
is negative definite, and thus non-singular at every point.
It follows from the implicit function theorem that the solution $\vx_0$
of \eqref{disk}
is a smooth function of the parameters $v_\sigma$.
\end{proof}

\subsection{Examples and further results}

\begin{example}[Two parties]\label{ET2}
  Suppose that there are two parties, $A$ and $B$, and
assume $v_A,v_B>0$.
The equation \eqref{tt1} is
  \begin{equation}
	\frac{v_A}{x_A} + \frac{v_{AB}}{x_{AB}}
=
	\frac{v_B}{x_B} + \frac{v_{AB}}{x_{AB}},	
  \end{equation}
which simplifies to $v_A/x_A = v_B/x_B$, so the system
\eqref{tt1}--\eqref{tt2} has the unique solution $x_A=v_A/(v_A+v_B)$,
$x_B=v_A/(v_A+v_B)$. 
\refT{TT} applies and thus
$n_A/n\to x_A=v_A/(v_A+v_B)$; see also \refT{Tvi}. 
This also follows from \refR{RThall}, which for two parties
says that we can ignore the ballots $AB$, leaving only ballots $A$ and $B$,
and then Thiele's method reduces to D'Hondt's for which the result is well
known. 
\end{example}

We continue with some examples with three parties.

\begin{example}
  Supose that there are three parties $A,B,C$, and 5 votes: 
1 $A$, 1 $B$, 1 $C$, 1 $AB$, 1 $AC$.
Then \eqref{tt1} is
\begin{equation}
  \frac{1}{x_A}+\frac{1}{x_A+x_B}+\frac{1}{x_A+x_C}
=
  \frac{1}{x_B}+\frac{1}{x_A+x_B}
=
  \frac{1}{x_C}+\frac{1}{x_A+x_C}.
\end{equation}
\refT{Tvi} applies, and thus \eqref{tt1}--\eqref{tt2} has a unique solution
in $\fSo$, which by symmetry has to satisfy $x_B=x_C$. 
Hence 
\eqref{tt1} simplifies to 
\begin{equation}
  \frac{1}{x_A}+\frac{1}{x_A+x_B}
=
  \frac{1}{x_B}.
\end{equation}
Furthermore, $x_B=x_C=(1-x_A)/2$, and we obtain
\begin{equation}
  \frac{1}{x_A}+\frac{2}{1+x_A}
=
\frac{2}{1-x_A}  
\end{equation}
which yields the quadratic equation $5x_A^2-1=0$.
Hence the maximum point $\vx_0$ is given by $x_A=1/\sqrt 5$,
$x_B=x_C=\frac12(1-1/\sqrt5)$. \refT{TT} yields
$\vp\to\vx_0=\frac{1}{2\sqrt5}(2,\sqrt5-1,\sqrt5-1)$.

The fact that the proportions converge to irrational numbers shows that even
in this simple example, there is no ultimate periodicity in the seat 
assignment. 

\begin{problem}
Is the sequence of seats assigned to $A$ quasiperiodic in some
sense? Can it be described explicitly?  
\end{problem}

More generally, a similar calculation shows that if the number of votes for
$A$ is changed to an arbitrary $v_A>0$ (with the other votes kept the same),
then $n_A/n\to x_A=\sqrt{v_A/(v_A+4)}$.
\end{example}

In the general (non-symmetric) case with 3 parties, \eqref{tt1}--\eqref{tt2}
lead (using Maple)
to a quartic equation for $x_A$, where the coefficients are polynomials of
degree 3 in the  vote numbers $v_\gs$.
We spare the reader the general formula, and give a numerical example.

\begin{example}
    Supose that there are three parties $A,B,C$, and 9 votes: 
1 $A$, 2 $B$, 3 $C$, 1 $AB$, 1 $AC$, 1 $BC$.
\refT{Tvi} applies and shows $\vp\to\vx_0$ for some 
solution $\vx_0=(x_A,x_B,x_C)$ to \eqref{tt1}--\eqref{tt2}.
Maple yields that $x_A$ is a root of
\begin{equation}
 135 x_A^{4}-161 x_A^{3}-22 x_A^{2}+64 x_A-10
=0.
\end{equation}
Numerically, 
$x_A=0.1797714258$, $x_B = 0.341215728$, 
$x_C=0.4790128462$.
\end{example}

\begin{example}[An exceptional case]\label{Eexc1}
Supose that there are three parties $A,B,C$, and 6 votes: 2 $A$, 2 $B$, 1
$AC$, 1 $BC$. Then \eqref{tt1} is
\begin{equation}
  \frac{2}{x_A}+\frac{1}{x_A+x_C}
=
  \frac{2}{x_B}+\frac{1}{x_B+x_C}
=
  \frac{1}{x_A+x_C}+\frac{1}{x_B+x_C}
\end{equation}
and it is easily found that the unique solution that also satisfies
\eqref{tt2} is $(\frac{2}3,\frac{2}3,-\frac{1}3)$. 
This solution lies outside $\fS$, so
\refT{TT} does not apply. However, \refT{TT2} still applies, and it is
easily verified that the maximum set $\cM$ consist of the single point
$(\frac12,\frac12,0)$; thus $p_{An}\to\frac12$, $p_{Bn}\to\frac12$ and
$p_{Cn}\to0$. 

In fact, it is easily seen from \eqref{thwi}
that $C$ will never get any seat, since always at
least one of $W_A$ and $W_B$ is larger than $W_C$. Furthermore, each pair of
sets goes to either $A,B$ or $B,A$, and thus for any even number of seats
$n$, $p_{An}=p_{Bn}=\frac12$ and $p_{Cn}=0$ exactly.  
\end{example}

\begin{example}[Another exceptional case]
  Suppose that there are three parties $A,B,C$ and two votes: 1 $A$ and 1 $BC$.
In this case, 
$\Psi(x_A,x_B,x_C)= x_A(x_B+x_C)=x_A(1-x_A)$, and it is easy to see that
$\cM=\set{(\frac12,x_B,\frac12-x_B):x_B\in[0,\frac12]}$, a line segment.

Indeed, in this case, of each pair of seats, one goes to $A$ and the other
to either $B$ or $C$. If ties are resolved by lot, almost surely 
$\vp_n\to(\frac12,\frac{1}4,\frac{1}4)$, but for other tie-breaking rules,
other limits in $\cM$ are possible, and $\vp_n$ may even oscillate without a
limit, for
example if a tie for seat $n$ is resolved in favour of $B$ when
$\floor{\log_2 n}$ is even, and in favour of $C$ otherwise.
\end{example}

\begin{example}\label{ETpairs}
\def\eta{v_{AC}}
\def\zeta{v_{BC}}
\def\xi{v_{AB}}
  Supose that there are three parties $A,B,C$, and only votes for
  combinations of two parties,
with $\xi,\eta,\zeta>0$ and $\xi+\eta+\zeta=1$.
Then \eqref{tt1} is
\begin{equation}
  \frac{\xi}{x_A+x_B} +  \frac{\eta}{x_A+x_C}
=   \frac{\xi}{x_A+x_B} +  \frac{\zeta}{x_B+x_C}
=   \frac{\eta}{x_A+x_C} +  \frac{\zeta}{x_B+x_C}
\end{equation}
  which yields
\begin{equation}\label{rudbeck}
  \frac{\xi}{x_A+x_B} 
= \frac{\eta}{x_A+x_C}
=  \frac{\zeta}{x_B+x_C}.
\end{equation}
The equations \eqref{rudbeck} and \eqref{tt2} have the unique solution
  \begin{equation}\label{olof}
(x_A,x_B,x_C)=\bigpar{\xi+\eta-\zeta,\xi+\zeta-\eta,\eta+\zeta-\xi}.
  \end{equation}
If the three numbers on the \rhs{} of \eqref{olof} are positive, 
then $\vx_0:=(x_A,x_B,x_C)\in\fSo$, so
\refT{TT}
applies and shows that $\vp\to \vx_0$ given by \eqref{olof}.
Note that \refT{Tvi} does not apply, but nevertheless we have the same
conclusions, with a limit $\vx_0$ that is a smooth function of the vote
numbers by \eqref{olof}. Hence the condition \eqref{vi} is not necessary for
good behaviour.

Suppose now that, say, $\xi=\eta+\zeta$.
Then $\vx_0$ given by \eqref{olof}
has one coordinate 0 and lies thus on the boundary
$\partial\fS$; nevertheless $\ddxi\psi(\vx_0) =0$  for every $i$, 
e.g.\ by \eqref{ddxi3},
and as
remarked after \eqref{cMiff}, this implies $\vx_0\in\cM$.
Furthermore, $\psi$ is strictly concave on $\fS$, and thus $\cM=\set{\vx_0}$.
Thus \refT{TT2} applies and yields $\vp\to\vx_0$ in this case too.

Finally, suppose that $\xi>\eta+\zeta$.
Then \eqref{olof} would yield $x_C<0$, so \eqref{tt1} has no solution in $\fS$.
It is easy too see that the maximum of $\psi$ on $\fS$ is attained on the 
part of the boundary with $x_C=0$, and then we have
$\Psi(x_A,x_B,0)=x_A^{\eta}x_B^{\zeta}(x_A+x_B)^{\xi}
=x_A^{\eta}x_B^{\zeta}$,
leading to the same
  equation as in \refE{ET2} (with some changes in notation) and
$\cM=\set{\vx_0}$ with 
\begin{equation}\label{olov}
  \vx_0=\Bigpar{\frac{\eta}{\eta+\zeta},\frac{\zeta}{\eta+\zeta},0}.
\end{equation}
Again, \refT{TT2} applies and yields $\vp\to\vx_0$.

In fact, it is easy to see that if $\xi\ge\eta+\zeta$, then party $C$ will
never get any seat by Thiele's method (we omit the details).
Hence, we may in this case ignore $C$ and the result follws by \refE{ET2}. 
(Note that in the case $\xi=\eta+\zeta$, \eqref{olof} and \eqref{olov} yield
the same result.)
\end{example}

\begin{example}
  Supose that there are three parties $A,B,C$, and no votes $AC$ or $BC$.
Then \eqref{tt1} becomes
\begin{align}
  \frac{v_A}{x_A}=\frac{v_B}{x_B},
&&&
\frac{v_A}{x_A}+\frac{v_{AB}}{x_A+x_B}=\frac{v_C}{x_C}
\end{align}
which leads to 
\begin{equation}
\frac{v_A+v_B+v_{AB}}{x_A+x_B}=\frac{v_C}{x_C}
\end{equation}
and thus to $x_C=v_C/\vqq$ and, \eg, 
$x_A=\frac{v_A}{v_A+v_B}(1-x_C)$.
This is generalized in \refT{Tblock}.
\end{example}

\begin{theorem}
  \label{Tblock}
Suppose that the $N$ parties are partitioned into a number of blocks 
$\cP_1, \cP_2, \dots$,
and that each voter votes for a subset of one of the blocks.
Assume also, for simplicity, assume that \eqref{vi} holds.
Then $\vp$ converges to a limit $\vx=(x_1,\dots,x_N)$ 
such that if $\cP_j$ is one of the blocks,
$q_j=\sum_{\sigma\subseteq\cP_j} v_\sigma/\vqq$ is the proportion of votes for this
block, and
$x'_i$ is the asymptotic proportion of seats assigned to party $i$ if the 
election is restricted to the parties in $\cP_j$ only (with the same votes
for them),  then
$x_i=qx'_i$ for every party $i\in \cP_j$.
\end{theorem}

\begin{proof}
  It suffices to consider the case of two blocks, $\cP_1$ and $\cP_2$.
For $j=1,2$, let $N_j:=|\cP_j|$,  the number of parties in block $\cP_j$, 
let $\vqq_j:=\sum_{\sigma\subseteq\cP_j}v_\sigma$, the number of
votes for block $\cP_j$, 
let $z_j:=\sum_{i\in \cP_j}x_i$ and, for $i\in\cP_j$,
$y_i:=x_i/z_j$. If $\sigma\subseteq\cP_j$, then thus
$x_\sigma=\sum_{i\in\sigma}z_jy_i=z_jy_{\sigma}$.
Consequently, if $\vy^j:=(y_i)_{i\in\cP_j}\in\fS_{N_j}$, and $\psi^j$
denotes $\psi$ 
defined as in \eqref{Tpsi} but for the votes $\sigma\subseteq\cP_j$ only, then
\begin{equation}
  \psi(\vx)=\sum_j\sum_{\sigma\subseteq\cP_j} v_\sigma(\log z_j+\log y_i)
=\sum_j \vqq_j\log z_j + \sum_j \psi^j(\vy^j).
\end{equation}
Evidently, this is maximized by maximizing each $\psi^j(\vy^j)$ separately,
and maximizing $\vqq_1\log z_1+\vqq_2\log z_2$ subject to $z_1+z_2=1$;
the latter leads to $z_j=\vqq_j/\vqq$, and the result follows.
\end{proof}

\begin{remark}
  \refT{Tblock} is very natural and satisfying. It means for example that a
  party cannot influence its shares of seats by tactically splitting into
  several parts and distributing votes among combinations of them in some
  clever way.

However, this asymptotic result does \emph{not} hold for small numbers of
seats. In fact, one of the main problems with Thiele's method when used in
Sweden in the 1910's (see the Historical note in \refS{S:intro})
was the possibility of such
manoeuvres. An example by \cite{Tenow1912} 
(see also \cite{SJV9}) is 3 seats and the votes 37
$ABC$, 13 $KLM$;
in this case Thiele's method reduces to D'Hondt's, and gives two seats to
$ABC$ and  one to $KLM$. However, if the $ABC$ voters split their votes as 1
$A$, 9 $AB$, 9 $AC$, 9 $B$, 9 $C$ then all three seats go to $ABC$.
 (Of course, Tenow considers individual candidates and not
parties, but for this example the result is the same.) 

It thus seems that Thiele's method behaves better asymptotically than for
a finite number of seats.
\end{remark}

\newcommand\AAP{\emph{Adv. Appl. Probab.} }
\newcommand\JAP{\emph{J. Appl. Probab.} }
\newcommand\JAMS{\emph{J. \AMS} }
\newcommand\MAMS{\emph{Memoirs \AMS} }
\newcommand\PAMS{\emph{Proc. \AMS} }
\newcommand\TAMS{\emph{Trans. \AMS} }
\newcommand\AnnMS{\emph{Ann. Math. Statist.} }
\newcommand\AnnPr{\emph{Ann. Probab.} }
\newcommand\CPC{\emph{Combin. Probab. Comput.} }
\newcommand\JMAA{\emph{J. Math. Anal. Appl.} }
\newcommand\RSA{\emph{Random Struct. Alg.} }
\newcommand\ZW{\emph{Z. Wahrsch. Verw. Gebiete} }
\newcommand\DMTCS{\jour{Discr. Math. Theor. Comput. Sci.} }

\newcommand\AMS{Amer. Math. Soc.}
\newcommand\Springer{Springer-Verlag}
\newcommand\Wiley{Wiley}

\newcommand\vol{\textbf}
\newcommand\jour{\emph}
\newcommand\book{\emph}
\newcommand\inbook{\emph}
\def\no#1#2,{\unskip#2, no. #1,} 
\newcommand\toappear{\unskip, to appear}

\newcommand\arxiv[1]{\texttt{arXiv:#1}}
\newcommand\arXiv{\arxiv}

\def\nobibitem#1\par{}


\begin{thebibliography}{999}

\bibitem[Balinski och Young(2001)]{BY}
M.\ L.\ Balinski and H.\ P.\ Young,
\emph{Fair Representation}.
2nd ed.,
Brookings Institution Press, Washington, D.C., 2001.

\bibitem[Billingsley(1968)]{Billingsley}
P.\ Billingsley,
\book{Convergence of Probability Measures}.
Wiley, New York, 1968.

\bibitem{brem}
J.\ Br\'emont, Dynamics of injective quasi-contractions, 
{\em Ergodic Theory $\&$ Dynam.\ Systems} {\bf 26} (2006), 19--44.

\bibitem{bruin}
H.\ Bruin and J.H.B.\ Deane, Piecewise contractions are asymptotically periodic, {\em Proc.\ Amer.\ Math.\ Soc.\  } {\bf 137}(4) (2009), 1389--1395.

\bibitem[Bugeaud(1993)]{bugeaudCR}
Y.\ Bugeaud, 
Dynamique de certaines applications contractantes, lin{\'e}aires par morceaux, sur $[0,1)$. 
\emph{C. R. Acad. Sci. Paris S{\'e}r. I Math.} \textbf{317} (1993), no.\ 6,
575--578. 

\bibitem[Bugeaud and Conze(1999)]{conze}
Y.\ Bugeaud, J.-P.\ Conze, 
Calcul de la dynamique de transformations lin\'eaires contractantes mod 1 et arbre de Farey,  
{\em Acta Arithm.\ } {\bf LXXXVIII}.3 (1999), 201--218.

\bibitem[Bugeaud and Conze(2000)]{BC-noise}
Y.\ Bugeaud and  J.-P.\ Conze, Dynamics of some contracting linear functions modulo 1,
\emph{Noise, Oscillators and Algebraic Randomness (Chapelle des Bois, 1999)}, 
379--387, Lecture Notes in Phys., \textbf{550}, Springer, Berlin, 2000.

\bibitem{cat}
E.\ Catsigeras, P.\ Guiraud, A.\ Meyroneinc and E.\ Ugalde,
On the asymptotic properties of piecewise contracting maps, 
{\em Dyn.\ Sys.\ } {\bf 31} 
(2016), no.\ 2, 107--135.

\bibitem[Coutinho(1999)]{cout}
R.\ Coutinho, {\em Din\^amica Simb\'olica Linear}, Ph.D.\ Thesis, Instituto
Superior T\'ecnico, Universidade de Lisboa, 1999. 

\bibitem{cout2}
R.\ Coutinho, B.\ Fernandez, R.\ Lima, and A.\ Meyroneine, Discrete time piecewise affine models 
of genetic regulatory networks, {\em J.\ Math.\ Biol.\ } {\bf 52} (2006), 524--570.

\bibitem{DHondt1878}
V.\ D'Hondt,
\emph{Question {\'e}lectorale: La repr{\'e}sentation proportionnelle des partis,
par un {\'e}lecteur}.
Bruylant, 
Brussels, 1878. 

\bibitem{DHondt1882}
V.\ D'Hondt,
\emph{Syst{\`e}me pratique et raisonn{\'e} de repr{\'e}sentation proportionnelle},
Muquardt, Brussels, 1882. 

\bibitem[Ding and Hemmer(1987)]{DingHemmer}
E.\ J.\ Ding and 
P.\ C.\ Hemmer,
Exact treatment of mode locking for a piecewise linear map. 
\emph{J. Statist. Phys.} \textbf{46} (1987), no. 1-2, 99--110.

\bibitem{feely}
O.\ Feely and L.\ O.\ Chua, The effect of the integrator leak in 
$\Sigma$--$\Delta$ modulation, 
{\em IEEE Transactions on Circuits and Systems} {\bf 38} (1991), 1293--1305.

\bibitem{bet1913}
S.\ von Friesen, G.\ Appelberg, I.\ Bendixson, E.\ \phragmen,
\emph{Bet{\"a}nkande ang{\aa}ende {\"a}ndringar i g{\"a}llande best{\"a}mmelser
  om den  proportionella valmetoden.}
Commission report, 3 December 1913, Stockholm, 1913.

\bibitem{electoralSystems}
M.\ Gallagher and P.\ Mitchell (Eds.), \book{The Politics of Electoral Systems}.
Oxford Univ. Press, Oxford, 2005.

\bibitem{gamb}
J.-M.\ Gambaudo and C.\ Tresser, On the dynamics of quasi-contractions, 
{\em Bull.\ Braz.\ Math.\ Soc.\ } {\bf 19}(1) (1988), 61--114.

\bibitem[Hardy and Wright(1960)]{HardyW} G.\ H.\ Hardy and E.\ M.\  Wright, 
\emph{An Introduction to the Theory of Numbers}. 4th ed., Oxford, at the
Clarendon Press, 1960. 

\bibitem{SJ262}
S.\ Janson,
Asymptotic bias of some election methods. 
\emph{Annals of Operations Research}, \vol{215} (2014), no. 1, 89--136.

\bibitem{SJV9}
S.\ Janson,
Phragm{\'e}n's and Thiele's election methods,
preprint, 2016: \arxiv{1611.08826}.

\bibitem[Laurent and Nogueira(2017+)]{laurent-nogueira}
M.\ Laurent and A.\ Nogueira, Rotation number of interval contracted rotations, preprint 2017: \arXiv{1704.05130v2}.

\bibitem[Loxton and van der Poorten(1977)]{lox}
J.\ H.\ Loxton and A.\ J.\ van der Poorten, Arithmetic properties of certain functions in 
several variables III, {\em Bull.\ Austral.\ Math.\ Soc.\ } {\bf 16} (1977), 15--47.

\bibitem[Mora and Oliver(2015)]{MoraO}
X.\ Mora and M.\ Oliver,
Eleccions mitjan{\c c}ant el vot d'aprovaci{\'o}.
El m{\`e}tode de Phragm{\'e}n i algunes variants.
\emph{Butllet{\'i} de la Societat Catalana de Matem{\`a}tiques}
\vol{30} (2015), no. 1, 57--101.

\bibitem{nog}
A.\ Nogueira and B.\ Pires, Dynamics of piecewise contractions of the interval,  {\em Ergodic Theory $\&$ Dynam.\ Systems}  {\bf 35}(7) (2015),  2198--2215.

\bibitem{nog2}
A.\ Nogueira, B.\ Pires and R.A.\ Rosales, Topological dynamics of piecewise $\lambda$-affine maps, {\em Ergodic Theory $\&$ Dynam.\ Systems}, to appear.

\bibitem[Phragm{\'e}n(1894)]{Phragmen1894}
E.\ Phragm{\'e}n,
Sur une m{\'e}thode nouvelle pour r{\'e}aliser, dans les {\'e}lections, la
repr{\'e}sentation proportionelle des partis.
\emph{{\"O}fversigt av Kongl. Vetenskaps-Akademiens F{\"o}rhandlingar 1894},
N:o 3, Stockholm, 133--137.

\bibitem[Phragm{\'e}n(1895)]{Phragmen1895}
E.\ Phragm{\'e}n,
\emph{Proportionella val. En valteknisk studie.}
Svenska sp\"orsm{\aa}l 25, Lars H\"okersbergs f{\"o}rlag, Stockholm, 1895.

\bibitem[Phragm{\'e}n(1896)]{Phragmen1896}
E.\ Phragm{\'e}n,
Sur la th\'eorie des {\'e}lections multiples,
\emph{{\"O}fversigt av Kongl. Vetenskaps-Akademiens F{\"o}rhandlingar 1896},
N:o 3, Stockholm, 181--191.

\bibitem[Phragm{\'e}n(1899)]{Phragmen1899}
E.\ Phragm{\'e}n,
Till fr{\aa}gan om en proportionell valmetod.
\emph{Statsvetenskaplig Tidskrift} \vol2 (1899),
nr 2,
297--305.
\url{http://cts.lub.lu.se/ojs/index.php/st/article/view/1949}

\bibitem[Pukelsheim(2014)]{Pukelsheim}
F.\ Pukelsheim, \emph{Proportional Representation. Apportionment Methods and Their
  Applications}, Springer, Cham, Switzerland, 2014.

\bibitem[Sainte-Lagu\"e(1910)]{StL}
A.\ Sainte-Lagu\"e,  
La repr{\'e}sentation proportionnelle et la m{\'e}thode des moindres carr{\'e}s.  
\emph{Ann. Sci. {\'E}cole Norm. Sup.} (3) \vol{27} (1910), 529--542.
\\
Summary:
\emph{Comptes rendus
hebdomadaires des s{\'e}ances
de l'Acad{\'e}mie des sciences},
\vol{151}
(1910), 377--378.

\bibitem[Tenow(1912)]{Tenow1912}
N.\ B.\ Tenow,
Felaktigheter i de Thieleska valmetoderna.
\emph{Statsvetenskaplig Tidskrift} 1912, 145--165.

\bibitem{Thiele}
T.\ N.\ Thiele, 
Om Flerfoldsvalg.
\emph{Oversigt over det Kongelige Danske Videnskabernes Selskabs Forhandlinger} 1895,
K{\o}benhavn, 1895--1896, 415--441.

\bibitem{veer}
P.\ Veerman, Symbolic dynamics of order-preserving orbits, {\em Physica D}
{\bf 29} (1987), 191--201.

\end{thebibliography}
\end{document}